\documentclass{amsart}

\usepackage{amssymb}
\usepackage{amsthm}
\usepackage{epsfig}
\usepackage{color}
\usepackage{enumitem}

\usepackage{latexsym,subfigure}
\usepackage{comment}
\usepackage{amsmath,amsthm,amsfonts,amssymb,graphicx,epsfig,latexsym,color}

\usepackage[pdftex,colorlinks]{hyperref}

\begin{document}

\newtheorem{theorem}{Theorem}[section]
\newtheorem{lemma}[theorem]{Lemma}
\newtheorem{corollary}[theorem]{Corollary}
\newtheorem{proposition}[theorem]{Proposition}
\newtheorem{question}[theorem]{Question}
\newtheorem{conjecture}[theorem]{Conjecture}
\newtheorem{case}{Case}
\newtheorem{subcase}{Case}[case]
\newtheorem{subsubcase}{Case}[subcase]
\newtheorem{claim}[theorem]{Claim}
\newtheorem{step}{Step}

\theoremstyle{definition}
\newtheorem{remark}[theorem]{Remark}
\newtheorem{definition}[theorem]{Definition}
\newtheorem{example}[theorem]{Example}

\newcommand\bY{\mathbf{Y}}

\newcommand\E{\mathbb{E}}
\newcommand\Z{\mathbb{Z}}
\newcommand\R{\mathbb{R}}
\newcommand\T{\mathbb{T}}
\newcommand\C{\mathbb{C}}
\newcommand\CC{\mathcal{C}}
\newcommand\N{\mathbb{N}}
\newcommand\G{\mathbf{G}}
\newcommand\A{\mathbb{A}}
\newcommand\HH{\mathbb{H}}
\newcommand\SL{\operatorname{SL}}
\newcommand\Upp{\operatorname{Upp}}
\newcommand\Cay{\operatorname{Cay}}
\newcommand\im{\operatorname{im}}
\newcommand\bes{\operatorname{bes}}
\newcommand\sss{\operatorname{ss}}
\newcommand\GL{\operatorname{GL}}
\newcommand\PGL{\operatorname{PGL}}
\newcommand\SO{\operatorname{SO}}
\newcommand\SU{\operatorname{SU}}
\newcommand\PSL{\operatorname{PSL}}
\newcommand\Rad{\operatorname{Rad}}
\newcommand\Mat{\operatorname{Mat}}
\newcommand\ess{\operatorname{ess}}
\newcommand\rk{\operatorname{rk}}
\newcommand\Alg{\operatorname{Alg}}
\newcommand\Supp{\operatorname{Supp}}
\newcommand\sml{\operatorname{sml}}
\newcommand\lrg{\operatorname{lrg}}
\newcommand\tr{\operatorname{tr}}
\newcommand\Hom{\operatorname{Hom}}
\newcommand\Lie{\operatorname{Lie}}
\newcommand\Inn{\operatorname{Inn}}
\newcommand\Out{\operatorname{Out}}
\newcommand\Aut{\operatorname{Aut}}
\renewcommand\P{\mathbb{P}}
\newcommand\F{\mathbb{F}}
\newcommand\Q{\mathbb{Q}}
\renewcommand\b{{\bf b}}
\def\g{\mathfrak{g}}
\def\h{\mathfrak{h}}
\def\n{\mathfrak{n}}
\def\a{\mathfrak{a}}
\def\p{\mathfrak{p}}
\def\q{\mathfrak{q}}
\def\b{\mathfrak{b}}
\def\r{\mathfrak{r}}

\def\bn{\mathbf{n}}

\newcommand{\kf}[1]{\marginpar{\tiny #1 }}
\newcommand{\red}{\textcolor{red}}
\newcommand{\blue}{\textcolor{blue}}

\title[Joint spectral radius and uniform growth]{On the joint spectral radius for isometries of non-positively curved spaces and uniform growth}
\date{\today}

\author{Emmanuel Breuillard}
\address{University of Cambridge, DPMMS}
\email{breuillard@maths.cam.ac.uk}

\author{Koji Fujiwara}
\address{Department of Mathematics, Kyoto University}
\email{kfujiwara@math.kyoto-u.ac.jp}
\thanks{ }

\begin{abstract} We recast the notion of joint spectral radius in the setting of groups acting by isometries on non-positively curved spaces and give geometric versions of results of Berger-Wang and Bochi valid for  $\delta$-hyperbolic spaces and for symmetric spaces of non-compact type. This method produces nice hyperbolic elements in many classical geometric settings.  Applications to uniform growth are given, in particular a new proof and a generalization of a theorem of Besson-Courtois-Gallot.
\end{abstract}

\maketitle

\section{Introduction}

This paper is concerned with the following general problem. Given a group $G$ generated by a finite set $S$. Suppose that $G$ contains elements with a certain property $\mathcal{P}$. Can  we estimate the shortest length of an element $g \in G$ which has the property $\mathcal{P}$ ?  Here the length of an element $g \in G$ is the smallest length of a word with letters in $S$ and $S^{-1}$, which represents $g$.

 For example the property $\mathcal{P}$ could be ``being of infinite order'', or  ``having an eigenvalue of modulus different from $1$'', say when $G$ is a matrix group. Or else when $G$ is a subgroup of isometries of a non-positively curved space $\mathcal{P}$ could be : ``having positive translation length''. This type of question is ubiquitous in problems dealing with uniform exponential growth. In this paper we give a method to find nice short words in a variety of situations when the group acts on a space with non-positive curvature. As an application we will give several uniform exponential growth results, in particular for groups acting on products of $\delta$-hyperbolic spaces. In \cite{breuillard-fujiwara} we will consider applications to the growth of mapping class groups. 

This method consists in studying the growth of the \emph{joint minimal displacement} of the generating set $S$. To fix ideas say $(X,d)$ is a metric space and $S \subset Isom(X)$ is a finite set of isometries of $X$. Let $G:=\langle S \rangle$ be the group generated by $S$ and $x\in X$ is a point. We define the \emph{joint displacement at $x$} by
$$L(S,x):= \max_{s \in S} d(x,sx)$$
and the \emph{joint minimal displacement} of $S$ by
$$L(S) := \inf_{x \in X} L(S,x).$$ 

This quantity appears in many places in geometric group theory, for example in the construction of harmonic maps as in the work of Gromov-Schoen \cite{gromov-schoen} and Korevaar-Schoen \cite{korevaar-schoen}, or in Kleiner's proof of Gromov's polynomial growth theorem \cite{kleiner}, and, when $X$ is a Hilbert space, in relation to reduced first cohomology and Hilbert compression as in \cite{Cornulier-Tessera-Valette} ; it is ubiquitous in the study of free group automorphisms and it is key to constructing limits of representations, e.g. see \cite{bestvina-survey}, and or in recent work about surface group representations in higher rank Lie groups, e.g. \cite{burger-pozzetti}.

In fact a more common quantity is the $\ell^2$ version of $L(S)$, sometimes called the \emph{energy} of $S$ and defined as the infimum of the averaged squared displacement $\frac{1}{|S|} \sum_{s \in S} d(sx,x)^2$. In this paper however we will exclusively consider $L(S)$, which is the $\ell^{\infty}$ version of the energy. One advantage of working with $L(S)$ is that it behaves well under set theoretic powers $S^n$ of $S$ (while the energy is more suitable to random walks as expounded for example in Gromov's paper \cite{gromov-random}).

It is natural to also consider the \emph{asymptotic joint displacement}
$$\ell(S) := \lim_{n \to +\infty} L(S^n)/n.$$
It is not hard to show (see Section \ref{gen} below) that this limit exists and coincides with the limit of $L(S^n,x)/n$ for any point $x$. When $S=\{g\}$ is a singleton, $L(g)$ is called the \emph{translation length} of $g$ and we call $\ell(g)$ the \emph{asymptotic translation length}. We may also consider the maximal translation length
$$\lambda(S) : = \max_{s \in S} \ell(s)$$
and the corresponding notion for products
$$\lambda_k(S) : = \max_{ 1 \leq j \leq k} \frac{1}{j} \lambda(S^j),$$
$$\lambda_\infty(S) : = \sup_{  j \in \N} \frac{1}{j} \lambda(S^j)$$
and compare these quantites to those defined above. The choice of normalization in the definition of $\lambda_k(S)$ is made so as to guarantee the following straightforward inequalities (see Section \ref{gen}).
\begin{lemma}[general nonsense lemma]\label{gen-ineq} If $(X,d)$ is any metric space and $S \subset Isom(X)$ a finite set of isometries, then  for all $k \in \N$
$$\lambda(S) \leq \lambda_k(S) \leq \lambda_\infty(S) \leq \ell(S) \leq \frac{1}{k}L(S^k) \leq L(S).$$ Moreover $\ell(S^k)=k\ell(S)$ and $\lambda_\infty(S^k)=k\lambda_\infty(S)$. Finally $$\lambda_\infty(S)=\limsup_{n \to +\infty} \frac{1}{n}\lambda(S^n).$$  
\end{lemma}
We will be interested in the following questions: To what extent are these inequalities sharp ? What is the growth of $L(S^n)$ as $n$ grows ? For which spaces do we always have $\lambda_\infty(S) = \ell(S)$ ? 

These quantities are interesting when $X$ is an unbounded metric space, especially in presence of some form of non-positive curvature, for example a $CAT(0)$-space or a $\delta$-hyperbolic space.

When $X=\SL_d(\R)/\SO_d(\R)$  and the metric $d$ is given by \begin{equation} \label{finsler} d(g,h) := \log ||g^{-1}h||,\end{equation} where the norm is the operator norm associated to a Euclidean scalar product on $\R^d$,  then any $S \subset \SL_d(\R)$ acts by isometries on $X$ and the quantity $\ell(S)$ is the ($\log$ of the) well-known \emph{joint spectral radius} of the set of matrices $S$. This quantity has been first studied by Rota and Strang in 1960 \cite{rota-strang} and further by Daubechies-Lagarias \cite{daubechies-lagarias} in the context of wavelets and iterated function systems and by Berger-Wang \cite{berger-wang}. The main result of Berger-Wang \cite[Thm. IV]{berger-wang} says that 
\begin{equation}\label{bw}\lambda_\infty(S)=\ell(S).\end{equation}
This means that the rate of growth of $\|S^n\|$ is identical to that of the maximal eigenvalue of an element in $S^n$, where
$$ \|S^n\| := \max_{g \in S^n} \|g\|,$$
 and a posteriori justifies the terminology ``joint spectral radius'' for $\exp(\ell(S))$ in that it generalizes the classical Gelfand formula expressing the modulus of the maximal eigenvalue  as the rate of exponential growth of the norm of powers of a single matrix. 

In the literature on the joint spectral radius authors usually consider arbitrary finite sets of matrices $S \subset M_d(\R)$ which are not necessarily invertible. In this paper however, because of our geometric point of view, we will focus on the invertible case.

\subsection{Geometric Berger-Wang identity and geometric Bochi inequality} One of the leitmotivs of this paper will be extend the Berger-Wang identity $(\ref{bw})$ and its refinements to geometric actions on non-positively curved spaces $X$. We will mainly focus on fairly classical spaces such as  Euclidean spaces, Hilbert spaces, or non-positively curved spaces such as symmetric spaces of non-compact type. We will also obtain satisfactory results for arbitrary $\delta$-hyperbolic spaces and will mention some properties valid in arbitrary $CAT(0)$-spaces. We hope that our work may provide motivation to study the joint minimal displacement and Berger-Wang type identities in other geometries of interest, such as the Teichmuller space or the Outer space. We begin with:

\begin{theorem}[geometric Berger-Wang identity]\label{bw-thm} The Berger-Wang identity $(\ref{bw})$ holds for every finite set $S$ of isometries of $X$, when $X$ is either a symmetric space of non-compact type, a tree, or an arbitrary $\delta$-hyperbolic space.
\end{theorem}

By contrast, we will show in Section \ref{euclidean} that the Berger-Wang identity fails for isometries of a Euclidean space. 

In the case of symmetric spaces, this result is closely related to the original Berger-Wang identity for matrices. However in higher rank it does not follow directly (nor does it imply it) as we stress in the following remark. In the case of $\delta$-hyperbolic spaces Theorem \ref{bw-thm} was recently obtained by Oregon-Reyes \cite{oregon-reyes}, whose paper addresses for the first time the question of extending the classical results on the joint spectral radius to other geometric contexts. 

\begin{remark}Note that, as usual, symmetric spaces  are considered with their defining $CAT(0)$ metric induced by a left-invariant Riemannian metric associated to the Killing form of the group. The quantities studied in this paper such as $L(S)$ and $\lambda(S)$ are sensitive to the choice of metric. For example the classical Berger-Wang identity $(\ref{bw})$ for matrices does not follow from the geometric Berger-Wang identity proved in the above theorem, nor does it imply it, say for $X=\SL_d(\R)/\SO_d(\R)$. Recall that the distance $(\ref{finsler})$ on $X=\SL_d(\R)/\SO_d(\R)$ is not the usual $CAT(0)$ Riemannian symmetric space metric, but rather a Finsler-type norm-like distance (as in \cite{abels-margulis} and \cite{parreau-preprint}).\end{remark}

In \cite{bochi} Bochi gave a different proof of $(\ref{bw})$, which yields a stronger inequality of the form \begin{equation}\label{bochi-ineq}\lambda_{k_0}(S) \geq L(S) -C,\end{equation} where $k_0$ and $C$ are constants depending only on the dimension. The Berger-Wang identity follows immediately by applying the Bochi inequality to $S^n$ and letting $n$ go to infinity. Note that a different proof of Bochi's inequality was given in \cite[Cor 4.6]{breuillard-gelander} and extended in \cite{breuillard-heightgap} to non-archimedean local fields. In this paper we  will prove geometric analogues of the Bochi inequality $(\ref{bochi-ineq})$  in the geometric settings mentioned in the above theorem. For example:

\begin{theorem}[geometric Bochi-type inequality for hyperbolic spaces]\label{bochi-hyp} There is an absolute constant $K>0$ such that the following holds. If  $(X,d)$ is a geodesic $\delta$-hyperbolic space, then
 $$\lambda_{2}(S)  \geq L(S) - K \delta $$ for every finite set $S \subset Isom(X)$.
\end{theorem}

Consequently $\ell(S) \geq L(S) - K \delta$, which answers a question raised in \cite[Question 6.1]{oregon-reyes}.

For isometries of trees (case when $\delta=0$) we have $\lambda_2(S)=\ell(S)=L(S)$. In follows in particular that $\lambda_2(S)=0$ implies $L(S)=0$, which is the content of a well-known lemma of Serre about tree isometries (\cite[I. Prop. 26]{serre}). The proof for $\delta$-hyperbolic spaces is a quasification of the corresponding proof for isometries of trees. It involves in particular a Helly-type theorem for hyperbolic spaces, which we prove in Section \ref{helly-sec}. See \cite{farb} and \cite{caprace-lytchak} for a related use of Helly-type theorems in geometric group theory.

As already mentioned Theorem \ref{bw-thm} for $\delta$-hyperbolic spaces follows from Theorem \ref{bochi-hyp}. We now record one more consequence of Theorem
\ref{bochi-hyp}  (see Proposition \ref{ell=0} for the proof), which is well-known and due to Gromov \cite[8.1]{gromov.hyp} (see also \cite[Thm 1.10]{oregon-reyes}).
\begin{corollary}\label{intro.ell=0}
Let $X$ be a geodesic $\delta$-hyperbolic space and $S \subset Isom(X)$ a finite set.
Assume that $\ell(S)=0$.
Then  $\langle S \rangle$ either has a bounded orbit on $X$, or 
fixes a unique point in $\partial X$.
\end{corollary}



A similar but slightly weaker statement when $X$ is a symmetric space is proven in Section \ref{sym-spaces}, based on the Bochi inequality applied to various linear embeddings of $X$ (see Theorem \ref{bochi-sym} and Proposition \ref{compa-compa}). In particular:

\begin{proposition}\label{boch-symsym} Let $X$ be a symmetric  space of non-compact type and $S \subset Isom(X)$ be a finite subset. Then : 
$$\lambda_{k_0}(S) \geq \frac{1}{\sqrt{d}} L(S) - C,$$ where $d \in \N$ is such that $X$ is a convex subspace of $\SL_d(\C)/\SU_d(\C)$,  $k_0\leq d^2$, and $C>0$ is a constant depending on $d$ only.
\end{proposition}

We suspect but were unable to prove that the multiplicative constant $\frac{1}{\sqrt{d}}$ in the above result can be taken to be $1$. See Theorem \ref{bochi-sym} and Question 4 in Section \ref{questions-sec}. On the other hand we will prove in Section \ref{escape} that the constant $C$ cannot be taken to be zero. In fact for each $n$ one may find an $S$ with $\ell(S)>0$ but $\lambda_n(S)=0$ (see Lemma \ref{so41}).

\subsection{Growth of joint minimal displacement} A consequence of the geometric Berger-Wang identity is that in order to guarantee the existence of a single element $g$ with positive translation length in the group generated by $S$, it is enough to show that $\ell(S)>0$. Accordingly this suggests that in order to find a nice hyperbolic element in a power of $S$, it is important to study the growth rate of $L(S^n)$.  In a CAT(0) space a simple geometric argument, similar to one used by V. Lafforgue in \cite[Lemma 2]{lafforgue} for Hilbert spaces, shows the following:

\begin{lemma}\label{cat0-growth} If $X$ is a CAT(0) metric space and $S$ a finite set of isometries, then 
$$L(S^n) \geq \frac{\sqrt{n}}{2} L(S).$$
\end{lemma}

The square root $n$ growth behavior is sharp as examples of Cornulier, Tessera and Valette \cite{Cornulier-Tessera-Valette} show (see Remark \ref{rkCTV}).  Of course if $\ell(S)>0$ then $L(S^n)$ grows linearly, but the above sublinear growth is useful in order for $L(S^n)$ to go above a certain threshold after which the linear growth can start. For example if $X$ is $\delta$-hyperbolic, we will show that $L(S^n) \geq n (L(S) - K\delta)$ for some absolute constant $K$, so that we get linear growth provided $L(S)>K\delta$. While if $X$ is a symmetric space of non-compact type a consequence of our analysis is:

\begin{proposition}\label{ell-L-sym} There is $c=c(X)>0$ such that if $S$ is a finite set of isometries of a symmetric space $X$ of non-compact type
$$c \cdot \min\{L(S),L(S)^2\} \leq \ell(S) \leq L(S).$$
\end{proposition}

Observe that a consequence of the left hand-side and of Theorem \ref{bw-thm} is the well-known fact that if a finitely generated group of isometries of $X$ is entirely made of elliptic elements then it fixes a point on $X$ or on its visual boundary (see \ref{bdy-fixed}).

When $X$ is a symmetric space of non-compact type, the Margulis lemma asserts the existence of a constant $\epsilon=\epsilon(X)>0$ depending only on $\dim X$ such that the following holds :  every finite set $S$ of isometries of $X$, which generates a discrete subgroup of $Isom(X)$, and has $L(S)<\epsilon$ must generate a virtually nilpotent subgroup. Combining this with the previous proposition, one obtains:

\begin{corollary}  There are constants $c_0,c_1>0$ depending only on $X$,  such that if $S$ is a finite set of isometries of a symmetric space $X$ of non-compact type generating a non virtually nilpotent discrete subgroup of $Isom(X)$, then
$$ c_1 \leq \ell(S) \leq L(S) \leq  c_0 \cdot \ell(S).$$
\end{corollary}

In \cite[Cor 1.3]{burger-pozzetti} Burger and Pozzetti show that in the case of maximal representations of a surface group in the Siegel upper-half space, the above inequalities hold already with $\lambda(S)$  in place of $\ell(S)$ and with an explicit value of $c_1$.

\subsection{Euclidean spaces}
When $X$ is a finite dimensional Euclidean space, it is still true that $L(S)=0$ if and only if $\ell(S)=0$, namely:

\begin{proposition} If $S$ is a finite set of isometries of Euclidean $\R^d$, then the following are equivalent:
\begin{enumerate}
\item $S$ has a common fixed point,
\item $L(S)=0$,
\item $\ell(S)=0$.
\end{enumerate}
\end{proposition}

While it is clear that $L(S)=0$ implies $\ell(S)=0$, the converse is slightly more subtle, because it can happen that $S$ has no global fixed point, and yet every single element in the group generated by $S$ has a fixed point (see Example \ref{exBass}). In particular one can have  $\ell(S)>0$ while $\lambda_\infty(S)=0$. This means that the Berger-Wang identity $(\ref{bw})$ fails for isometries of Euclidean spaces. See Section \ref{euclid-sec}.

By contrast, when $X$ is an infinite dimensional Hilbert space one can have $\ell(S)=0$ while $L(S)>0$. A sublinear cocycle need not be trivial in the reduced first cohomology \cite[3.9]{Cornulier-Tessera-Valette}, see Section \ref{hilbert-sec}.

\subsection{Large torsion balls and escape from elliptics}

Schur proved in 1907 \cite{schur} that a subgroup of $\GL_n(\C)$ all of whose elements are of finite order must be finite. In \cite{breuillard-heightgap} the first author proved a uniform version of Schur's theorem:

\begin{theorem}[uniform Schur theorem \cite{breuillard-heightgap}] There is $N=N(d) \in \N$ such that for  every finite symmetric subset $S \subset \GL_d(\C)$ containing $1$ and generating an infinite group, there is an element $g \in S^N$ of infinite order.
\end{theorem}

Examples of Grigorchuk and de la Harpe \cite{grigorchuk-harpe} and of Bartholdi and Cornulier \cite{bartholdi-cornulier} show that $N(d)$ must grow to infinity with $d$. In Section \ref{escape} we will show that, by contrast with the case of torsion elements, one cannot escape elliptic elements in general:

\begin{lemma}[no uniform escape from elliptics]\label{so41} In the real Lie group $\SO(4,1)$, for each $N \in \N$ one can find a pair $S=\{a,b\}$ such that for all words $w$ of length at most $N$ in $a^{\pm 1},b^{\pm 1}$ the cyclic subgroup $\langle w \rangle$ is bounded, but the subgroup $\langle S \rangle$ is unbounded and even Zariski dense (in particular it does not fix a point on hyperbolic $4$-space $\mathbb{H}^4$ nor on the boundary $\partial \mathbb{H}^4$).
\end{lemma}

The proof is based on the Tits alternative and the Borel-Larsen theorem \cite{borel} on the dominance of word maps on semisimple algebraic groups.

On the other hand  the geometric Bochi-type inequalities proved in this paper give a tool to produce quickly (i.e. in $S^n$ for $n$ small) a non-elliptic element. The time needed to produce this element however depends on $L(S)$. For example Theorem \ref{formula.S.hyp} in Section \ref{serre-hyp} shows that, when $X$ is a $\delta$-hyperbolic space, there is $g \in S \cup S^2$ with $\ell(g)>0$ (and in particular of infinite order) provided $L(S) > K\delta$. However $L(S^n)$ may well remain under $K\delta$ for large $n$, even though $\ell(S)>0$. For example quotients of infinite Burnside groups provide examples of $\delta$-hyperbolic groups with large torsion balls in their Cayley graph (\cite{olshanski, bartholdi-cornulier}).

Lemma \ref{so41} also shows that even though the geometric Berger-Wang identity $\lambda_\infty(S)=\ell(S)$ holds in symmetric spaces by Theorem \ref{bw-thm}, the additive constant $C$ in the Bochi-type inequality $\lambda_k(S) \geq \ell(S)/\sqrt{d} - C$, which is a consequence of Proposition \ref{boch-symsym}, cannot be zero.

\subsection{Uniform exponential growth (UEG)}

If $S$ is a finite generating subset in a group $\Gamma$, we define
$$h(S)=\lim_{n\to +\infty} \frac{1}{n}\log |S^n|, $$
and call $h(S)$ the {\it entropy} of $\Gamma$ for $S$.

The original motivation for the present paper was a theorem of Besson-Courtois-Gallot \cite{BCG-jems}. They showed that given $a>0$ and $n \in \N$, there is a constant $c(n,a)>0$ such that if $M$ is a complete Riemannian manifold of dimension $n$ with pinched sectional curvature $\kappa_M \in [-a^2,-1]$ and $\Gamma$ is a discrete group of isometries of $M$ generated by a finite set $S$, then 
$$h(S)=\lim_{n\to +\infty} \frac{1}{n}\log |S^n| > c(n,a)>0,$$
provided $\Gamma$ is not virtually nilpotent.

We will generalize this result and give a new proof of it, seeing it as a fairly direct consequence of the Bochi-type inequality of Theorem \ref{bochi-hyp}. Our approach provides the additional information that generators of a free semigroup can be found in a bounded ball -- a feature the Besson-Courtois-Gallot proof did not yield. The proof will be given in Section \ref{application-ueg}, but let us briefly explain here how this works. Due to its negative curvature, the manifold $M$ is a $\delta$-hyperbolic metric space for some $\delta>0$ which is independent of $n$ and $a$. Once a hyperbolic element in $S$ or a bounded power $S^n$ has been found, a simple ping-pong argument gives generators of a free semi-group (cf. Section \ref{ping-pong} below). Thanks to the geometric Bochi inequality for hyperbolic spaces, i.e. Theorem \ref{bochi-hyp}, in order to find a hyperbolic element, we only need to check that $L(S)>K \delta$ or at least that $L(S^n)>K\delta$ for some controlled $n$ (cf. Theorem \ref{hyp.tri}). 

The manifold is also $CAT(0)$, so Lemma \ref{cat0-growth} applies, and we see that we only need to rule out the possibility that $L(S)$ is very small.
 But if we assume that the group $\Gamma$ is discrete, the Margulis lemma tells us that there is a constant $\epsilon=\epsilon(n,a)>0$ such that if $L(S)<\epsilon$ and $\langle S \rangle$ is discrete, then $\langle S \rangle$  is virtually nilpotent. This ends the proof.

Using instead the generalized Margulis lemma proved by Green, Tao and the first author in \cite{BGT} this argument yields the following generalization of the Besson-Courtois-Gallot theorem. 

\begin{theorem}[UEG for hyperbolic spaces with bounded geometry]\label{bgt-bcg} Given $P \in \N$, there is a constant $N(P)\in \N$ such that the following holds. Assume that $X$ is a geodesic $\delta$-hyperbolic space with the property that every ball of radius $2\delta$ can be covered by at most $P$ balls of radius $\delta$.  Let $S$ be a finite symmetric set of isometries generating a group $\Gamma$ and containing $1$. Then either $\Gamma$ is virtually nilpotent, or $S^N$ contains two generators of a free semigroup and in particular:
$$h(S) > \frac{1}{N}\log 2.$$
\end{theorem}

We stress that the constant $N$ depends only on $P$ and not on $\delta$.

\subsection{Hyperbolic groups}

It is known that a non-elementary hyperbolic group
has uniform exponential growth and even uniform  uniform exponential growth, that is subgroups have uniform exponential growth. See \cite{delzant,koubi, guirardel-champetier}. Uniform exponential growth
follows immediately from our general result on actions on hyperbolic spaces (see Theorem \ref{hyp.tri}). The point is that when $X$ is the Cayley graph of a hyperbolic group, then it is straightforward that $L(S^n) \geq n$
if $S$ generates the group so the Margulis lemma used above in Theorem \ref{bgt-bcg} is irrelevant.

\begin{theorem}[Growth of  hyperbolic groups]\label{ueg-hyp-intro} There is an absolute constant $C_1>0$ such that if $G$ is a group with finite generating set $S$ whose Cayley graph $Cay(G,S)$ is $\delta$-hyperbolic, then either $G$ is finite or virtually cyclic, 
or $S^{M}$ contains two hyperbolic elements that are generators of a free semi-group, where $M$ is the least integer larger than $C_1\delta$.  In particular, $h(S) \ge (\log 2)/M$.
\end{theorem}

We stress that $M$ depends only on $\delta$. This special feature was not explicit
for example in \cite{koubi}. Stated as such the theorem is sharp inasmuch as  $M$ must depend on $\delta$ (it tends to infinity as $\delta$ goes to infinity) : indeed Olshanski \cite{olshanski} gave for sufficiently large primes $p$ and radii $R$ examples of non-elementary $2$-generated Gromov hyperbolic groups (with large $\delta=\delta(R,p)$) whose ball of radius $R$ is made of $p$-torsion elements (and hence there is no element of infinite order in a ball of small radius). 

The same conclusion about $h(S)$ has recently been obtained independently by Besson, Courtois, Gallot and Sambusetti in their recent preprint \cite{BCGS-preprint} without exhibiting a free semi-group, but with the following explicit lower bound $h(S) \ge \log 2/(26\delta+16)$. Note that by contrast sharpness of such a lower bound is not known as it is still an open problem whether or not there is an absolute constant $c>0$ such that $h(S)>c$ for every generating set $S$ of an arbitrary hyperbolic group (independently of $\delta$). 

Uniform uniform exponential growth also 
immediately follows from Theorem \ref{hyp.tri}.
In this case,  $S$ may not generate the whole
group, and the constant $M$ must depend not only on $\delta$, but also on the size of $S$. See Corollary \ref{hyp.uueg} and Remarks \ref{other-ref} and \ref{olshanski-exples}.

\subsection{Uniform Tits alternative for groups acting on trees}
In \cite{breuillard-gelander} Gelander and the first author proved a uniform Tits alternative for linear groups. Namely given a non virtually solvable finitely generated group $\Gamma$ contained in $\GL_n(k)$ for some field $k$, there is an integer $N=N(\Gamma)$ such that for every symmetric generating set $S$ of $\Gamma$ the ball $S^N$ contains a pair of free generators of a free subgroup. The number $N$ was later shown in \cite{breuillard-heightgap, breuillard-tits} to be a constant depending only on the dimension $n$ and in particular independent of the group $\Gamma$. A natural question arises as to whether or not a similar phenomenon occurs for subgroups of $Isom(X)$, where $X$ is a tree.  We will give the following counter-example:

\begin{proposition}\label{counter-tree} Given $N \in \N$ there are two isometries $a,b$ of a trivalent tree without common fixed point on the tree nor on its boundary, such that for every two words $w_1,w_2$ of length at most $N$ in $a, b$ and their inverses, the subgroup $\langle w_1,w_2\rangle$ is not free non-abelian.
\end{proposition}

In other words the uniform uniform Tits alternative does not hold for non-elementary subgroups of isometries of a tree. Note that the assumptions imply that $\langle a,b\rangle$ contains a non-abelian free subgroup. This is to be contrasted with the fact (see Proposition \ref{semi-tree}) that under the same assumptions there always is a pair of words of length at most $3$, which generates a free semi-group. On the other hand it is not clear whether or not the uniform Tits alternative holds for a fixed subgroup $\Gamma$, (see Question 2. in Section \ref{questions-sec}).

\bigskip

\noindent \emph{Acknowledgment.} We are grateful to Martin Bays, Mohammad Bardestani, Marc Burger, Yves de Cornulier,  Gilles Courtois, Thomas Delzant, Nicolas Monod, Andrea Sambusetti and Cagri Sert for useful conversations and references.  The first author acknowledges support from ERC grant no. 617129 GeTeMo.
The second author is supported by 
 Grant-in-Aid for Scientific Research (15H05739). The project was started in Hawa\"i in 2008 and completed in Cambridge in 2017. The authors would also  like to thank the Isaac Newton Institute for Mathematical Sciences, Cambridge, for support (through  EPSRC grant no EP/K032208/1) and hospitality during the programme on Non-Positive Curvature, where work on this paper was undertaken.

\setcounter{tocdepth}{1}

\tableofcontents

\section{The joint minimal displacement}\label{gen}

In this section we recall the geometric quantities introduced in the introduction and prove the general inequalities they satisfy, i.e. Lemma \ref{gen-ineq}. In this section $(X,d)$ is an arbitrary metric space. 

 Let $S$ be a finite set of isometries and $x$ a point in $X$. Recall that we have defined $L(S,x)$, $L(S)$, $\ell(S)$ and $\lambda_k(S)$ for $k\in \N_{\geq 1} \cup {\infty}$ at the beginning of the introduction.

\subsection{Proof of Lemma \ref{gen-ineq}}

First we make the following simple observation:

\bigskip
\noindent \emph{Claim 1:} If $U,V$ are finite subsets of isometries of $X$ and $x \in X$, then:
\begin{equation} \label{subadd} L(UV,x)  \leq L(V,x) + L(U,x)\end{equation}

Indeed by the triangle inequality:
\begin{eqnarray*}L(UV,x)  = \max_{u \in U, v \in V} L(uv,x) &\leq&  \max_{u \in U, v \in V} L(uv,ux) + L(u,x)  = \max_{u \in U, v \in V} L(v,x) + L(u,x) \\ &\leq& 
L(V,x) + L(U,x).\end{eqnarray*}

From this we get that $n \mapsto L(S^n,x)$ is subadditive, and therefore by the subadditive lemma:

\bigskip
\noindent \emph{Claim 2:} the following limit exists and is independent of $x$ 
\begin{equation}\label{elldef}\ell(S):=\lim_n \frac{1}{n} L(S^n,x) = \inf_{n \geq 1}  \frac{1}{n} L(S^n,x).
\end{equation}
We take the above as a definition for $\ell(S)$. To see that this limit does not depend on the point $x$ simply note that for every finite set $S$ and every pair of points $x,y$
$$ L(S,x) = \max_{s \in S} d(sx,x) \leq  \max_{s \in S} \{d(sx,sy + d(sy,y) + d(y,x)\} \leq 2d(x,y) + L(S,y)$$
and hence exchanging the roles of $x$  and $y$
\begin{equation}|L(S,x)-L(S,y)| \leq 2d(x,y).\end{equation}

\bigskip
\noindent \emph{Claim 3:} For every $n \in \N$ we have
\begin{equation} \frac{1}{n} L(S^n) \leq L(S)
\end{equation}

Indeed applying $(\ref{subadd})$ iteratively we have $L(S^n,x) \leq n L(S,x)$ for all $x$.

\bigskip

\noindent \emph{Claim 4:} The sequence $L(S^n)/n$ converges to $\ell(S)$ and 
\begin{equation} \ell(S) = \lim \frac{1}{n} L(S^n) = \inf_{n \geq 1} \frac{1}{n} L(S^n)
\end{equation}

Indeed, since $L(S^n,x) \geq L(S^n)$ we get immediately from $(\ref{elldef})$ that $\ell(S) \geq \limsup \frac{1}{n}L(S^n).$ On the other hand $\ell(S) \leq \frac{1}{n}L(S^n,x)$ for all $x$ and all $n$. Minimizing in $x$ we get $\ell(S) \leq \frac{1}{n}L(S^n)$ and hence $\ell(S) \leq \inf_{n \geq 1} \frac{1}{n} L(S^n) \leq \liminf  \frac{1}{n} L(S^n)$.

We conclude immediately that:

\bigskip
\noindent \emph{Claim 5:} $\ell(S^n)= n\ell(S)$ for every $n\in \N$.

\bigskip

We now turn to $\lambda(S)$ and $\lambda_k(S)$, which we have defined in the introduction as:

$$\lambda(S) := \max_{s \in S} \ell(s)$$
and 
$$\lambda_k(S) := \max_{1 \leq j \leq k} \frac{1}{j} \lambda(S^j).$$

We can now complete the proof of Lemma \ref{gen-ineq}.

\bigskip
\noindent \emph{Claim 6:} $\lambda(S) \leq \ell(S).$
\bigskip

Indeed $(\ref{elldef})$ implies that $\ell(s) \leq \ell(S)$ for every $s \in S$. It follows that $\lambda(S) \leq \ell(S)$ and thus that $\lambda_k(S) \leq  \ell(S)$ for every $k \in \N$. Finally we have:

\bigskip
\noindent \emph{Claim 7:} For every $k,n \in \N$ we have \begin{equation}\label{claim7}\lambda_n(S) \leq \frac{1}{k} \lambda_{n}(S^k) \leq \lambda_{kn}(S).\end{equation}
\bigskip

To see this note that  $\ell(s^k) = k \ell(s)$ for every isometry $s$ and every $k \in \N$. In particular 
$$\lambda(S^k) \geq k \lambda(S).$$
From this the left hand side of Claim 7 follows easily, while the right hand side is formal.

Now, given $j \in \N$ and applying the left hand side of $(\ref{claim7})$ to $S^j$ in place of $S$, with $n=1$ and letting $k$ tend to infinity we see that:
$$\frac{1}{j}\lambda(S^j) \leq \limsup_{k \to +\infty} \frac{1}{k} \lambda(S^k)$$
and hence 
$$\lambda_\infty(S) = \max_{j \geq 1} \frac{1}{j}\lambda(S^j) = \limsup_{k \to +\infty} \frac{1}{k} \lambda(S^k).$$

To complete the proof of Lemma \ref{gen-ineq} it only remains to verify that  $\lambda_\infty(S^k)=k\lambda_\infty(S)$ for every $k \in \N$. This is clear by letting $n$ tend to infinity in Claim 7.

\subsection{Joint displacement and circumradius}

We define the minimal circumradius $r(S)$ of $S$ to  be the lower bound of all positive $r>0$ such that there exists some $x \in X$ and a ball of radius $r$ which contains $Sx=\{sx; s \in S \}$. This quantity is closely related to $L(S)$ as the following lemma shows:

\begin{lemma}\label{randl} Suppose $S \subset Isom(X)$ is a finite set.
Then $$r(S) \leq L(S) \leq 2r(S).$$ Moreover $r(gSg^{-1})=r(S)$ for every $g \in Isom(X)$.
\end{lemma}

\begin{proof}  If $r>r(S)$, there exists $x \in X$ such that $Sx$ is contained in a ball of radius $r$. In particular, $d(sx,x) \leq 2r$ for every $s \in S$. This means that $L(S) \leq L(S,x) \leq 2r$. Hence $L(S) \leq 2r(S)$.

Conversely, if $x$ is such that $L(S) \geq L(S,x) - \epsilon$, then $d(sx,x) \leq L(S) + \epsilon$ for all $s \in S$, and in particular $Sx$ belongs to the ball of radius $L(S) + \epsilon$ centered at $x$. Hence $r(S) \leq L(S) + \epsilon$. Since $\epsilon>0$ is arbitrary, we get $r(S) \leq L(S)$.

That $r(gSg^{-1})=r(S)$ for every $g \in Isom(X)$ follows immediately from the definition of $r(S)$.
\end{proof}

\section{General CAT(0) spaces}\label{genCat0}

The goal of this section is to recall some definitions and basic properties of groups acting by isometries on CAT($0$) spaces and describe some basic examples, such as Euclidean and Hilbert spaces, where the Berger-Wang $(\ref{bw})$ identity fails. We also relate the vanishing of $L(S)$ to the presence of fixed points.

We first recall the notion of CAT($0$) space. A good reference book is \cite{BH}.

For a geodesic segment $\sigma$, we denote by $|\sigma|$ its length. 

If $\Delta=(\sigma_1, \sigma_2, \sigma_3)$ is a triangle  in a metric space $X$ with 
$\sigma_i$ a geodesic segment, a triangle $\overline \Delta=(\overline \sigma_1,
\overline \sigma_2, \overline \sigma_3)$ in Euclidean $\R^2$
is called a {\it comparison triangle} if $|\sigma_i|=|\overline \sigma_i|$ for $i=1,2,3$.
A comparison triangle exists if the side-lenghts satisfy 
the triangle inequality. 

We say that a triangle $\Delta$ is CAT($0$) if 
$$d(x,y) \le d(\bar x, \bar y)$$
for all points $x,y$ on the edges of $\Delta$ and the corresponding points
$\bar x, \bar y$ on the edges of the comparison triangle $\overline \Delta$
in $\R^2$. A geodesic space is a CAT($0$) space if all triangles are CAT($0$). Complete CAT($0$) spaces are often called {\it Hadamard spaces}.

A geodesic metric space is CAT($0$) if and only if every geodesic trangle with vertices $a,b,c$ satisfies the following inequality:
\begin{equation}\label{cat0-ineq}2d(a,m)^2 \leq d(a,b)^2 + d(a,c)^2 - \frac{1}{2}d(b,c)^2.\end{equation}

If $X$ is a simply connected, Riemannian manifold whose sectional 
curvature is non-positive, then it is a CAT($0$) space. In particular symmetric spaces of non-compact type are CAT($0$) spaces. So are Euclidean and Hilbert spaces.

\subsection{Minimal displacement of a single isometry}

This section is devoted to the proof of Proposition \ref{prop.monod}  below. This fact is likely to be well-known to experts, but in lack of reference, we decided to include a proof.

 Recal that an isometry $g$ is said to be \emph{semisimple} if the infimum is attained in the definition of 
$$L(g):=\inf_{x \in X} d(gx,x).$$
 When $X$ is a $CAT(0)$ metric space, isometries are classified into three classes (see \cite{BH}): $g$ is said to be

\begin{itemize}
\item elliptic, if $g$ is semisimple and $L(g)=0$ ($\iff$ fixes a point in $X$),
\item hyperbolic, if $g$ is semisimple and $L(g)>0$,
\item parabolic otherwise.
\end{itemize}

It is known that $g$ is elliptic (hyperbolic, parabolic)
if and only if $g^n$ is elliptic (hyperbolic, parabolic, resp.)
for some $n\not=0$, \cite[II.6.7, II.6.8]{BH}.

For $CAT(0)$ spaces, it turns out that the minimal displacement coincides with the rate of linear growth of an arbitrary orbit, namely:

\begin{proposition}\label{prop.monod}
Let $X$ be a CAT(0) metric space, and $g$
be an isometry. Then we have:
$$L(g)=\ell(g).$$
In particular $L(g^n)=nL(g)$ for each $n >0$ and for any point $x \in X$, we have
$$L(g)=  \lim_{n \to \infty} \frac{1}{n}d(g^nx,x).$$
\end{proposition}

\begin{proof}
 It is enough to prove that 
$$ L(g^2) = 2L(g),$$
because iterating we will find that $\ell(g)=\lim L(g^{2^n})/2^n = L(g).$ Since $L(g^2) \leq 2L(g)$ always (see Lemma \ref{gen-ineq}) we only need to verify that $L(g) \leq L(g^2)/2$. For this we have the following simple argument, which we learned from Nicolas Monod. Given $x \in X$ consider the geodesic triangle $x,gx,g^2x$. The mid-point $y$ between $x$ and $gx$ is mapped under $g$ to the mid-point $gy$ between $gx$ and $g^2x$. Using a Euclidean comparison triangle, we see from the CAT(0) assumption that
$$d(y,gy) \leq \frac{1}{2}d(x,g^2x).$$
In particular $L(g) \leq \frac{1}{2}L(g^2,x)$ and minimizing in $x$ we obtain what we wanted.
\end{proof}

\subsection{Hilbert spaces and affine isometric actions}\label{hilbert-sec}
Affine isometric actions on Hilbert spaces have been studied by many authors, in particular in connection to Kazhdan's property $(T)$ and the Haagerup property. We refer the reader to the work  Cornulier-Tessera-Valette \cite{Cornulier-Tessera-Valette} for background ; see also the work Korevaar-Schoen \cite{korevaar-schoen} and \cite{kleiner} for two interesting geometric applications.

When $X$ is a Hilbert space, one can easily relate the vanishing of the quantities $\ell(S)$ and $L(S)$ to the cohomological properties of the affine isometric action associated to the finitely generated group $\Gamma:=\langle S \rangle$. Let us describe briefly here this connection. The linear part of the $\Gamma$ action by isometries on $X$ is a unitary representation $\pi$ of $\Gamma$ on $X$. The translation part is a cocycle, i.e. a map $b:\Gamma \to X$ such that $b(gh) = \pi(g)b(h) + b(g)$ for all $g,h \in \Gamma$ ; we denote the linear space of cocycles from $\Gamma$ to $X$ by $Z$. A cocycle is a called a co-boundary if there is $x \in X$ such that $b(g)=\pi(g)x - x$. The closure $\overline{B}$ of the space of $B$ co-boundaries gives rise to the reduced first cohomology group of the action $\overline{H}^1(\pi,X)= Z/\overline{B}$. 

It is also natural to study the growth of cocycles. In \cite{Cornulier-Tessera-Valette} a cocycle $b:\Gamma \to X$ is called \emph{sublinear} if $||b(g)|| = o (|g|_S)$, when the word length $|g|_S$ with respect to the generating set $S$ goes to infinity. It is straightforward to check that $b$ is sublinear when $b \in \overline{B}$ \cite[Cor. 3.3]{Cornulier-Tessera-Valette}.

The relationship with our quantities $\ell(S)$ and $L(S)$ is a follows:

\begin{enumerate}
\item $L(S)=0$ if and only if $b$ vanishes in $\overline{H}^1$, i.e. $b \in \overline{B}$,
\item $\ell(S)=0$ if and only if $b$ is sublinear.
\end{enumerate}

We note that there are classes of discrete groups (e.g. polycyclic groups as shown in \cite[Thm 1.1]{Cornulier-Tessera-Valette}) for which $\ell(S)=0$ if and only if $L(S)=0$. While there are others \cite[Prop. 3.9]{Cornulier-Tessera-Valette},  where we may have $L(S)>0$ and $\ell(S)=0$. We will show below (in Proposition \ref{linearescape}) that the latter examples can only happen in infinite dimension.

\subsection{Euclidean spaces}\label{euclid-sec}
In this subsection we briefly describe the case when $X=\R^d$ is a Euclidean space. Proofs will be given in Section \ref{euclidean}.

\begin{proposition}\label{dim23} A subgroup $G$ of $Isom(\R^d)$ with $d=2,3$ all of whose elements have a fixed point must have a global fixed point. In particular $\lambda_\infty(S)=0$ implies $L(S)=\ell(S)=0$.
\end{proposition}

This is no longer true in dimension $4$ and higher. Indeed we have:

\begin{proposition} When $d \geq 4$, one can find a finite set $S$ in $Isom(\R^d)$ such that $\ell(S)>0$ but $\lambda_\infty(S)=0$.
\end{proposition}

The example is given by a subgroup of $Isom (\R^4)$ generated by two rotations within distinct centers and whose rotation parts generate a free subgroup of $\SO(4,\R)$  whose non trivial elements never have $1$ as an eigenvalue. See Example \ref{exBass}.

Consequently:

\begin{corollary} The Berger-Wang identity $(\ref{bw})$ fails on Euclidean spaces of dimension $d \geq 4$.
\end{corollary}

However we will show that if a finite set of isometries does not admit a global fixed point, then it always has a positive rate of escape. Namely:

\begin{proposition} If $S$ is a finite set in $Isom(\R^d)$, then the following are equivalent:
\begin{enumerate}
\item $L(S)=0$,
\item $\ell(S)=0$, 
\item $S$ has a common fixed point.
\end{enumerate} 
\end{proposition}

In other words: if $\Gamma$ is a group and $\pi$ a finite dimensional unitary representation, then for any cocycle $b:\Gamma \to \mathcal{H}_\pi$ the following are equivalent:
\begin{enumerate}
\item $b$ is in the closure of coboundaries,
\item $b$ is sublinear, 
\item $b$ is a coboundary.
\end{enumerate} 

This conveniently complements \cite[Cor 3.7]{Cornulier-Tessera-Valette}.

\subsection{A diffusive lower bound on joint displacement}

With just the CAT($0$) property one always gets $\sqrt{n}$ growth for the joint displacement $L(S^n)$. This may not seem very surprising as already any random walk is expected to have at least a diffusive behaviour (see Remark \ref{gromov-rem} below). However one interesting feature of the following lower bound is the absence of any additive constant and the linear dependence in terms of $L(SS^{-1})$.

\begin{proposition}\label{disp} Let $X$ be a $CAT(0)$  geodesic metric space and $S$ a finite subset of $Isom(X)$. For every $n\in \N$ we have:
$$L(S^n) \geq \frac{\sqrt{n}}{2}L(SS^{-1}).$$
\end{proposition}

\begin{proof} Recall that $r(S)$ denotes the infinimum over all points $x,y$ of the radius $r$ of the balls centered at $y$ which contain $Sx$ and that   (Lemma \ref{randl}) $r(S) \leq L(S)$. By definition of $r(S^n)$ if $r>r(S^n)$, then there exists some $x,y \in X$ such that $S^{n}x \subset B(y,r)$. This means that $sS^{n-1}x$ lies in $B(y,r)$ for every $s \in S$ and hence $S^{n-1}x$ lies in the intersection of all balls  $B(s^{-1}y,r)$ for $s \in S$.

\begin{lemma}\label{intersection-cat} Let $X$ be a  CAT($0$)  geodesic metric space and $B(y,r)$ and $B(z,r)$ two balls of radius $r>0$. Let $m$ be a mid-point of a geodesic between $y$ and $z$. Then $B(y,r) \cap B(z,r)$ is contained in the ball centered at $m$ with radius $(r^2 - \frac{1}{4}d(y,z)^2)^{1/2}$.

\end{lemma}
\begin{proof} This is straightforward from the CAT($0$) inequality $(\ref{cat0-ineq})$.
\end{proof}

Now from this lemma $s_1^{-1}B \cap s_2^{-1}B$ is contained in the ball centered at the mid-point between $s_1^{-1}y$ and $s_2^{-1}y$ and with radius $(r^2 - \frac{1}{4}d(s_1^{-1}y,s_2^{-1}y)^2)^{1/2}$. It follows that $S^{n-1}x$ lies in a ball of that radius.

By definition of $r(S^{n-1})$, it follows that $r(S^{n-1})^2 \leq r^2 - \frac{1}{4}d(s_2s_1^{-1}y,y)^2$. This holds for all $s_1,s_2 \in S$ and all $r>r(S^n)$, thus: $\frac{1}{4}L(SS^{-1},y)^2 \leq r(S^n)^2 - r(S^{n-1})^2$. We conclude that $L(SS^{-1})^2 \leq 4(r(S^n)^2-r(S^{n-1})^2)$.
Finally, summing over $n$, we obtain the desired result.
\end{proof}

The above lower bound is useful to show that even if $L(SS^{-1})$ is very small, there will be some controlled $n$ for which $L(S^n)$ has macroscopic size.

\begin{remark}\label{rkCTV} The behaviour in $\sqrt{n}$ of the lower bound is sharp for general CAT($0$) spaces. Indeed Cornulier, Tessera and Valette exhibited an affine isometric action of the free group on a Hilbert space with a $\sqrt{n}$ upper bound on cocycle growth, see \cite[Prop. 3.9]{Cornulier-Tessera-Valette}. When $X$ is a product of symmetric spaces and Euclidean spaces however, then $L(S^n)$ grows linearly provided $L(S)>0$. This follows from the combination of Propositions \ref{linearescape} and  \ref{ell-L-sym}.
\end{remark}

\begin{remark} In \cite{lafforgue} V. Lafforgue gave another proof of Shalom's theorem \cite{shalom-inventiones} that a group without property $(T)$ has non-trivial first reduced cohomology. His main lemma is essentially a version of Proposition \ref{disp} in the case when $X$ is a Hilbert space. 
\end{remark}

\begin{remark}\label{gromov-rem}
 In \cite{gromov-random} Gromov investigates the growth of the energy of a random walk on a general CAT($0$) space. In particular from his \emph{harmonic growth inequalities} in \cite[3.4]{gromov-random} one gets that the average displacement of a random walk grows like $\sqrt{n}$, which is another way to recover the $\sqrt{n}$ growth in Proposition \ref{disp}.
\end{remark}

\subsection{Fixed points on the boundary}
Let $X$ be a complete CAT($0$) geodesic space. Note that the sublevel sets of the function $x \mapsto L(S,x)$ are convex subsets. Note further that any nested sequence of bounded closed convex non-empty sets has non-empty intersection (see e.g. \cite{monod}*{Theorem 14}). In particular we have:

\begin{proposition} If $x \mapsto L(S,x)$ tends to infinity when $x$ leaves every bounded subset of $X$, then $L(S,x)$ achieves its minimum $L(S)$ at some point $x_0 \in X$.
\end{proposition}

If $X$ is locally compact one has:

\begin{proposition} Suppose $X$ is a CAT($0$) locally compact geodesic space such that no sublevel sets of $x \mapsto L(S,x)$ is bounded. Then $S$ has a global fixed point on the visual boundary $\partial X$.
\end{proposition}

For this see for example \cite{korevaar-schoen}*{2.2.1}. Recall that the visual boundary $\partial X$ is the set of equivalence classes of infinite geodesic rays $[0,+\infty) \to X$, where two rays $(x_t)_t$ and $(y_t)_t$ are equivalent if $d(x_t,y_t)$ is uniformly bounded. We get:

\begin{corollary}\label{bdy-fixed} Suppose $X$ is a CAT($0$) locally compact geodesic space. And $S \subset Isom(X)$ a finite set of isometries such that $L(S)=0$, then $S$ has a global fixed point in $X \cup \partial X$.
\end{corollary}

\subsection{CAT($0$) and Gromov hyperbolic spaces}

If $X$ is only assumed to be $CAT(0)$, then there may be parabolic isometries $g$ with $L(g)>0$. For example, the product of a parabolic isometry of the hyperbolic plane with a non trivial translation of $\R$. A less obvious example is given by the warped product $(\exp(-y)+C)dx^2+dy^2$ on $\R^2$, with $C>0$. This space is $CAT(0)$ and translation $(x,y) \mapsto (x+1,y)$ is a parabolic isometry with positive displacement $C$ (see Remark 2.4. in \cite{karlsson-margulis}).

However if we assume additionally that $X$ is Gromov hyperbolic, then parabolic isometries must have zero displacement. More precisely, we have:

\begin{proposition}\label{axis}
Let $X$ be a complete CAT(0) space which is
Gromov hyperbolic.
Let $g$ be an isometry of $X$ with $L(g)>0$.
Then $g$ is hyperbolic in the sense that $L(g)$
is achieved on a geodesic, which is the unique
$g$-invariant geodesic.
\end{proposition}

 For the definition of Gromov hyperbolic spaces we refer to Section \ref{serre-hyp} below as well as the books \cite{BH, Bowditch,coornaert-delzant}. 

Before we start the proof, we quote some results. Let $X$ be a complete CAT(0) space and $\partial X$
its visual boundary at infinity. Set $\overline{X}=X \cup \partial X$. For a closed convex subset $C \subset X$, let $\overline{C} \subset \overline{X}$ be its closure.

\begin{theorem}[\cite{monod}*{Prop. 23}]\label{monod.boundary}
Assume that a complete CAT(0) space $X$ is Gromov hyperbolic. 
Then for any nested family $\mathcal F$ of non-empty closed convex subsets $C \subset X$, 
the intersection $\cap_{C \in \mathcal F} \overline C$ is non-empty.
\end{theorem}

In particular we have the following.
\begin{corollary}\label{parabolic.fix}
Let $X$ be a Gromov hyperbolic complete CAT(0) space. Then a parabolic isometry $g$ has a fixed point in $\bar X \backslash X$.
\end{corollary}
\proof
The family of sublevel sets of $x \mapsto L(g,x)$ has empty intersection in $X$, for otherwise $g$ would have a fixed point in $X$. By the previous theorem, the closures in $\bar X$ of these sublevel sets has non-empty intersection in $\bar X \backslash X$.
Each point in the intersection is fixed by $g$.
\qed

We quote a standard fact on $\delta$-hyperbolic space.
\begin{theorem}\label{delta.axis}
Let $X$ be a $\delta$-hyperbolic space and $g$ an isometry of $X$.
If $L(g)$ is sufficiently large, compared to $\delta$, then 
$g$ is ``hyperbolic" in the sense that there is an infinite quasi-geodesic
$\gamma$ in $X$ which is $g$-invariant. 

Moreover, let $x \in X$ be any point and $m$ the midpoint of 
a geodesic segment from $x$ to $g(x)$. Form a $g$-invariant piecewise
geodesic, $\gamma$, joining the points $\{g^n(m)\}_{n \in \Z}$ in this order
by geodesics. 
Then, there exists a constant $M$, which depends only on $\delta$, such that 
for any points $p,q \in \gamma$, the Hausdorff distance between the part
in $\gamma$ from $p$ to $q$ and a geodesic from $p$ to $q$ is at most $M$.
\end{theorem}

\begin{remark}
The first claim appears in \cite[\S8, Prop 24]{ghys}
with $L(g) > 26 \delta$.
The advantage to retake $x$ to $m$ is that we have uniform bounds on the quasi-geodesic
constants. In fact, we have 
$2|m-g(m)|-|m-g^2(m)| \le 4 \delta$, namely, the three points $m, g(m), g^2(m)$
are nearly on a geodesic. It follows that the path $\gamma$ is a $(K,L)$-quasi geodesic
with constants $K, L$ depending only on $\epsilon$. Once we have that, the 
existence of $M$ is by the Morse lemma.
A detailed argument is, for example, in  \cite{fujiwara} for $L(g) > 1000 \delta$.
\end{remark}

For an isometry $a$ of $X$ and a constant $C$,
define 
$$Fix_C(a)=\{x\in X| d(x,a(x)) \le C \}.$$
This is a closed, possibly empty,  set. It is convex if $X$ is CAT(0).

We start the proof of the proposition \ref{axis}.
\proof
Since it suffices to show that $g^n$ is hyperbolic for some $n>0$, 
by Proposition \ref{prop.monod},
we may assume that $L(g)$ is as large as we want 
by replacing $g$ by a high power. 
So, we assume that $L(g)$ is large enough compared to the hyperbolicity constant $\delta$, so that Theorem \ref{delta.axis} applies to $g$.

Set $C_0=L(g)$.  We will show $Fix_{C_0}(g)$ is not empty. 
For each $C > C_0$, the set $Fix_C(g)$
is non-empty convex set which is invariant by $g$, 
and is $\delta$-hyperbolic. 

By Theorem \ref{delta.axis}, there is a $g$-invariant quasi-geodesic, 
but we may assume that this path is contained in $Fix_C(g)$.
Indeed, if we start with a point $x \in Fig_C(g)$, the point $m$
is also in $Fig_C(g)$ since it is convex, and so are all points $g^n(m)$. Therefore $\gamma$ in the theorem is contained in $Fic_C(g)$. Let us denote this $\gamma$ by $\gamma_C$.

Now, the Hausdorff distance of any two of those quasi-geodesics
$\gamma_C, C >C_0$, is at most $2M+2\delta=M'$.
To see that, observe that the Hausdorff distance of the two quasi-geodesics
is finite since they are both $g$-invariant. 
Since they are both bi-infinite, in fact the bound
is $2M+2\delta$.

In particular there is a metric ball, $B \subset X$,  of radius $M'$ 
such that  for any $C > C_0$, $B \cap \gamma_C \not= \emptyset$,  therefore
$B \cap Fix_C(g)$ is not empty.

Set $\mathcal F = \{B \cap Fix_C(g)| C >C_0\}$. It is a nested family of bounded convex closed subsets of $X$. Hence the family has a non-empty intersection, in particular
$\cap_{C >C_0} Fix_C(g)$ is not empty, and therefore
$Fix_{C_0}(g)$ is not empty. 
\qed

\bigskip

In view of the above propositions, in the case of a complete CAT(0) and $\delta$-hyperbolic space, the classification of isometries becomes:

\begin{itemize}
\item elliptic when $g$ fixes a point in $X$,
\item hyperbolic when $L(g)>0$ (in this case $g$ fixes a unique geodesic and acts by translation by $L(g)$ on it),
\item parabolic when $L(g)=0$ (in this case $g$ does not fix a point in $X$ but fixes a point in $\partial X$).
\end{itemize}

\section{Quantitative Serre's lemma for groups acting on trees}\label{trees}

The goal of this section is to prove Proposition \ref{formula.S} below, which is an extension of the following well-known lemma due to Serre (\cite[I. Prop. 26]{serre}).  In the next section we will generalize everything to $\delta$-hyperbolic spaces.

\begin{lemma}[Serre's lemma] \label{serre-lemma} If $a,b$ are isometries of a simplicial tree, such that $a$, $b$ and $ab$ each have some fixed point, then $a$ and $b$ have a common fixed point. 
\end{lemma}

We first note the following simple fact:

\begin{lemma} If $a$ is an isometry of a simplicial tree, then $L(a)=\ell(a)=\lambda(a)$.
\end{lemma}

Indeed either $a$ fixed a point on the tree, and all quantities vanish, or $a$ is a hyperbolic isometry translating along an axis $\Delta_a=\{x ; d(ax,x)=L(a)\}$, so that in particular $L(a^n)=nL(a)$ and thus $\lambda(a)=L(a)$.





\begin{proposition}[A formula for the joint minimal displacement of a pair]\label{formula.pair} Let $a,b$ be isometries of a simplicial tree. Then 
$$L(\{a,b\}) = \max \{ L(a),L(b) , \frac{L(ab)}{2}\} = \lambda_2(\{a,b\}).$$
\end{proposition}

More generally this formula extends to an arbitrary finite subset of isometries:

\begin{proposition}[A formula for the joint minimal displacement] \label{formula.S} Let $S$ be a finite set of isometries of a simplicial tree. Then 
$$L(S) = \max_{a,b \in S} \{ L(a) , \frac{L(ab)}{2}\} = \lambda_2(S).$$
\end{proposition}

Recall that by definition $L(S)= \inf_{x \in X} \max_{s \in S} d(x,sx)$ and $L(g)=L(\{g\})$ for an isometry $g$. And $\ell(S)= \lim L(S^n)/n$.

\begin{corollary}[Growth of joint minimal displacement] \label{growth.S} Let $S$ be a finite set of isometries of a simplicial tree. Then $L(S)=\ell(S)=\lambda_2(S)$ and indeed for every $n \in \N$,
$$L(S^n) = nL(S).$$
\end{corollary}

We now pass to the proofs of the above statements. The proofs are quite simple and we do not claim much originality here. Even though we were not able to find the above statements in the existing literature, all ingredients in their proofs are well-known and can be found, for example in \cite[3.3]{chiswel}. 

First we make the following simple observations:  Let  $g$ is an isometry of a simplicial tree $X$, let $\Delta_g$ be the ``axis'' of $g$, $\Delta_g:= \{x \in X ; d(gx,x)=L(g)\}$ and let $x$ be any point in $X$.

\begin{enumerate}[label=(\alph*)]
\item we have $d(gx,x) = 2d(x,\Delta_g) + L(g)$,
\item $L(g)=d(gm,m)$, where $m$ is the midpoint of the geodesic segment $[x,gx]$. 
\end{enumerate}

To see this, consider a point $y$ on $\Delta_g$, which minimizes the distance between $x$ and $\Delta_g$ and note that the geodesic between $y$ and $gy$ lies entirely in $\Delta_g$, so that in particular the concatenation of the geodesic segments from $x$ to $y$ and from $y$ to $gy$ remains a geodesic. Furthermore the midpoint between $x$ and $gx$ will coincide with the midpoint between $y$ and $gy$.

\begin{proof}[Proof of Proposition \ref{formula.pair}] Let the point $x$ realize the infimum of $ \max\{d(x,ax), d(x,bx)\}$. Assume that $L(\{a,b\}) > \max\{L(a),L(b)\}$. This implies that the axes $\Delta_a$ and $\Delta_b$ do not intersect. We claim that $x$ must be the midpoint of the geodesic segment  $[ax,bx]$ and that $d(x,ax)=d(x,bx)$. 

To see this look at the triangle with vertices $x,ax$ and $bx$. If it is not flat with $x$ the midpoint of  $[ax,bx]$, then the geodesics $[x,ax]$ and $[x,bx]$ intersect on some segment near $x$. But $y \mapsto d(y,ay)$ decreases as the point  $y$ moves away from $x$ on this segment. This is a consequence of item $(a)$ above: it decreases strictly unless $x$ is on the axis $\Delta_a$ of $a$, and it decreases until it reaches the midpoint $m$ of the segment $[x,ax]$, where we have $L(a)=d(m,am)$ by item (b) above. The same holds for $b$. Since $L(\{a,b\},x) > \max\{L(a),L(b)\}$, we conclude that unless $x$ is on both axes $\Delta_a$ and $\Delta_b$, this contradicts the minimality of $x$.

For the same reason $x$ is also the midpoint of the geodesic segments $[b^{-1}x,ax]$ and  $[bx,a^{-1}x]$. By item $(b)$ above applied to $x$ viewed as a midpoint of $[b^{-1}x, ax]$ (noting that $ax=(ab)b^{-1}x$) we get:
$$d(ab x, x)=L(ab).$$
But
\begin{equation}\label{ab} d(abx,x)= d(bx,a^{-1}x) = 2d(ax,x) = 2d(bx,x)=2L({a,b}).\end{equation}
So this shows that $L({a,b})=L(ab)/2$  and this ends the proof of the proposition.
\end{proof}

We may now extend our formula to an arbitrary finite set of isometries. Recall that for an isometry $a$ of a metric space $(X,d)$ and a number $A\geq 0$ we denote $Fix_A(a):=\{ x \in X ; d(ax,x) \leq A\}$.

\begin{proof}[Proof of Proposition \ref{formula.S}] This follows easily from Proposition \ref{formula.pair}. Indeed set $M=\max_{a,b \in S} \{L(a), L(ab)/2\}$. By  Proposition \ref{formula.pair} we have $M\geq L(\{a,b\})$ for every pair $a,b \in S$. Hence $Fix_M(a)$ and $Fix_M(b)$ intersect for every pair $a,b \in S$. However recall that each $Fix_M(a)$ is a convex subset (it is a subtree), and in any tree any collection of subtrees which pairwise intersect non trivially must have a non-empty intersection. Therefore there is a point $z$ such that $d(z,az) \leq M$ for all $a \in S$ and we have established $L(S)\leq M$. The opposite inequality is obvious. 
\end{proof}

\begin{proof}[Proof of Corollary \ref{growth.S}] This is obvious combining  Lemma \ref{disp} and Proposition \ref{formula.S}.
\end{proof}


\section{The joint minimal displacement in hyperbolic spaces}\label{serre-hyp}

In this section we extend the results of Section \ref{trees} to isometries of $\delta$-hyperbolic spaces and prove Theorem \ref{bochi-hyp} from the introduction, which implies Theorem \ref{bw-thm} in the case of hyperbolic spaces.

First we recall the definition of $\delta$-hyperbolicity and prove a general lower bound on $L(S^n)$ using the same circumcenter argument as in Proposition \ref{disp}.

Gromov introduced the notion of a $\delta$-hyperbolic metric space in  \cite{gromov.hyp}.  Recall that a metric space $(X,d)$ is said to be geodesic if any two points can be joined by a geodesic (i.e. length minimizing) continuous path. A geodesic triangle $\Delta=(\sigma_1, \sigma_2, \sigma_3)$ in a geodesic space $X$ is said to be $\delta$-{\it thin} if $\sigma_i$ is contained in the $\delta$-neighborhood
of $\sigma_j \cup \sigma_k$ for any permutation of $(i,j,k)$ of $(1,2,3)$.

There is another notion that is closely related to 
$\delta$-thinness. A point $c$ is a $\delta$-{\it center} of 
a  geodesic triangle $\Delta=(\sigma_1, \sigma_2, \sigma_3)$
if the distance from $c$ to every $\sigma_i$ is $\le \delta$,
\cite{Bowditch}.
It is easy to see that if a (geodesic) triangle is 
$\delta$-thin, then it has a $\delta$-center.
Conversely, if a geodesic triangle has a $\delta$-center, then 
it is $6 \delta$-thin (see \cite[Lemma 6.5]{Bowditch}.

The metric space $(X,d)$ is said to be $\delta$-{\it hyperbolic} if any geodesic triangle is $\delta$-thin. 
Good references on the geometry of $\delta$-hyperbolic spaces include \cite{BH, Bowditch, coornaert-delzant}.

Note that we can also define the $\delta$-hyperbolicity of a geodesic 
space if every geodesic triangle has a $\delta$-center
(for example \cite{Bowditch}).
As we said, those two definitions are equivalent, but 
the constant $\delta$ may differ, so if we want to 
stress the difference, we call it $\delta$-{\it center}-hyperbolicity.

 We note that there is also a notion of $\delta$-hyperbolic space where the space is not required to be geodesic (via the so-called $4$-points condition, see \cite{Bowditch}). However the geodesic assumption is necessary in our theorems. For example it can be seen easily that results such as Lemma \ref{translation.length} or Theorem  \ref{formula.S.hyp} fail when the space is not assumed geodesic. As pointed out in \cite[Remark 4.4]{oregon-reyes} one can take for instance the usual hyperbolic disc with a large ball (centered say at the origin) removed. This is still a $\delta$-hyperbolic space in the sense of the $4$-points condition with the same $\delta$, but an elliptic isometry centered at the origin will have $\ell(g)=0$ while $L(g)$ will be large.

\subsection{Using the circumradius to relate joint minimal and joint asymptotic displacements}

Here we show:

\begin{proposition}\label{disp2} Let $X$ be a $\delta$-hyperbolic geodesic metric space and $S$ a finite subset of $Isom(X)$. For every $n\in \N$ we have:
$$\frac{1}{n}L(S^n) \geq \frac{L(SS^{-1})}{2} - 2 \delta $$
\end{proposition}

\begin{proof} The argument is analogous to that of Proposition \ref{disp} in the CAT(0) setting.  Recall that $r(S)$ denotes the infinimum over all points $x,y$ of the radius $r$ of the balls centered at $y$ which contain $Sx$. Also recall  (Lemma \ref{randl}) that $$r(S) \leq L(S).$$ Let $n \geq 1$ and $r>r(S^{n})$. By definition of $r(S^n)$, there exists some $x,y \in X$ such that $S^{n}x$ is contained in the ball $B$ of radius $r$ centered at $y$. This means that $sS^{n-1}x$ lies in $B$ for every $s \in S$ and hence $S^{n-1}x$ lies in the intersection of all balls  $s^{-1}B$ for $s \in S$. We require the following

\begin{lemma}\label{intersection-hyp} Let $X$ be a $\delta$-hyperbolic geodesic metric space and $B(y,r)$ and $B(z,r)$ two balls of radius $r>0$. Let $m$ be a mid-point of a geodesic between $y$ and $z$. Then $B(y,r) \cap B(z,r)$ is contained in the ball centered at $m$ with radius $r - \frac{1}{2}d(y,z) + 2\delta$.

\end{lemma}
\begin{proof} This is straightforward from the definition of $\delta$-hyperbolicity.
\end{proof}

Now from this lemma $s_1^{-1}B \cap s_2^{-1}B$ is contained in the ball centered at the mid-point between $s_1^{-1}y$ and $s_2^{-1}y$ and with radius $r- \frac{1}{2}d(s_1^{-1}y,s_2^{-1}y) + 2\delta$. It follows that $S^{n-1}x$ lies in a ball of that radius.

By definition of $r(S^{n-1})$, it follows that $r(S^{n-1}) \leq r - \frac{1}{2}d(s_2s_1^{-1}y,y) + 2\delta$. This holds for all $s_1,s_2 \in S$ and all $r>r(S^n)$, thus: $\frac{1}{2}L(SS^{-1},y) \leq r(S^n) - r(S^{n-1}) +2 \delta$. We conclude that $\frac{1}{2}L(SS^{-1}) \leq r(S^n)-r(S^{n-1}) +2\delta$.
Finally, summing over $n$, we obtain the desired result.
\end{proof}

\bigskip

\begin{corollary}\label{cheap-growth}Let $X$ be a $\delta$-hyperbolic geodesic metric space and $S$ a finite symmetric subset of $Isom(X)$. We have: 
$$\ell(S) \leq \frac{L(S^2)}{2} \leq \ell(S) + 2 \delta.$$
\end{corollary}

\begin{proof} Clear from the combination of Proposition \ref{disp2} and Lemma \ref{gen-ineq}.
\end{proof}

This result was obtained quite cheaply using the circumradius. Using a more delicate analysis, based on a refinement of Serre's lemma, we will prove below in Proposition \ref{formula.S.hyp} a much stronger result, a Bochi-type inequality, which yields a hyperbolic element of large translation length.

\subsection{Powers of a single element}

We begin by showing that the asymptotic translation length of a single element is controled by its translation length, provided the latter is large enough.

\begin{lemma}\label{translation.length} There is  a universal constant $C>0$ such that the following holds. Let $g$ be an isometry in a $\delta$-hyperbolic space. Then for all $m >0$, $L(g^m) \geq  m(L(g) - C\delta)$. In particular:
$$L(g) - C \delta \leq \ell(g) \leq L(g).$$
\end{lemma}

\begin{proof} We first prove the lemma for $m=2$. For the proof we may assume $\delta>0$. Pick a point $y$ such that $L(g^2,y) \leq L(g^2) + \delta$. Consider the three points $y,gy$ and $g^2y$. We will show that unless $L(g)=O(\delta)$ this triangle is $O(\delta)$-flat and $gy$ is $O(\delta)$ away from a mid point of $[y,g^2y]$. 

 Let $q$ be a point at distance at most $\delta$ from all three geodesics joining these points, which is given to us by the $\delta$-hyperbolicity assumption. Since $d(y,gy) = d(gy,g^2y)$ we conclude that $|d(y,q) - d(q,gy)| = O(\delta)$. In a $\delta$-hyperbolic space any two geodesics joining two given points are at distance $O(\delta)$ of each other. Therefore $d(p,q) =O(\delta)$, where $p$ is a mid-point between $y$ and $g^2y$. 

Now if $r$ is a mid-point between $y$ and $gy$, then $r$ is $O(\delta)$ close to either $[q,gy]$  or to $[y,q]$ according as $d(y,r) \geq d(y,q)$ or not. In the first  case $d(r,gr)= O(\delta)$, which implies that $L(g) =O(\delta)$ and there is nothing to prove. In the second case $gr$ will be at most $O(\delta)$ away from $[q,g^2y]$ and we get:
$$|d(y,g^2y) - [d(y,r) + d(r,q) + d(q,gr) + d(gr,g^2y)]| = O(\delta),$$
while $d(r,q) \leq d(y,r) + O(\delta)$ and $d(q,gr) \leq d(gr,g^2y)+O(\delta)$. Combining these two facts we get:
$$2d(r,gr) \leq d(y,g^2y) + O(\delta),$$
from which the inequality $2L(g) \leq L(g^2) + O(\delta)$ follows immediately. 

We now pass to the general case, when is an arbitrary integer $m > 2$. First we observe that if $L(g^2) \geq 2(L(g) - C\delta)$ for every isometry $g$, then a straigthforward induction shows that $L(g^{2^n}) \geq 2^n (L(g) - 2C \delta)$ for every $g$ and every integer $n$. But observe that given any point $x$ the sequence $m \mapsto L(g^m,x)$ is subadditive. In particular, by the subadditive lemma, $\{\frac{L(g^m,x)}{m}\}_m$ converges towards $$\inf_{m \geq 1} \frac{L(g^m,x)}{m} = \lim_{m \mapsto + \infty} \frac{L(g^m,x)}{m}.$$ Letting $m$ grow along powers of $2$ we see that the latter is at least $L(g)-2C\delta$. This means that $L(g^m,x) \geq m(L(g)-2C\delta)$ for every integer $m$ and every point $x$. In particular $L(g^m) \geq m(L(g)-2C\delta)$ as desired.
\end{proof}



\subsection{Bochi-type formula for hyperbolic spaces}

Here we prove Theorem \ref{bochi-hyp} from the introduction.

\begin{proposition}[Joint minimal displacement of a pair]\label{formula.pair.hyp}  There is an absolute constant $K>0$ such that if $\delta\geq 0$ and  $(X,d)$ is a $\delta$-hyperbolic geodesic space, and $a,b$ are two isometries of $X$, then 
$$ L(\{a,b\}) -K \delta \leq  \max \{ \ell(a) , \ell(b), \frac{\ell(ab)}{2}\} \leq L(\{a,b\}).$$
\end{proposition}

Recall our notation used throughout for a finite set $S$ of isometries of a metric space $(X,d)$ we denote by $L(S,x):= \max_{s \in S} d(x,sx)$, $L(S):=\inf_{x \in X} L(S,x)$ and $\ell(S):=\lim L(S^n)/n$, while $\lambda_k(S) := \max_{1\leq j\leq k} \frac{1}{j} \max_{g \in S^j}  \ell(g)$.

\begin{theorem}[Theorem \ref{bochi-hyp} from the introduction] \label{formula.S.hyp} There is an absolute constant $K>0$ such that if $\delta\geq 0$ and  $(X,d)$ is a $\delta$-hyperbolic geodesic space, and $S$ is a finite set of isometries of $X$, then 
$$L(S) -K \delta \leq \lambda_2(S) = \max_{a,b \in S} \{ \ell(a) , \frac{\ell(ab)}{2}\} \leq L(S).$$
\end{theorem}

As a consequence of this proposition, we obtain the following strengthening of Proposition \ref{disp2}.

\begin{corollary}[Growth of joint minimal displacement] \label{growth.S.hyp} There is an absolute constant $K>0$ such that the following holds. Let $S$ be a finite set of isometries of a $\delta$-hyperbolic space. Let $n \in \N$. Then
$$n (L(S) - K\delta) \leq \lambda_2(S^n) \leq L(S^n).$$
Moreover $\lambda_\infty(S) = \ell(S)$.
\end{corollary}

We now pass to the proofs. As always in $\delta$-hyperbolic geometry arguments are modeled on the tree case and this is why we decided to include the special case of trees separately even though it is of course implied by the $\delta$-hyperbolic case by setting $\delta=0$. A key ingredient will be Lemma \ref{displacement.versus.distance} below, which computes the displacement of an isometry in terms of its translation length and the distance to its axis.

Before we embark in the proof we begin by recalling the following basic fact about tree approximation in $\delta$-hyperbolic spaces. 

\begin{lemma}(see \cite[Prop. 6.7]{Bowditch}) \label{bowditchLemma} Given $k$ points $x_1,...,x_k$ in a $\delta$-hyperbolic metric space $(X,d)$ there is a metric tree $T$ embedded in $X$ such that for all $i,j$
\begin{equation}\label{comp.dist}d(x_i,x_j) \leq d_T(x_i,x_j) \leq d(x_i,x_j) + C_k \delta,\end{equation} where $d_T(x_i,x_j)$ is the length of the geodesic path joining $x_i$ and $x_j$ in $T$. Here $C_k$ is a constant depending only on $k$.
\end{lemma}

Recall that for an isometry $a$ of a metric space $(X,d)$ and a number $A\geq 0$ we denote $$Fix_A(a):=\{ x \in X ; d(ax,x) \leq A\}.$$

\begin{lemma}\label{displacement.versus.distance}There are absolute constants $c,K>0$ with the following property.  Let $(X,d)$ be a geodesic $\delta$-hyperbolic metric space. Let $M\geq c \delta$ and $a \in Isom(X)$. Assume that $Fix_M(a)$ is non-empty and let $x \notin Fix_M(a)$. Then 
\begin{equation}\label{dist}|2 d(x, Fix_{M}(a))+M - d(ax,x)| \leq K\delta,\end{equation}
and
\begin{equation}\label{midpoint} d(q,aq) \leq  M + K\delta,\end{equation}
for every mid-point $q$ between $x$ and $ax$.
\end{lemma}

\begin{proof} Let $\epsilon>0$ and pick $y \in Fix_M(a)$ such that $d(x,y) \leq d(x,Fix_M(a))+ \epsilon$. We may write:
$$d(ax,x) \leq d(ax,ay)+d(ay,y)+d(y,x) \leq 2d(x,y) +M,$$
so one side of $(\ref{dist})$ follows immediately. 

To see the other  side consider the four points $x,y,ax,ay$ and apply Lemma \ref{bowditchLemma} to these four points with $k=4$. Let $w$ be the intersection of the three geodesics in $T$  connecting $ax$, $ay$ and $x$, and let $z$ be the intersection of the three geodesics in $T$  connecting $ax$, $y$ and $x$. First we claim that:
\begin{equation}\label{yz} d(y,z) \leq \epsilon + O(\delta)
\end{equation}

To see this, first recall that in a $\delta$-hyperbolic space the Gromov product of a pair of points based at a third point is equal up to an error of $4\delta$ to the distance from the third point to the geodesic between the pair (see \cite[Lemma 6.2]{Bowditch}). Now note that if $\alpha$ is the geodesic joining $y$ and $ay$, then $\alpha$ is entirely contained in $Fix_M(a)$. Consequently $d(x,\alpha) \geq d(x,y) - \epsilon$. Hence $(y,ay)_x \geq d(x,\alpha) - 4\delta \geq d(x,y) - \epsilon - 4\delta$, which unfolding the Gromov product means:
$$d(x,z) + d(z,ay) \geq d(x,ay) \geq d(x,y)+ d(y,ay) - 2\epsilon - O(\delta).$$
In view of $(\ref{comp.dist})$, it follows that the same holds with $d_T$ in place of $d$. However:
$$ d_T(x,y) + d_T(y,ay) =  d_T(x,z) + d_T(z,ay) + 2d_T(y,z),$$
so we conclude that $(\ref{yz})$ holds.

At this pont we note that $d(ay,y) \geq M - 2\epsilon$. Indeed by the intermediate value theorem, we may pick a point $u$ on a geodesic between $x$ and $y$ such that $d(au,u)=M$. Then $d(x,u) \leq d(x,y)$, while  $u \in Fix_M(a)$ so that $d(x,u) \geq d(x,Fix_M(a)) \geq d(x,y) - \epsilon$. Consequently $d(u,y) \leq \epsilon$ and hence $d(ay,y) \geq d(au,u) - 2\epsilon$ as desired.

Now two cases occur, according as $w$ belongs to the tree geodesic $[ay,z]_T$ or to $[z,x]_T$. In the first case, by symmetry the previous argument also shows that 
\begin{equation}\label{yw} d(ay,w) \leq \epsilon + O(\delta)
\end{equation}
Combining $(\ref{yz})$ and $(\ref{yw})$ (and using $(\ref{comp.dist})$ then yields 
\begin{eqnarray*}
d(ax,x) &\geq& d(ax,w)+d(w,z)+d(z,x) - O(\delta) \\
&\geq& d(ax,ay) + d(ay,y) + d(y,x) - 2 \epsilon - O(\delta) \\
 &\geq& 2d(x,Fix_M(a)) + M - 6 \epsilon -O(\delta),
\end{eqnarray*}
from which the remaining side of $(\ref{dist})$ is immediate since $\epsilon$ is arbitrary.

We can also observe now from the second inequality above that any mid-point $p$ between $y$ and $ay$ is satisfies $|d(x,p) - d(ax,p)| = O(\delta)$. The $\delta$-hyperbolicity of the space then implies that $d(p,q) = O(\delta)$ for every mid-point $q$ between $x$ and $ax$. On the other hand $d(p,ap) \leq M$ (because $y \in Fix_M(a)$). Therefore $d(q,aq) \leq M + O(\delta)$. And $(\ref{midpoint})$ holds.

In the second case, $z$ is also the intersection of the tree geodesics between the three points $y, ay$ and $ax$ and the previous argument once again shows that 
\begin{equation}\label{ayz} d(ay,z) \leq \epsilon + O(\delta)
\end{equation}
This means (using $(\ref{yz})$) that $d(ay,y) \leq 2\epsilon +O(\delta)$ and hence that  $M \leq 4\epsilon +O(\delta)$.  So this case happens only when $M=O(\delta)$, a case we discard by choosing $c$ large enough.
\end{proof}

\begin{proof}[Proof of Proposition \ref{formula.pair.hyp}] Let $\epsilon>0$ be arbitrary and let $x$ be a point in $X$ such that $L(\{a,b\},x)\leq L(\{a,b\}) + \epsilon$. At the end of the argument we will let $\epsilon$ tend to $0$. Let $K$ be a large absolute constant, whose value we will specify later. Without loss of generality, we may assume that $L(\{a,b\}) > \max\{L(a),L(b)\} + K\delta + 3\epsilon$, for otherwise there is nothing to prove.

We can apply Lemma \ref{displacement.versus.distance} to $a$ with $M_a=\max\{L(a)+\epsilon,K\delta\}$. Let $\overline{x}$ be a point in $Fix_{M_a}(a)$ such that $d(x,\overline{x}) \leq d(x,Fix_{M_a}(a)) + \epsilon$.  Then $(\ref{dist})$ of Lemma \ref{displacement.versus.distance} shows that the piecewise geodesic path between $x$, $\overline{x}$, $a\overline{x}$ and $ax$ is geodesic up to an error $O(\delta)+\epsilon$. More precisely if $y$ is any point on a geodesic between $x$ and $\overline{x}$, then $(\ref{dist})$ gives
\begin{align}d(ay,y) &\leq d(ay,a \overline{x}) + d(a \overline{x},\overline{x}) + d(\overline{x},y) \nonumber \\  &\leq 2d(\overline{x},y) + L(a) + \epsilon + O(\delta)   \nonumber \\ &\leq  d(ax,x) -2d(x,y) + \epsilon + O(\delta) \label{ttt} \end{align}

Now we may conclude that $|d(ax,x) - d(bx,x)| \leq K\delta + 3\epsilon$. To see this assume by contradiction that $d(x,ax) > d(x,bx) + 2t + \epsilon$ with $2t > 2\epsilon + K\delta$ and $t < d(x,Fix_{M_a}(a))$. We may choose $y$ at distance $t$ from $x$ on the geodesic $[x,\overline{x}]$. Then $L(\{a,b\},x) = d(ax,x)$, while $d(by,y) \leq d(bx,x) + 2t  < d(ax,x) - \epsilon$ and $d(ay,y) \leq d(ax,x) -2d(x,y) + \epsilon +K\delta$ by the above inequality, so $d(ay,y) < d(ax,x)-\epsilon$  so $$L(\{a,b\},y) <L(\{a,b\},x) -\epsilon,$$ which contradicts the choice of $x$. This argument is valid if $ 2d(x,Fix_{M_a}(a)) > 2\epsilon + K\delta$. But we are always in this case, because otherwise $L(\{a,b\},x) \leq d(x,ax) \leq M_a + 2\epsilon +K\delta$, which is in contradiction with our assumption that $L(\{a,b\}) > \max\{L(a),L(b)\} + K\delta + 3\epsilon$.

 Now we claim that $x$ is at most $O(\delta) + \epsilon$ away from the midpoint between $ax$ and $bx$. By the above discussion it is enough to show that it is that close to the geodesic between $ax$ and $bx$. Let us verify this claim. Without loss of generality we may assume that $d(x,Fix_{M_a}(a))  \geq  d(x,Fix_{M_b}(b))$. Now consider the triangle with vertices $x,ax$ and $bx$. It is $\delta$-thin by hyperbolicity. Thus there is a point  $z$  at distance at most $\delta$ from all three geodesics. In particular the geodesic between $z$ and $x$ is in the $O(\delta)$ neighborhood of both the geodesic between $x$ and $ax$ and the geodesic between $x$ and $bx$. If $y$ is any point on $[x,z]$ such that $d(x,y) \leq d(x,Fix_{M_b}(b)) + \epsilon$. Then we also have  $d(x,y) \leq d(x,Fix_{M_a}(a)) + \epsilon$ and the above inequality $(\ref{ttt})$ for $a$ and its analogue for $b$ applied at $y$ show that 
$$L(\{a,b\},y) \leq L(\{a,b\},x) - 2d(x,y) + \epsilon + O(\delta),$$
and thus, given our choice of $x$,
$$d(x,y) \leq \epsilon + O(\delta).$$ Choosing $y$ as far from $x$ as possible, we see that this means either that $d(x,z) \leq \epsilon + O(\delta)$ and we have proved our claim, or that $d(x,Fix_{M_b}(b)) \leq 2\epsilon + O(\delta)$.  Let us rule out the second case. 

If $d(x,Fix_{M_b}(b)) \leq 2\epsilon + O(\delta)$ the entire segment $[x,bx]$ lies in $Fix_{M'_b}(b)$, where $M'_b=M_b+2\epsilon + O(\delta) \leq L(b) +3 \epsilon +O(\delta)$. So again if $y$ is any point on $[x,z]$ such that $d(x,y) \leq d(x,Fix_{M_a}(a)) + \epsilon$, then not only $d(ay,y) \leq d(ax,x) -2d(x,y) + \epsilon + O(\delta)$ but also $d(by,y) \leq M'_b \leq L(b) + 3 \epsilon + O(\delta) $. Now we may choose our numerical constant $K>0$ so large that the implied constant in the big O in the two last inequalities is at most $K/2$ say. Then $d(by,y) \leq L(\{a,b\})  - K\delta/2$ and $d(ay,y) \leq L(\{a,b\}) -2d(x,y) + 2\epsilon + K\delta/2$ so 
$$L(\{a,b\},y)   \leq L(\{a,b\}) +  \max\{- K\delta/2, 2 \epsilon + K \delta/2 - 2d(x,y)\},$$
and we conclude that 
$$2d(x,y) \leq 2 \epsilon +K\delta/2.$$
Again choosing $y$ as far from $x$ as possible, we see that this means that  $d(x,Fix_{M_a}(a)) \leq 2\epsilon + K/\delta/2$. In particular $d(ax,x)$ and $d(bx,x)$ are both at most $\max\{M_a,M_b\} + 4\epsilon + K \delta \leq \max\{L(a),L(b)\} + 5\epsilon + 2K\delta$. The last quantity is $< L(\{a,b\})$  by assumption and this is impossible. This ends the proof of the claim.

We have thus shown that $x$ is at most $O(\delta) + \epsilon$ away from the midpoint between $ax$ and $bx$. The same reasoning applies to the segments between $b^{-1}x$ and $ax$, as well as between $bx$ and $a^{-1}x$. If $m$ is the midpoint of $[b^{-1}x,ax]$ then $(\ref{midpoint})$ of Lemma \ref{displacement.versus.distance} shows that 
$$L(ab) \leq d(abm,m) \leq L(ab) +\epsilon+ O(\delta),$$
so $d(abx,x) \leq L(ab) + O(\delta) + 2\epsilon$. On the other hand
\begin{eqnarray*} 2d(x,bx) &\leq& d(bx,a^{-1}x) + O(\delta) + 2\epsilon \\ &\leq& d(abx,x) + O(\delta) + 2\epsilon \\ &\leq& L(ab) + O(\delta) + 4\epsilon.\end{eqnarray*}
And similarly $ 2d(x,bx) \leq L(ab) + O(\delta) + 4\epsilon$, so that $$L(\{a,b\}) \leq L(\{a,b\},x) \leq  \frac{1}{2}L(ab) + O(\delta) + 2\epsilon.$$ Letting $\epsilon$ tend to zero, the proof is complete.

\end{proof}



\begin{proof}[Proof of Theorem \ref{formula.S.hyp}] The right hand side inequality is obvious. The left hand side follows easily from Proposition \ref{formula.pair.hyp} combined with Helly's theorem for hyperbolic spaces, namely Theorem \ref{helly.hyperbolic}. Indeed pick $\epsilon>0$ and  set $$M=\max_{a,b \in S} \{L(a), L(ab)/2\} + K\delta + \epsilon,$$ where $K$ is as in Proposition \ref{formula.pair.hyp}.  This proposition tells us that  $Fix_M(a)$ and $Fix_M(b)$ intersect non trivially for every pair of isometries $a,b \in S$. From our Helly-type theorem for hyperbolic spaces Theorem \ref{helly.hyperbolic} we conclude that the $28\delta$-neighborhoods of all $Fix_M(a)$, $a \in S$, intersect non-trivially. Each such neighborhood is clearly contained in $Fix_{M+56\delta}(a)$. We have established that $L(S)\leq M + 56\delta$. Letting $\epsilon$ tend to $0$, this ends the proof.
\end{proof}

\begin{proof}[Proof of Corollary \ref{growth.S.hyp}] Recall from Claim 7 equation $(\ref{claim7})$ in the proof of Lemma \ref{gen-ineq} in Section \ref{gen} that
$$n\lambda_2(S) \leq \lambda_2(S^n).$$  Now by Theorem  \ref{formula.S.hyp} we get the desired inequalities:
$$n(L(S) - K\delta ) \leq n \lambda_2(S) \leq \lambda_2(S^n) \leq L(S^n) \leq nL(S).$$
If we now apply Theorem  \ref{formula.S.hyp} to $S^n$ we get
$$ L(S^n) - K \delta \leq \lambda_2(S^n)$$
Letting $n$ tend to infinity we also get $\lambda_\infty(S)=\ell(S)$.
\end{proof}

We give a consequence of Proposition \ref{growth.S.hyp} anc Lemma \ref{cat0-growth}.

\begin{corollary}\label{growth.S.cat0}
Let $X$ be CAT(0) and $\delta$-hyperbolic. Then for any $\epsilon >0$, 
there exists an integer $N=N(\delta, \epsilon)>0$ with the following property.
For any finite set $S \subset Isom(X)$ with $1 \in S$, one of the following holds:
\begin{enumerate}
\item $L(S)<\epsilon$,
\item there is $g \in S^N$ such that $\ell(g) \geq L(S)$.
\end{enumerate}
\end{corollary}

\proof
By Theorem \ref{formula.S.hyp},
$L(S) \le \lambda_2(S) + K \delta$.
Also, for any $n>0$, 
$L(S^n) \le \lambda_2(S^n) + K \delta$.
But since $X$ is CAT(0), we also have
$\frac{\sqrt n}{2} L(S) \le L(S^n)$ by Lemma \ref{cat0-growth}.
Combining those two inequalities, and taking $n$ large enough, 
we obtain the desired conclusion.
\qed

\section{Helly type theorem for hyperbolic spaces}\label{helly-sec}

It is well-known and easy to prove that any family of convex subsets of a tree with non-empty pairwise intersection must have a non-empty intersection.  In this section we prove the following extension of this fact to hyperbolic spaces:

\begin{theorem}[Helly for hyperbolic spaces]\label{helly.hyperbolic} Let $(X,d)$ be a $\delta$-hyperbolic geodesic metric space. Let $ (C_i)_i$ be a family of convex subsets of $X$ such that $C_i \cap C_j \neq \varnothing$ for all $i,j$. Then the intersection of all $(C_i)_{28\delta}$ is non-empty.
\end{theorem}

Recall that for a subset $E \subset X$ we denote by $(E)_t$ the $t$-neighborhood of $E$, i.e. 
$$(E)_t := \{ x \in X ; d(x,E) \leq t\}.$$

The proof is insprired by \cite{chepoi-et-al}. We begin with a lemma:

\begin{lemma} Let $C_1$ and $C_2$ be convex subsets of $X$ with non empty intersection. Let $z$ be a point in $X$. Let $x_i$ be a point in $C_i$ such that $d(z,x_i) \leq  d(z,C_i)+ \delta$. Assume that $d(z,x_1) \geq d(z,x_2) - \delta$. Then $d(x_1,C_2) \leq 28 \delta$.
\end{lemma}

\begin{proof} Pick $u \in C_1 \cap C_2$. It is a standard fact about $\delta$-hyperbolic spaces that the distance to a geodesic and the Gromov product of the end points are equal within an error $4 \delta$ (see \cite[Lemma 6.2]{Bowditch}). This means that $(u,x_1)_z \geq d(z,ux_1) - 4\delta \geq d(z,x_1) - 5 \delta$. Unfolding the Gromov product, we get:
$$d(z,u) +10 \delta \geq d(u,x_1) + d(x_1,z)$$
The same holds for $x_2$ in place of $x_1$. This means that the paths $u,x_1,z$ and $u,x_2,z$ are almost geodesic (i.e. $10\delta$-taut in the terminology of \cite[chap. 6]{Bowditch}). Hence they must be very close to each other. Applying \cite[Lemma 6.4]{Bowditch} we see that both paths are within $27\delta$ of each other. In particular $d(x_1,C_2) \leq d(x_1,ux_2) \leq 28 \delta$, which is what we wished for. 
\end{proof}

\begin{proof}[Proof of Theorem \ref{helly.hyperbolic}] Pick any point $z \in X$.  The point of the proof is that the previous lemma identifies one specific point that must be in the intersection. Let $x_i \in C_i$ as in the previous lemma. Without loss of generality, we may assume that $d(x_1,z) \geq d(x_i,z) - \delta$ for all indices $i$. Then we may apply the lemma to all pairs $C_1$ and $C_i$. And conclude that $x_1$ belongs to the desired intersection. This ends the proof.
\end{proof}

\section{Symmetric spaces of non-compact type}\label{sym-spaces}

In this section we prove the Berger-Wang identity for symmetric spaces of non-compact type, as well as the Bochi-type inequality stated in the introduction. In particular we establish Proposition \ref{ell-L-sym} and Theorem \ref{bw-thm} for these spaces.

So $X$ is assumed to be a symmetric space of non-compact type associated to a real semisimple algebraic group $G$. The space $(X,d)$ is then $CAT(0)$ and the distance is $G$-invariant. We refer to \cite{BH} for background on these spaces and to \cite{parreau,kapovich-leeb} for finer properties. We only recall here the following important example:

\begin{example}\label{ex-sym} Let $P_d$ be the symmetric space associated to $G=\SL_d(\C)$ and $K=SU(d,\C)$, that is $X=P_d=G/K$. Recall the Cartan decomposition $G=KAK$, where $A$ is the Lie subgroup of diagonal matrices with positive real entries. The distance on $P_d$ is left $G$-invariant and is given by the following simple formula:
\begin{equation}\label{def-ell2}d(gx_0,x_0)=d(ax_0,x_0) = \sqrt{  (\log a_1)^2 + \ldots + (\log a_d)^2},\end{equation} where $g=kak\in KAK$, $x_0$ is the point in $P_d$ fixed by the maximal compact subgroup $K$ and $a=diag(a_1,\ldots,a_d)$. 

This example is important also because every symmetric space of non-compact type arises as a convex (totally geodesic) subspace of some $P_d$. In fact if $M$ is an arbitrary  symmetric space of non-compact type, then $Isom(M)$ is a linear semisimple Lie group with finitely many connected components and thus embeds as a closed subgroup of some $\SL_d(\C)$ for some $d$. By the Karpelevich-Mostow theorem the  connected component of the identity  $Isom(M)^0$, which is a semisimple Lie group, admits a convex totally geodesic orbit within the symmetric space $P_d$, which is isometric to $M$. In particular every isometry of $M$ extends to an isometry of $P_d$.
\end{example}

We also recall here that, since $(X,d)$ is a locally compact CAT($0$) space we know that $g \in Isom(X)$ satisfies $L(g)=0$ if and only if $\ell(g)=0$ and if and only if $g$ fixes a point in $\overline{X} = X \cup \partial X$ (Corollary \ref{bdy-fixed}). This is also equivalent to saying that the eigenvalues of $g$ under some (or any) faithful linear representation of $Isom(X)$ have modulus $1$.

To begin with, we recall the following fact:

\begin{proposition}[subgroups of elliptics]\label{ell-sub} Suppose $X$ is a symmetric space of non-compact type and $S \subset Isom(X)$ a finite set of isometries. Then the following are equivalent:
\begin{enumerate}
\item $\lambda_\infty(S)=0$,
\item $L(S)=0$,
\end{enumerate} 
Moreover in this case $S$ fixes a point in $\overline{X}$.
\end{proposition}

Note the contrast with Euclidean spaces, where the analogous statement fails (see Section \ref{euclidean}). This proposition, or at least some variant, is likely well-known, but lacking a reference we will include a proof for the reader's convenience.

\begin{remark} Observe that $L(S)=L(S \cup S^{-1})$. So the conditions of the previous proposition are also equivalent to $\lambda_\infty(S \cup S^{-1})=0$, which amounts to say that $L(g)=0$ for every $g$ in the subgroup generated by $S$.
\end{remark}

\begin{remark} In Section \ref{escape} we will show that we cannot replace $\lambda_\infty(S)$ by $\lambda_k(S)$ for some finite $k$ independent of $S$ in the above proposition. 
\end{remark}

We now pass to the Berger-Wang identity. 

\begin{theorem}[geometric Berger-Wang]\label{bochi-sym} Let $X$ be a symmetric space of non-compact type. For every $\epsilon>0$ there is $k=k(X,\epsilon) \in \N$ such that for every finite set $S \subset Isom(X)$ one has:
$$\lambda_k(S) \geq (1-\epsilon)\ell(S) - \epsilon.$$
In particular the Berger-Wang identity holds, i.e. $\lambda_\infty(S)=\ell(S)$.
\end{theorem}

We also relate the joint minimal displacement to the asymptotic minimal displacement and find the following general inequality:

\begin{proposition}\label{comp-comp} Let $X$ be a symmetric  space of non-compact type viewed as a convex subspace of $P_d := \SL_d(\C)/\SU_d(\C)$, and let $S \subset Isom(X)$ be a finite subset, then : 
\begin{equation} \frac{1}{\sqrt{d}} L(S) - \log\sqrt{d} \leq \ell(S) \leq L(S).
\end{equation}
\end{proposition}

Similarly using the Bochi inequality (see Proposition \ref{bochi-original} below) we will show:

\begin{proposition}\label{compa-compa}Let $X$ be a symmetric  space of non-compact type viewed as a convex subspace of $P_d := \SL_d(\C)/\SU_d(\C)$, and let $S \subset Isom(X)$ be a finite subset, then : 
$$\lambda_{k_0}(S) \geq \frac{1}{\sqrt{d}} L(S) - C,$$ where $C>0$ is a constant depending on $d$ only and $k_0\leq d^2$.
\end{proposition}

\begin{remark}
Bochi's original proof of his inequality (Proposition \ref{bochi-original} below) gave a worse estimate on $k_0$ (exponential in $d$). In \cite[Cor. 4.6]{breuillard-gelander} we gave a different proof of Bochi's inequality with the above $d^2$ bound. Furthermore the dependence of $C$ on $d$ is not effective in Bochi's proof from \cite{bochi}. In \cite{breuillard-radius} we make effective the argument from  \cite[Cor. 4.6]{breuillard-gelander} and give an explicit estimate on the constant. 
\end{remark}

\begin{remark} Since rank one symmetric spaces are Gromov hyperbolic the stronger inequality:
$$\lambda_{k_0}(S) \geq L(S) - C$$ 
holds in these cases even with $k_0=2$ according to Theorem \ref{bochi-hyp}. It is likely that a similar linear lower bound (with no multiplicative constant in front of $L(S)$) holds for general symmetric spaces as well (although not necessarily with $k_0=2$). The best we could do in this direction however is the lower bound from Theorem \ref{bochi-sym}.
\end{remark}

Before passing to the proofs of the above results, we would like to explain the connection with the well-studied notion of \emph{joint spectral radius} of a finite set of matrices.

\begin{remark}[Connection with the joint spectral radius] Recall that if $S \subset M_d(\C)$ is a finite set of matrices, then the joint spectral radius $R(S)$ of $S$ (in the sense of Rota and Strang \cite{rota-strang}) is defined as 
$$R(S) = \lim_{n \to +\infty} \|S^n\|_2^{1/n},$$ where $\|Q\|_2 := \max_{g \in Q} \|g\|_2$ for $Q \subset M_d(\C)$ and $\|g\|_2$ is the operator norm of $g$ acting on Hermitian $\C^d$.
As it turns out, when $S \subset \SL_d(\C)$, we can interpret $R(S)$ as $\exp \ell(S)$ for a suitably defined left invariant distance on the homogeneous space $P^{\infty}_d:=\SL_d(\C)/\SU_d(\C)$. This distance is defined by:
\begin{equation}\label{def-ell-infty}d(gx_0,x_0) = \max_i \{\log a_i\}\end{equation}
 where $g=kak\in KAK$, $x_0$ is the point in $P^{\infty}_d$ fixed by the maximal compact subgroup $K=\SU_d(\C)$ and $a=diag(a_1,\ldots,a_d) \in A$ the subgroup of diagonal matrices with positive real entries. 
Note that even though $ P^{\infty}_d$ and the symmetric space $P_d$ discussed in Example \ref{ex-sym} have the same underlying space $\SL_d(\C)/\SU_d(\C)$, the distance is not the same.

Observe that as we have defined it:
$$d(gx_0,x_0) = \log \|g\|_2,$$
and hence we have:
$$\ell^{P^{\infty}_d}(S) = \log R(S),$$
where $\ell^{P^{\infty}_d}(S)$ is the $\ell(S)$ defined in the introduction for the metric space $X=P^{\infty}_d$.

Similarly the largest modulus $\Lambda(g)$ of an eigenvalue of $g \in \SL_d(\C)$ is the limit $\lim \|g^n\|^{1/n}$, that is $\exp \ell^{P^{\infty}_d}(g)$. Therefore we see that for a finite subset $S$ in $\SL_d(\C)$ the Berger-Wang identity $$R(S)= \limsup_{n\to +\infty} \max\{\Lambda(g), g \in S^n\}^{1/n}$$ proved in \cite{berger-wang} is simply the statement:
$$\ell^{P^{\infty}_d}(S)=\lambda^{P^{\infty}_d}_\infty(S).$$

Comparing $(\ref{def-ell2})$ and $(\ref{def-ell-infty})$ we see that all quantities pertaining to $P_d$ are comparable to the corresponding quantities pertaining to $P_d^{\infty}$. In particular $$\ell^{P_d^\infty}(S) \leq \ell^{P_d}(S) \leq \sqrt{d} \cdot \ell^{P_d^\infty}(S).$$

In \cite{bochi} Bochi gave a different proof of the Berger-Wang identity, which gave much more, namely an eigenvalue lower bound in terms of $R(S)$. He proved:

\begin{proposition}(Bochi inequality \cite[Thm B.]{bochi})\label{bochi-original} There are constants $c=c(d)>0$ and $k_0= k_0(d) \in \N$ such that if $S \subset M_d(\C)$ is a finite set of matrices, then 
$$ c \cdot R(S) \leq  \max_{j \leq k_0} \max_{g \in S^j} \Lambda(g)^{1/j} \leq R(S).$$
\end{proposition}

It is also known \cite{bochi} that 
$$R(S) = \inf_{\|\cdot\|} \|S\|$$
where the infimum is taken over the operator norm associated to an arbitrary choice of (not necessarily hermitian) norm on the real vector space $\C^d$. Since by John's ellipsoid theorem every two operator norms are equivalent up to a factor $\sqrt{2d}$, we have that 
$$R(S) \leq \inf_{x \in \GL_d(\C)} \|xSx^{-1}\|_2 \leq \sqrt{2d} \cdot R(S).$$  Combining this with $(\ref{def-ell2})$ and $(\ref{def-ell-infty})$ we obtain in particular the following inequalities between the joint minimal displacement, the joint asymptotic displacement and the joint spectral radius of a finite subset $S$ of $\SL_d(\C)$

\begin{equation}\label{comp1}\log R(S) \leq L^{P_d}(S) \leq \sqrt{d} \log (\sqrt{2d} R(S)) \end{equation}
and, replacing $S$ by $S^n$ and passing to the limit:
\begin{equation}\label{comp2}\log R(S) \leq \ell^{P_d}(S) \leq \sqrt{d} \log (R(S)) \end{equation}

\end{remark}

\bigskip

We now turn to the  proofs of the statements above.

\begin{proof}[Proof of Proposition \ref{ell-sub}]  That $(2)$ implies $(1)$ is clear by  Lemma \ref{disp}. For the implication  $(1)$ implies $(2)$, we first reduce to the case when $X=P_d$ is the symmetric space of Example \ref{ex-sym}.  As described in this example, $X$ embeds as totally geodesic subspace of $P_d$ for some $d$ and $Isom(X)$ embeds in $Isom(P_d)$. Since $P_d$ is CAT($0$) and $X$ is a closed convex subset, the nearest point projection $\pi_X$  from $P_d$ to $X$ is  well-defined and it is a distance non-increasing map \cite[Prop II.2.4]{BH}. Consequently for $x \in P_d$ and $S \subset Isom(X)$ we have $L(S,\pi_X(x)) = L(S,x).$ This means that, when $S \subset Isom(X)$, all quantities $L(S)$, $\ell(S)$, $\lambda_k(S)$ coincide when defined in $X$ or in $P_d$. So now we assume as we may that $X=P_d$.

Recall Burnside's theorem (as in \cite[1.2]{bass} for example) : if $G \leq GL_d(\C)$ is an irreducible subgroup, then there are finitely many $g_1,...,g_{d^2} \in G$ and $t_1,...,t_{d^2} \in M_d(\C)$ such that $x=\sum \tr(g_ix)t_i$ for every $x \in M_d(\C)$. So if $g$ has eigenvalues of modulus $1$ for each $g \in G$, then $G$ is contained in a bounded part of $M_d(\C)$. Note that Burnside's theorem holds just as well if $G$ is a semi-group containing $1$ in place of a subgroup: indeed the linear span in $M_d(\C)$ of a semigroup in $\GL_d(\C)$ is invariant under the subgroup it generates so it must be all of $M_d(\C)$ provided the subgroup acts irreducibly.

We now show $(1)$ implies $(2)$ for $P_d$. For $g \in Isom(P_d)$ the condition $L(g)=0$ is equivalent to the requirement that all eigenvalues of $g$ have modulus $1$. Taking a composition series for $\langle S \rangle$, Burnside's theorem implies that $\langle S \rangle$ is conjugate to a bloc upper-triangular subgroup of $\GL_d(\C)$, and bounded (i.e. relatively compact) in each bloc. Conjugating by an appropriate diagonal element, we can conjugate $S$ into any neighborhood of the (compact) bloc diagonal part. This means that choosing $y \in P_d$ we can make  $L(S,y)$ arbitrarily small. Therefore $L(S)=0$ as desired.
\end{proof}

\begin{proof}[Proof of Theorem \ref{bochi-sym}] As in the proof of Proposition \ref{ell-sub} we may assume that $X$ is the symmetric space $P_d$ of Example \ref{ex-sym}. For $g \in \SL_d(\C)$ with eigenvalues $a_1,...,a_d$ ordered in such a way that $|a_1| \geq \ldots \geq |a_d|$, we set 
$$j(g)=(\log |a_1| , \ldots , \log |a_d| )$$
and we observe that
\begin{equation}\label{length-def}\ell(g)= \sqrt{(\log |a_1|)^2 + \ldots + (\log |a_d|)^2} = \|j(g)\|_2.\end{equation}
If $g=k_1ak_2$, with $k_1,k_2 \in K=\SU_d(\C)$ and $a$ diagonal, then we set $\kappa(g)=j(a)$ and note that $\|g\|=\|\kappa(g)\|_2$.

We will use Bochi's inequality in various irreducible representations of $\SL_d(\C)$. in the spirit of Kostant's paper \cite{kostant-asens}.

Recall that irreducible linear representations of $\SL_d(\C)$ are parametrized by a highest weight $\bn:=(n_1,\ldots,n_d)$, where the $n_i$'s are integers satisfying $n_1\geq \ldots \geq n_d \geq 0$. We denote the associated representation by $(\pi_{\bn}, V_{\bn})$. We may find a hermitian scalar product and an orthonormal basis of the representation space $V_{\n}$ of $\pi_{\bn}$ in which $\pi_{\bn}(K)$ is unitary and $\pi_{\bn}(a)$ diagonal for every diagonal $a \in \SL_d(\C)$. Then $\pi_{\bn}(a)$  has maximal eigenvalue equal to its operator norm and equal to $$\langle \bn , j(a) \rangle := n_1 \log |a_1| + \ldots + n_d \log |a_d|.$$ Furthermore:
\begin{eqnarray}
\log \|\pi_{\bn}(g)\| = \langle \bn , \kappa(g) \rangle, \\
\log \Lambda(\pi_{\bn}(g)) = \langle \bn , j(g) \rangle.
\end{eqnarray}

Fix $\epsilon>0$. Note that there are finitely many integer vectors $\bn_1,\ldots,\bn_m$ in the quadrant $\mathcal{Q}=\{ x \in \R^d ; x_1\geq\ldots\geq x_d \geq 0\}$ such that for all $x \in \mathcal{Q}$ we have:
\begin{equation}\label{approx}\sup_{1\leq i\leq m} \langle \frac{\bn_i}{\|\bn_i\|} , x \rangle \geq \|x\|(1-\epsilon).\end{equation}
Indeed just pick rational points forming an $\epsilon$-net near the unit sphere in $\mathcal{Q}$. 

Let $\ell^{\bn}(S)$ be the asymptotic joint displacement of $\pi_{\bn}(S)$ with respect to uniform norm on $V_{\bn}$. Namely
$$\ell^{\bn}(S) := \lim_k \frac{1}{k} \max_{g \in S^k} \log \|\pi_\bn (g)\||$$
and similarly
$$\lambda_k^{\bn}(S) := \max_{j \leq k} \frac{1}{j}  \max_{g \in S^k} \log \Lambda(\pi_\bn (g)).$$

Note that 

\begin{equation}\label{pas}\lambda_k^{\bn}(S)   \leq \lambda_k(S) \|\bn\|\end{equation}

because $$\log \Lambda (\pi_\bn(g)) = \langle \bn , j(g) \rangle \leq \|j(g)\|_2  \|\bn\| =\ell(g) \|\bn\|.$$

Now using $(\ref{approx})$ we may write for each $k\geq 1$
\begin{eqnarray}
\ell(S) \leq  \frac{1}{k} \max_{g \in S^k} \|\kappa(g)\|  \leq \frac{1}{k} \frac{1}{(1-\epsilon)} \max_{g \in S^k} \sup_{1\leq i \leq m}  \langle \kappa(g), \frac{\bn_i}{\|\bn_i\|} \rangle  \\ \leq \frac{1}{(1-\epsilon)}  \sup_{1\leq i \leq m}   \frac{1}{\|\bn_i\|}  \frac{1}{k} \max_{g \in S^k} \log \| \pi_{\bn_i}(g)\|.
\end{eqnarray}

Passing to the limit as $k$ tends to infinity we get:

\begin{equation}\label{las}
\ell(S) \leq \frac{1}{(1-\epsilon)}  \sup_{1\leq i \leq m}   \frac{\ell^{\bn_i}(S)}{\|\bn_i\|}. 
\end{equation}

Then the Bochi inequality for the joint spectral radius (Proposition \ref{bochi-original}) implies that for each $\bn$ there is an integer $k_\bn$ and a positive constant $C_\bn$ such that for all $S$,
$$ \ell^{\bn}(S) \leq  \lambda_{k_\bn}^{\bn}(S) + C_\bn.$$

Setting $K=\max\{k_{\bn_i} , i=1,\ldots,m\}$ and $C=\max\{ C_{\bn_i} / \|n_i\| ; i = 1,\ldots, m\}$ we get from $(\ref{las})$ and $(\ref{pas})$

$$(1-\epsilon) \ell(S) \leq \sup_{1\leq i \leq m}   \frac{\lambda_{K}^{\bn_i}(S)} {\|\bn_i\|}   + C \leq \lambda_K(S) + C.$$
Finally we may choose the smallest integer $n$ such that $C/n < \epsilon$ and, changing $S$ into $S^n$, from Claim 7 equation $(\ref{claim7})$ in the proof of Lemma \ref{gen-ineq} we obtain:
$$(1-\epsilon) \ell(S) \leq  \lambda_{Kn}(S) + \epsilon.$$

\end{proof}



\begin{proof}[Proof of Proposition \ref{comp-comp}] This is just the combination of $(\ref{comp1})$ and $(\ref{comp2})$ given that, as argued in the proof of Proposition \ref{ell-sub} $L(S)$ and $\ell(S)$ defined in $X$ coincide with their counterpart in $P_d$.
\end{proof}

\begin{proof}[Proof of Proposition \ref{compa-compa}] This follows from Proposition \ref{bochi-original} and $(\ref{comp2})$ after we note that $\log \Lambda(g) \leq \ell(g)$ for $g \in \SL(\C)$. The bound $k_0 \leq d^2$  follows from the different proof of Bochi's inequality given in \cite[Cor. 4.6]{breuillard-gelander}.
\end{proof}

We end this section by recording two consequences of the above analysis (compare with Corollary \ref{growth.S.cat0}).

\begin{corollary}[Escaping elliptic elements in symmetric spaces]\label{cor-sym} Let $X$ be a symmetric space of non-compact type and $\epsilon>0$, then there is $N(\dim X,\epsilon)>0$ such that for every finite subset $S \subset Isom(X)$ with $1 \in S$  one of the following holds:
\begin{enumerate}
\item $L(S) < \epsilon,$
\item there is $g \in S^{N}$ such that $L(g) \geq L(S)$.
\end{enumerate}
\end{corollary}

\begin{proof} Since $X$ is $CAT(0)$,  Theorem \ref{disp} shows that $L(S^n) \geq \frac{\sqrt{n}}{2} L(S)$ for every integer $n$. In particular by Proposition  \ref{compa-compa}, if $L(S) \geq \epsilon$ and $\sqrt{N} \gg 2\sqrt{d}(C/\epsilon+1)$, then there is $g \in S^N$ such that $L(g) \geq L(S)$ as desired.
\end{proof}

Recall that for every symmetric space $X$, there exists a positive constant $\epsilon>0$, the Margulis constant of $X$, such that if $S$ is a finite set of isometries of $X$ generating a discrete subgroup of $Isom(X)$ and if $L(S)<\epsilon$, then the subgroup generated by $S$ is virtually nilpotent (see \cite{raghunathan}, \cite{burago-zahlgaller}). Hence we also get:

\begin{corollary}\label{cor-margulis} Let $X$ be a symmetric space of non-compact type, then there is $N>0$ such that for every finite subset $S \subset Isom(X)$ containing $1$ and generating a discrete subgroup, which is not virtually nilpotent, there is $g \in S^{N}$ such that $L(g) \geq L(S)$.
\end{corollary}

\begin{proof} Apply Corollary \ref{cor-sym} with $\epsilon$ being the Margulis constant of $X$.
\end{proof}





\begin{remark} The best $N=N(\epsilon,d)$ for which Corollary \ref{cor-sym} holds tends to infinity as $\epsilon$ goes to $0$. For the same reason the constant $C$ in Proposition \ref{compa-compa} cannot be taken to be $0$ even at the cost of decreasing the multiplicative constant in front of $L(S)$. An example showing this in $SL(2,\R)$ is given below in Section \ref{SL2}.
\end{remark}


We can now give the 

\begin{proof}[Proof of Proposition \ref{ell-L-sym}] This follows from the combination of Proposition \ref{comp-comp} and Lemma \ref{cat0-growth}. We embed $X$ as a closed convex subspace of the symmetric space $P_d$ from Example \ref{ex-sym} for some $d$. Let $n$ be the smallest positive integer such that $L(S) \sqrt{n}/2 \geq 2 \sqrt{d}\log d$. If $n=1$, then $L(S)  \geq 4 \sqrt{d} \log \sqrt{d}$ and thus the inequality in Proposition \ref{comp-comp} yields $\ell(S) \geq \frac{3}{4} L(S)/\sqrt{d}$.  When $n \geq 2$,  Lemma \ref{cat0-growth} and Proposition \ref{comp-comp} imply
$$n \ell(S) \geq \frac{1}{\sqrt{d}} L(S^n)  - \log \sqrt{d} \geq \frac{\sqrt{n}}{2\sqrt{d}} L(S)  - \log \sqrt{d} \geq \frac{\sqrt{n}}{4\sqrt{d}} L(S)$$
On the other hand if $n \geq 2$, then the minimality of $n$ implies that $L(S)\sqrt{n}\leq 16 \sqrt{d} \log \sqrt{d}$ and the result follows.
\end{proof}

As shown in \cite[Corollary 4.6]{breuillard-gelander}  Proposition \ref{bochi-original} also holds for Bruhat-Tits buildings. In fact it was observed and shown later in \cite[Lemma 2.1]{breuillard-heightgap} that in this case the constant $c$ can be taken to be $1$, that is the joint spectral radius is equal to the renormalized maximal eigenvalue,  namely:
$$ R(S)=   \max_{j \leq k_0} \max_{g \in S^j} \Lambda(g)^{1/j},$$
if  $S \subset M_d(k)$ is a finite set of matrices and $k$ a non-archimedean local field.  The same argument as in the proof of Proposition \ref{compa-compa} readily implies:

\begin{lemma}[Bochi inequality for Bruhat-Tits buildings]\label{escape-BT} Let $X$ be a Bruhat-Tits building associated to a non-archimedean local field $k$ and a semisimple algebraic group $\G$ of dimension $d$. Then for every finite subset $S \subset \G(k)$ containing $1$ there is $g \in S^{O(d^2)}$ such that
$$L(g) \geq L(S)/\sqrt{d}.$$
\end{lemma}

Similarl the argument in the proof of Theorem \ref{bochi-sym} shows that $\ell(S)=\lambda_\infty(S)$ if $S \subset \G(k)$. An immediate consequence is also the following: 

\begin{theorem}[Escaping elliptic elements in Bruhat-Tits buildings]\label{cor-BT} Let $X$ as in Theorem \ref{escape-BT}. There is $N=N(d)\in \N$, such that for every finite subset $S \subset \G(k)$ containing $1$ we have:
\begin{enumerate}
\item either $L(S)=0$ and $S$ fixes a point in $X$,
\item or $L(S) >0 $ and there is $g \in S^{N}$ such that $L(g) \geq L(S)$.
\end{enumerate}
\end{theorem}

Compare with Corollary \ref{cor-sym} and note the absence of $\epsilon$. This result can be seen as a quantitative version of the fact, proved in \cite{parreau}, that if a subgroup of $G(k)$ is made entirely of elliptic elements, then it must fix a point in $X$. In particular we can always escape from elliptic elements in bounded time. By contrast, we will show in Section \ref{escape} that this property fails in certain symmetric spaces of non-compact type.

Even though results such as Theorems \ref{cor-sym} and \ref{cor-BT} fail for general $CAT(0)$ spaces (as they fail already for Euclidean spaces $X=\R^d$) it is worth investigating for which classes of $CAT(0)$ spaces they hold. For example:

\bigskip

\noindent  {\bf Question:} Does  Theorem  \ref{cor-BT} hold for the isometry group of an arbitrary affine building ? does it hold for isometries of a finite dimensional $CAT(0)$ cube complex ?

We note that \cite{kar-sageev} answers this question positively for $CAT(0)$ square complexes.

\bigskip

\bigskip





\section{Almost elliptics in $PSL_2(\R)$ with large displacement}\label{SL2}

The purpose of this section is to prove Lemma \ref{so41} from the introduction, which gives a simple example showing that Theorems \ref{bochi-hyp} is best possible. More generally we will show:

\begin{proposition}\label{almostell} Let $X$ be a symmetric space of non-compact type. Then there is no $N=N(X) >0$ such that for every finite set $S \subset Isom(X)$ with $1 \in S=S^{-1}$, there is $g \in S^N$ such that $L(g) \geq L(S)$. 
\end{proposition}

This shows that the additive term $-K\delta$ is necessary (i.e. that we cannot take $K=0$) in Theorem  \ref{bochi-hyp}, and the $N(d,\epsilon)$ must tend to infinity as $\epsilon \to 0$ in Corollary \ref{cor-sym}. Note that this is in stark contrast to the non-archimedean case (see Theorem \ref{cor-BT}).

The proof is based on the fact that zooming in near a point in the hyperbolic plane  $X=\mathbb{H}^2$, the metric becomes almost Euclidean, while on the Euclidean plane we can explicitely construct a set $S$ with the desired properties. We begin by giving this construction:

\begin{example}\label{plane}
Let $\R^2$ be the Euclidean plane. Subgroups of $Isom(\R^2)$ which consist only of elliptic elements must have a global fixed point in $\R^2$. Indeed the commutator of two non trivial rotations with different fixed points is a non-trivial translation.

However the following is an example showing that we may have $L(S)=1$, while no translation with significant translation length can be found in $S^N$ for any fixed $N$.

Take small numbers $x_1,x_2>0$ and set $\theta_{i}\in (0,\pi)$ so that $2\sin(\theta_{i}/2)=x_i$. Assume that $\theta_1$ and $\theta_2$ are independent (i.e. that $1$, $\frac{\theta_1}{2\pi}$ and $\frac{\theta_2}{2\pi}$ are $\Q$-linearly independent). Let $S:=\{1,R_1^{\pm 1},R_2^{\pm 1}\}$, where $R_1$ is the rotation around the point $p_1:=(-x_1^{-1},0)$ and angle $\theta_{1}$, and $R_2$ the rotation around the point $p_2:=(x_2^{-1},0)$ and angle $\theta_2$.

Note that due to the independence assumption on $\theta_1$ and $\theta_2$ the linear parts of $R_1$ and $R_2$ generate a subgroup of $\SO(2,\R)$ which is free abelian of rank $2$. Consequently any word in $R_1$ and $R_2$ which gives a translation in $Isom(\R^2)$ must belong to the commutator subgroup of the free group.

Moreover $L(S)=1$ because $Fix_1(R_1)=\{p; |R_1p-p| \leq 1\}$ and $Fix_2(R_2)=\{p; |R_2p-p| \leq 1\}$ are the discs with radius $x_1^{-1}$ and $x_2^{-1}$ respectively centered around $p_1$ and $p_2$ respectively. They intersect at $(0,0)$.

However, any word $w=w(R_1,R_2)$ in $R_1$ and $R_2$ which is a translation in $Isom(\R^2)$ belongs to the commutator subgroup of the free group, and hence the powers of $R_1$ in this word sum up to $0$ and so do the powers of $R_2$. As a result, for every $\epsilon>0$ and every given $N \in \N$, one can choose small but positive $x_1$ and $x_2$ so that any such $w$ of length at most $N$ maps the origin $(0,0)$ at a distance at most $\epsilon$ from itself. Hence $L(w) \leq \epsilon$ while $L(S)=1$.
\end{example}

\begin{proof}[Proof of Proposition \ref{almostell}] It is enough to prove this in the case when $X=\mathbb{H}^2$ is the hyperbolic plane, because $\mathbb{H}^2$ always embeds as a closed convex subspace of $X$ stabilized by a copy of $PSL(2,\R)$ in $Isom(X)$.  

Suppose, by way of contradiction, that there is $N \in \N$ such that $max_{S^N} L(g) \geq L(S)$ for every finite subset $S \subset Isom(\mathbb{H}^2)$ containing $1$.

 Let $x_1,x_2$ be two small positive numbers to be determined later. On the hyperbolic plane $X=\mathbb{H}^2$ consider a base point $x_0$ and a geodesic through $x_0$. For every $\epsilon>0$ let $p_1(\epsilon)$ and $p_2(\epsilon)$ be points on this geodesic on opposite sides of the base point $x_0$ such that $d(x_0,p_i(\epsilon))=\epsilon x_i^{-1}$. Let $R_i(\epsilon)$ be the hyperbolic rotation fixing $p_i(\epsilon)$ and of angle $\theta_i(\epsilon) \in (0,\pi)$ defined in such a way that $d(R_i(\epsilon)x_0,x_0)=\epsilon$. 

Also note that the sets $\{x \in \mathbb{H}^2 ; d(R_i(\epsilon)x,x) \leq \epsilon \}$ are two hyperbolic disc centered at $p_i(\epsilon)$ that intersect only at the base point $x_0$.  It follows that 
\begin{equation}\label{Leps} L(S_\epsilon) = L(S_\epsilon,x_0) = \epsilon.\end{equation}

Consider the renormalized pointed metric space $(X_\epsilon,d_\epsilon,x_0):=(X,d/\epsilon,x_0)$. In the Gromov-Hausdorff topology for pointed metric spaces, this family of metric spaces converges, as $\epsilon$ tends to $0$, to the Euclidean plane $(\R^2,0)$ with its standard Euclidean metric. The points $p_1(\epsilon)$ and  $p_2(\epsilon)$ converge to two points $p_1$ and $p_2$, which, after choosing coordinates, can be assumed to be $(-x_1^{ -1}, 0)$ and $(x_2^{-1},0)$. Moreover the hyperbolic rotations $R_i(\epsilon)$ converge to their Euclidean counterpart $R_i$ based at $p_i$ with angle $\theta_i$ defined by $||R_i(0,0)||=1$, i.e. $x_i^{-1}=2\sin(\theta_i/2)$ as in Example \ref{plane}. Consequently any word of length at most $N$ in $S_\epsilon:=\{1,R_1(\epsilon)^{\pm 1}, R_2(\epsilon)^{\pm 1}\}$ converges to the isometry of $\R^2$ given by the same word with letters $R_1$ and $R_2$. In particular:
$$\lim_{\epsilon \to 0} w(R_1(\epsilon), R_2(\epsilon))x_0 = w(R_1(\epsilon), R_2(\epsilon)) (0,0),$$
and 
\begin{equation}\label{li}\lim_{\epsilon \to 0} \frac{1}{\epsilon} L(w(R_1(\epsilon), R_2(\epsilon)),x_0) =  L(w(R_1, R_2),(0,0)),\end{equation}
where $L(g,x)$ is the displacement of $g$ at $x$ in $\mathbb{H}^2.$

Now by our assumption for each $\epsilon$ there is a word $w=w_\epsilon$ of length at most $N$ such that 
\begin{equation}\label{globmin} L(w(R_1(\epsilon), R_2(\epsilon)) \geq L(S_\epsilon) = \epsilon.\end{equation}
There are boundedy many words of length at most $N$, so letting $\epsilon$ tend to $0$ along a certain sequence only we may assume that $w$ is independent of $\epsilon$. Then there are two options according as $w(R_1,R_2)$ is a translation or a rotation. Suppose first that $w$ is a translation. We have:
$$\frac{1}{\epsilon}L(w(R_1(\epsilon), R_2(\epsilon), x_0) \geq \frac{1}{\epsilon}L(w(R_1(\epsilon), R_2(\epsilon))) \geq 1,$$ while by $(\ref{li})$ the left handside converges (as $\epsilon$ goes to zero along the subsequence) to $L(w(R_1,R_2),(0,0))$. However, $w(R_1,R_2)$ being a translation means that the word $w$ belongs to the commutator of the free group and the powers of $R_1$  sum to $0$ as do the powers of $R_2$ (due to the independence of $\theta_1$ and $\theta_2$ as in Example \ref{plane}). This implies that $L(w(R_1,R_2),(0,0)) < 1/100$ provided $x_1$ and $x_2$ are chosen larger than some absolute constant. This is a contradiction.

We are left with the case when $w(R_1,R_2)$ is a rotation, with center say $c \in \R^2$. Since $(X,d/\epsilon,x_0)$ Gromov-Hausdorff converges to Euclidean $\R^2$, there must be some point $c(\epsilon) \in \mathbb{H}^2$ which converges to $c$. This also means that 
$$ \lim_{\epsilon \to 0} \frac{1}{\epsilon}d(w(R_1(\epsilon), R_2(\epsilon)) c(\epsilon),c(\epsilon)) = d(w(R_1, R_2)c,c)=0.$$ 
In particular:
$$ L(w(R_1(\epsilon), R_2(\epsilon)) = o(\epsilon),$$
which is in contradiction with $(\ref{globmin})$. This ends the proof.
\end{proof}

\section{Euclidean spaces and linear escape of cocycles}\label{euclidean}

In this section we assume that the metric space $X$ is the Euclidean space $\R^d$. We will prove the results stated in Section \ref{genCat0} regarding isometric actions on $\R^d$. In particular we will show an example of a finitely generated group of $Isom(\R^d)$ without global fixed point all of whose elements are elliptic. And we will show that every affine isometric action on $\R^d$ without global fixed point as positive drift, i.e. that cocycles that are not co-boundaries have linear rate of escape.

We begin with the former.

\begin{example}[A subgroup of rotations without global fixed point] \label{exBass}
This is an example of a subgroup of $Isom(\R^{2n})$ for each $n \geq 2$ without global fixed point in $\R^{2n}$ all of whose elements are elliptic. A similar example for $n=2$ for isometries of $\R^4 \simeq \C^2$ with linear part in $\SU_2(\C)$ is due to Bass, answering a question of Kaplansky (\cite[counterexample 1.10]{bass}).

Consider $S=\{1,A^{\pm 1},B^{\pm 1}\}$, where $A$ and $B$ are two rotations fixing different fixed points $p_A \neq p_B$. We may choose the linear parts $R_A$ and $R_B$ in $O(2n)$ of $A$ and $B$ in such a way that the group they generate is a free group which does not have $1$ as an eigenvalue. This fact follows easily from Borel's theorem on the dominance of words maps \cite{borel, larsen} in simple algebraic groups (here $\SO(2n)$) and from the fact that in even dimension a generic rotation does not have $1$ as an eigenvalue. To see this consider that, due to the dominance of word maps, the preimage in $\SO(2n) \times \SO(2n)$ of the elements having $1$ as an eigenvalue via any word map is an algebraic subvariety of positive co-dimension, hence has empty interior. By Baire's theorem a Baire generic pair will lie outside the union of these subvarieties when the word ranges over all reduced words on two letters in the free group.

A consequence of this property is that every element in the subgroup $\Gamma$ of $Isom(\R^{2n})$ generated by $A$ and $B$ is elliptic (i.e. fixes a point in $\R^{2n}$). Indeed, since $R_A$ and $R_B$ generate a free subgroup, every non trivial element of $\Gamma$ has a non trivial linear part and this linear part does not have $1$ as an eigenvalue. But such isometries must fix a point. So $L(g)=0$ for every $g \in \Gamma$, while clearly $L(S) \geq \frac{1}{2}|p_A - p_B| \min{|\theta|}$, where the minimum is taken over the angles $\theta$ of $R_A$ and $R_B$ (i.e the numbers in $[0,\pi]$ such that $\exp(\pm i\theta)$ are the eigenvalues of $R_A$ and $R_B$).
\end{example}

This example is to be contrasted with Proposition \ref{dim23}, which we restate here:

\begin{proposition}A subgroup $G$ of $Isom(\R^d)$ with $d=2,3$ all of whose elements have a fixed point must have a global fixed point.
\end{proposition}

\begin{proof} Let $G_0$ be the index two subgoup of orientation preserving isometries. When $d=2$, the commutator subgroup $[G_0,G_0]$ is made of translations. Hence $G_0$ is abelian. Two commuting isometries must preserve the fixed point set of each other. It follows that $G_0$ is either trivial or has a unique global fixed point, which must then be fixed by G (since $G_0$ is normal in $G$).

When $d=3$ generic orientation preserving isometries are skew translations and do not have fixed points, in fact isometries with a fixed point form a closed subset of empty interior.  In particular the closure $H$ of $G_0$ is a proper closed subgroup of $Isom(\R^3)^0$. Say $\SO(3,\R)$ denotes the stabilizer of the origin in $\R^3$. If the linear part of $H$ is not all of $\SO(3,\R)$ it must either be finite or conjugate to $\SO(2,\R)$. In the latter case  all elements in $H$ have parallel axes and $H$ must preserve an affine plane (the orthogonal to the axes) and be either trivial or have a unique global fixed point by the $d=2$ case ; in this case $G$ fixes this point. If the linear part of $H$ is finite, then every element of $H$ has bounded order forcing $G_0$ itself (and hence $G$) to be finite (Burnside) and hence to have a global fixed point. Finally if the linear part of $H$ is all of $\SO(3,\R)$, then $H$ must be conjugate to $SO(3,\R)$ (and hence have the origin as its unique global fixed point, which must then be fixed by all of $G$). To see this note first the following simple geometric fact: if $g,h$ are two rotations of angle $\pi$ with disjoint axis, then $gh$ is a skew rotation with axis the common perpendicular to the axes of $g$ and $h$ and translation length twice the distance between the axes. Now pick $h \in H$ a rotation of angle $\pi$ and axis $\Delta$. If $g\Delta$ intersects $\Delta$ for all $g\in H$, then any two lines of the form $g\Delta$, $g\in H$, intersect. This means either that the $H$-orbit of $\Delta$ is made of all lines through a single point (and then this point is the unique global fixed point), or that it lies in an affine plane (the one spanned by any two of the lines) and that $H$ preserves that plane, a fact not compatible with the assumption that its linear part is all of $\SO(3,\R)$.
\end{proof}

Now we show that non trivial cocycles have linear rate of escape.

\begin{proposition}\label{linearescape} For any $d\geq 2$ and $S\subset Isom(\R^d)$ a finite set, then $\ell(S)=0$ if and only if $S$ has a global fixed point in $\R^d$.
\end{proposition}

\begin{proof} First we give a proof of the  easier fact that $L(S)=0$ implies the existence of a global fixed point. Indeed note that the fixed point set $Fix(s)$ of an isometry $s$ of $\R^d$ is an affine subspace and, when $Fix(s)$ is non-empty, the sublevel sets $Fix_t(s)= \{x \in \R^d ; d(x,sx) \leq t \}$ are  convex subsets of $\R^d$ that are products of $Fix(s)$ by the unit ball for a positive definite quadratic form on the orthogonal of $Fix(s)$. In particular there is $c_s>0$ such that $Fix_t(s) \subset (Fix(s))_{c_s t}$ for all $t>0$. Similarly the intersection of all $(Fix(s))_t$ for $s$ varying in $S$ and all $t>0$ coincides with the intersection of all $Fix(s)$, $s\in S$ (note that for any two affine subspaces there is $c>0$ such that the intersection of their $t$-neighorhoods is contained in the $ct$-neighborhood of their intersection). So we obtain that $S$ has a global fixed point.

Now we will show the slighty more delicate fact that the absence of global fixed point implies $\ell(S)>0$. To do this we will first assume that $\Gamma$ is dense in a Lie subgroup of the form $R \ltimes V$, where $R$ is a closed Lie subgroup of $O(V)$ and $V\leq \R^d$ is a non-zero vector subspace. Afterwards we will reduce to this case. 

For $\epsilon>0$ consider a finite covering of the unit sphere in $V$ by Euclidean balls of radius $\epsilon$. In each ball pick an element $\gamma \in \Gamma$ whose translation part $t_\gamma$ belongs to this ball and whose rotation part $r_\gamma$ belongs to the ball of radius $\epsilon$ around the identity in $R$, i.e. $\|r_\gamma - 1\| \leq \epsilon$ for the operator norm on endomorphisms of $V$. This gives us a finite list $\gamma_1,...,\gamma_N$ of elements of $\Gamma$. They all belong to $S^k$ for some integer $k$.

Now we make the following observation. If $x \in V \setminus \{0\}$ and $g=(r_g,t_g) \in Isom(V)$, then 

$$\|gx\|^2 = \|r_g x + t_g \|^2 = \|x\|^2 + \|t_g\|^2 + 2\langle r_g x, t_g\rangle.$$
In particular:
$$\|gx\|^2 \geq   \|x\|^2 + \|t_g\|^2 + 2\langle r_g x, t_g\rangle \geq (\|x\| + \frac{1}{2})^2$$
provided $\frac{1}{2} \leq \|t_g\|$ and $\langle \frac{ r_g x}{\| r_g x\|} , t_g \rangle \geq \frac{1}{2}$. Now if $t_g$ lies in the same $\epsilon$-ball as $x/\|x\|$ and $\|r_g-1\| \leq \epsilon$, then these conditions are satisfied, provided $\epsilon$ is small enough ($\epsilon=1/10$ does it). So we have shown that given any $x \in V \setminus \{0\}$, there is one $\gamma_i \in S^k$ such that 
$$\|\gamma_i x\| \geq \|x\| + \frac{1}{2}.$$
Starting with any point $x_0 \in  V \setminus \{0\}$ this immediately implies that for all $n\geq 1$,
$$L(S^{nk},x_0) \geq \frac{n}{2},$$ and thus $\ell(S) \geq 1/2k >0$ as desired.

We now explain how to reduce to the above special case. Let $H$ be the closure of $\Gamma$ in $Isom(\R^d)$. The Lie subgroup $H$ has the form $R\ltimes(\Delta \oplus V)$, there $V\leq \R^d$ is a vector subspace and $\Delta$ is a discrete subgroup of $\R^d$ contained in the orthogonal of $V$. The isometric $H$-action on $\Delta \oplus V$ factorizes modulo $\Delta$ to an isometric $H$-action on $V$. It is enough to show that $\ell(S)>0$ in the quotient action. The quotient action contains all translations from $V$. So by the analysis above it is enough to know that $V$ is non-zero. Note that, since $\Delta$ is discrete and invariant under conjugation by $R$, there is a finite index subgroup of $R$ commuting with $\Delta$. So if $V$ were trivial, there is a finite index subgroup $\Gamma_0$ of $\Gamma$ such that any element $g$ in this finite index subgroup will satisfy $t_{g^n} = nt_g$. It follows that $\ell(S)>0$ unless $\Gamma_0$  lies in $R$. But then $\Gamma_0$ (and hence $\Gamma$ taking the barycenter of a finite orbit) has a global fixed point.
\end{proof}

\section{Escaping elliptic isometries  of  symmetric spaces}\label{escape}

The goal of this section is mainly to provide certain counter-examples showing that in general one cannot escape elliptic elements in bounded time when taking products of isometries of symmetric spaces. 

In this section we say that an element $g \in Isom(X)$ is a \emph{generalized elliptic element} if $\ell(g)=0$. Since here $X$ will be CAT($0$), this is equivalent to the condition $L(g)=0$ by Proposition \ref{prop.monod}. For symmetric spaces of non-compact type, this condition is equivalent to requiring that all eigenvalues of $g$ (under some or any faithful linear representation of $Isom(X)$) have modulus one.

We recalled in Proposition \ref{ell-sub} the well-known fact that any subgroup of isometries of a symmetric space $X$ of non-compact type, which is made exclusively of generalized elliptic elements, must fix a point in $X$ or its boundary. So if $G$ is a subgroup of $Isom(X)$ generated by a finite set $S$ we have $L(S)=0$ if and only if $L(g)=0$ for all $g \in G$. And if this happens then $G$ fixes a point either in $X$ or on the visual boundary $\partial X$ (see Proposition \ref{bdy-fixed}).

The following natural question then arises:

\bigskip

\noindent {\bf Question:} Do we always escape from elliptics (or generalized elliptics) in bounded time ? namely does there exist $N=N(X) \in \N$ such that for every finite symmetric set $S \subset Isom(X)$ containing $1$, either $L(S)=0$ or there is $g \in S^N$ such that $L(g)>0$ ?

\bigskip

This section is devoted to answering this question. The answer, for hyperbolic spaces, is a little surprising:

\begin{proposition}\label{propques} The above question has a positive answer if $X$ is an $n$-dimensional hyperbolic $\mathbb{H}^n$, when $n=2,3$, but a negative answer when $n \geq 4$.
\end{proposition}

So if $X$ is hyperbolic $n$-space with $n\geq 4$, then there are subgroups of isometries that pretend to be elliptic on a ball of arbitrarily large radius, even though they have no global fixed point in $X$ nor on its boundary (and we will even build Zariski-dense examples when $n$ is even). Further below we will also answer the question completely for arbitrary symmetric spaces  of non-compact type.

Similarly a subgroup of isometries of the Euclidean plane or $3$-space, all of whose elements are rotations, must have a global fixed point (Proposition \ref{dim23} above). But  this is no longer the case in $\R^{n}$ for $n \geq 4$ by Example \ref{exBass}. So the dimension threshold is the same as for hyperbolic spaces. 

Note that Example \ref{exBass} can of course be embedded in a symmetric space $X$ (e.g. by viewing $Isom(\R^{2n})$ as a bloc upper triangular subgroup of $SL_{2n+1}(\R)$). This group will be made of elliptic elements only. Even though it will not fix a point in the symmetric space $X$, it will fix a point on the boundary (see Corollary \ref{bdy-fixed}). In particular $L(S)=0$ in this example (even though $L(S)$ and $\ell(S)$ are strictly positive, when the group is viewed as a subgroup of isometries of Euclidean space).

In the case of $\SL_2(\C)$ a simple matrix computation yields the following avatar of Serre's lemma (Lemma \ref{serre-lemma}):

\begin{proposition} Let $a,b \in \SL_2(\C)$. Assume that $a,b$ and $ab$ have all their eigenvalues of modulus $1$. Then one of the following holds:
\begin{enumerate}
\item $a$ and $b$ can be simultaneously conjugated into $SU_2(\C)$,
\item $a$ and $b$ can be  simultaneously conjugated into the subgroup of upper triangular matrices with eigenvalues of modulus $1$,
\item  $[a,b]:=aba^{-1}b^{-1}$ is loxodromic (i.e. its eigenvalues have modulus $\neq 1$).
\end{enumerate}
\end{proposition}

Since $\PSL_2(\R) = Isom(\mathbb{H}^2)^0$ and  $\PSL_2(\C) = Isom(\mathbb{H}^3)^0$ the first part of Proposition \ref{propques} follows easily. 

\begin{proof} We first recall the following well-known fact:

\noindent {\bf Claim:} For any $a,b \in \SL_2(\C)$ write $x=\frac{1}{2}\tr(a)$, $y=\frac{1}{2}\tr(b)$, $z=\frac{1}{2}\tr(ab)$. Then 
$$\frac{1}{2}\tr([a,b]) = 2(x^2+y^2+z^2) - 4xyz -1.$$
The proof is omitted: it is a simple computation using Cayley-Hamilton. One writes $a^2-2x a +1 =0$ and similarly for $b^2$ and $(ab)^2$ in order to expand any word in $a$ and $b$ as a linear combination of $a$, $b$, $ab$, $aba$ and $bab$, then one takes the trace.

Note that a matrix $u\in \SL_2(\C)$ has eigenvalues of modulus $1$ if and only if $\tr(u)$ is real and belongs to $[-2,2]$.

Suppose first that neither $a$ nor $b$ is diagonalizable. Then after changing $a$ and $b$ into their opposite if necessary, we may assume that both $a$ and $b$ are unipotent, hence have trace $2$, that is $x=y=1$. Then $\frac{1}{2}\tr[a,b] = 2z^2 - 4z +3$. This achieves its minimum at $z=1$ only and the minimum is $1$. This means that $[a,b]$ is loxodromic unless $z=1$. In some basis the matrices of $a$ and $b$ read:
$$ a=\left(
  \begin{array}{cc}
    1& t \\
    0& 1 \\
  \end{array}
\right)  \textnormal{ and }   b=\left(
  \begin{array}{cc}
    \alpha & \beta \\
    \gamma & \delta \\
  \end{array}
\right).
$$
We may assume that $t\neq 0$. We compute easily $\tr(ab)=\tr(b)+t\gamma$. Hence if $[a,b]$ is not loxodromic, we conclude that $t\gamma=0$. It follows that $\gamma=0$, which is case $(2)$ in the proposition.

We may thus assume that $a$ is diagonalizable, and $a \notin \{\pm 1\}$. In some basis the matrices of $a$ and $b$ now read:
$$ a=\left(
  \begin{array}{cc}
    e^{i\theta}& 0 \\
    0&  e^{-i\theta} \\
  \end{array}
\right)  \textnormal{ and }   b=\left(
  \begin{array}{cc}
    \alpha & \beta \\
    \gamma & \delta \\
  \end{array}
\right)$$
with $\cos(2\theta) \neq 1$. Then we compute: $$\frac{1}{2}\tr [a,b] = \cos(2\theta) + \alpha \delta (1-\cos(2\theta)).$$

If $[a,b]$ is not loxodromic, then $\tr[a,b]$ is real in $[-2,2]$. Consequently $\alpha \delta \in \R$. Since $\alpha+\delta = \tr(b)$ is real, we conclude that $\alpha$ and $\delta$ are complex conjugates and $\alpha \delta= |\alpha|^2$. Since $\tr[a,b] \leq 2$, the above formula forces $|\alpha| \leq 1$. If $|\alpha|=1$, then $\alpha \delta =1$ and thus $\beta \delta=0$, which means that we are in case $(2)$ of the proposition.

We may thus assume that $|\alpha|<1$ and also that $\beta \delta \neq 0$. Up to conjugating simultaneously $a$ and $b$ by a diagonal matrix $diag(t,t^{-1})$, for a suitable real $t>0$, we may assume that $\gamma = - \overline{\beta}$. Indeed:
$$t^{-2} \gamma = - \overline{t^2 \beta} \Leftrightarrow t^4=-\gamma/\overline{\beta} = (1-|\alpha|^2)/|\beta|^2.$$
But now $\delta=\overline{\alpha}$ and $\gamma=-\overline{\beta}$. This means that $b \in SU_2(\C)$. So we are in case $(1)$ of the proposition. This ends the proof.
\end{proof}

To prove the second part of Proposition \ref{propques}, we need to construct a counter-example in $Isom(\mathbb{H}^4)$.

\begin{example}[No escape in hyperbolic $4$-space]\label{noescape} Let $A,B$ be the two rotations in $\SO(4)$ constructed in Example \ref{exBass}. They generate a free subgroup of $\SO(4)$ whose non trivial elements do not have $1$ as an eigenvalue. Now observe that the union of all conjugates of $\SO(4)$ inside $\SO(4,1)$  has non-empty interior. One way to see this is to argue that this set is definable in real algebraic geometry and has the same dimension as $\SO(4,1)$ itself, because the absolute (complex) ranks of $\SO(4)$ and of $\SO(4,1)$ coincide (they are equal to 2). Moreover an element of $\SO(4)$ lies in the interior if and only if $1$ is not one of its eigenvalues. This implies that for every integer $N \geq 1$ there are neighborhoods $\mathcal{U}_N(A)$ and $\mathcal{U}_N(B)$  in $\SO(4,1)$ of $A$ and $B$ respectively such that $w(a,b)$ lies in the interior of elliptic elements for every non trivial reduced word $w$ of length at most $N$ in the free group on two letters, and every choice of $a$ in $\mathcal{U}_N(A)$ and $b$ in $\mathcal{U}_N(B)$. But we may choose such a pair $a,b$ so that it generates a Zariski-dense subgroup of $\SO(4,1)$. Indeed the set of pairs generating a non-Zariski dense subgroup of a semisimple algebraic group is a proper closed subvariety of the product (\cite[Thm 3.3]{guralnick} or \cite[Thm 4.1]{BGGT}). A Zariski-dense subgroup cannot be bounded, for otherwise it would be contained in a conjugate of the maximal compact subgroup $\SO(4)$. So by Proposition \ref{ell-sub} it must contain a non-elliptic element.
\end{example}

\begin{remark}[real closed fields] This counter-example shows that Proposition \ref{ell-sub} is very special to the field of real numbers. It does not hold for a general real closed field, or say for an ultrapower $K$ of the reals. Over such fields $K$ the example yields a Zariski-dense subgroup of $\SO(4,1)(K)$ all of whose elements are contained in some conjugate of $\SO(4)(K)$. 
\end{remark}

The above counter-example can be made to work (with the exact same proof) in any symmetric space $X$ for which the elliptic elements of $Isom(X)^0$ have non-empty interior. Namely for each $N \in \N$ one can find pairs generating a Zariski-dense subgroup with the entire $N$-ball of the Cayley graph contained in the set of elliptic elements. Elliptic elements have non-empty interior  if and only if the fundamental rank $\rk_\C G - \rk_\C K$ vanishes, where $G=Isom(X)^0$ and $K$ is a maximal compact subgroup of $G$ (by \cite[Example 3]{papadima} this condition is equivalent to the vanishing of the Euler characteristic of the compact dual of $X$). For example if $X=\mathbb{H}^n$ this happens if and only if $n$ is even. On the other hand, if $X$ is such that elliptic elements have empty interior in $Isom(X)^0$, then it is always possible to escape them, and in fact escape the set of generalized elliptic elements, in bounded time, because they are contained in a proper algebraic subvariety of $Isom(X)^0$. 

\begin{proposition}[general symmetric spaces] Let $X=G/K$ be a symmetric space of non-compact type and $G=Isom(X)^0$.  If $\rk_{\C} G > \rk_{\C} K$, then there is $N=N(X) \in \N$ such that for every finite set $S \subset G$ generating a Zariski-dense subgroup of $G$, $S^N$ contains an element $g$ with an eigenvalue of modulus different from $1$.
\end{proposition}

\begin{proof} An element $g \in G$ is generalized elliptic (i.e. $\ell(g)=0$) if and only if all its eigenvalues under some or any faithful linear representation of $G$ have modulus $1$. This can be read off the characteristic polynomial of $g$ by asking that its real irreducible factors are either $X \pm 1$ or $X^2+2bX +1$, where $b \in [-1,1]$. Therefore generalized elliptic elements form a definable set in real algebraic geometry \cite{bochnak-coste-roy} whose Zariski-closure is a subvariety of positive co-dimension in $G$. Then the lemma follows by ``escape from subvarieties'' (see \cite{emo} \cite[Lemma 3.11]{BGT}). 
\end{proof}

\section{Producing free semi-groups by Ping-Pong}\label{ping-pong}

We prove a proposition concerning  free semi-groups in $\delta$-hyperbolic geometry.
A hyperbolic isometry $g$ of a $\delta$-hyperbolic space $X$
has two fixed points, which are in the ideal boundary of $X$.
We denote the fixed point set by $Fix(g)$. 

\begin{proposition}\label{semi.group}
For $\delta \ge 0$, there is an absolute (numerical) constant $\Delta$ with the 
following property.  
Let $X$ be a $\delta$-hyperbolic space, and $g,h$ isometries of $X$.
Suppose $L(g), L(h) > \Delta \delta$. (Then $g, h$ are hyperbolic isometries.)
Assume $Fix(g) \not= Fix(h)$.

Then  the pair $\{g,h\}$, maybe after taking inverses of 
one or both of those,  generates a free semi-group.
\end{proposition}

This proposition appears  as Proposition 7.1 in \cite{BCG-jems} when  
$X$ is an Hadamard manifold of $K \le -1$.
The case when $X$ is a simplicial tree was treated by Bucher and de la Harpe  in \cite[Lemma]{harpe.bucher}.
In that case, $\delta=0$. 
We give a proof of Proposition \ref{semi.group} since we do not know 
a reference in this context. It is not optimal, but $\Delta \le 10000 $.
We follow the strategy in \cite{harpe.bucher} and use the following 
well-known Ping-Pong lemma. Finding
suitable sets $A, B$ in the $\delta$-hyperbolic setting
is more complicated than the tree case.

\begin{lemma}\label{pingpong.semigroup}
Let $X$ be a set and $g,h$ injective maps from $X$ to $X$.
Suppose there are non-empty subsets $A, B$ in $X$ such that 
$$A \cap B =\emptyset, g(A \cup B) \subset A, h(A\cup B) \subset B.$$
Then $g,h$ generate a free semi-group. 
\end{lemma}

We will use axes of $g,h$ to define $A,B$.
We believe the following result  is well known to specialists, but do not know a reference exactly for this statement, so 
we will give an argument at the end of this section.
\begin{lemma}\label{quasi.axis}
If $g$ is an isometry of a $\delta$-hyperbolic space $X$ with $L(g) > 1000 \delta$, then there is a $g$-invariant, 
piecewise-geodesic, $\gamma$, 
parametrized by the arc-length, such that 
\\
(1) $\gamma$ is a $(\frac{9}{10},24\delta)$-quasi-geodesic.
\\
(2) For any points $x,y \in \gamma$, the Hausdorff
distance between a geodesic $[x,y]$ and the sub-path of $\gamma$ from $x$ to $y$ 
is at most $12 \delta$.
\end{lemma}

We call this quasi-geodesic $\gamma$ an {\it axis} of $g$ in this section. 
The fixed point set $Fix(g)$ consists of the two end-points 
of $\gamma$, which are $\gamma(\pm \infty)=\lim_{t \to \pm \infty}
\gamma(t)$, respectively for $+$ and $-$.

Let $\pi_\gamma:X \to \gamma$ be the nearest points projection.
This is a {\it coarse map} in the sense that the image of each point
is not a point but a set in $X$, but it is uniformly bounded: for every $x \in X$, 
$$ {\rm diam} \, \pi_\gamma(x) \le 30 \delta.$$
\proof
If $\delta=0$ then $X$ is a tree and the conclusion holds, 
so that we assume $\delta >0$.
Let $M(\delta)=12 \delta$ be the constant in Lemma \ref{quasi.axis}.
Let $x \in X$ and $y,z \in \pi_\gamma(x)$.
We show $|y-z| \le 2 M + 6 \delta=30\delta$.
Assume not. Let $w$ be the midpoint of $[y,z]$, 
and take $v$ on $[x,y]$ (or on $[x,z]$) with 
$|w-v| \le \delta$. Also take $w' \in \gamma$ with 
$|w-w'| \le M$.
Then $|w'-v| \le M + \delta$.
On the other hand, 
$|y-v| \ge M + 3 \delta - \delta$, so that 
$|x-v| \le |x-y| - M - 2 \delta$.
We then have $|x-w'| \le |x-v|+|v-w'| \le |x-y|  - \delta < |x-y|$,
so that $y \in \gamma $ is not a nearest point from $x$ ($w' \in \gamma$
is closer), a contradiction. 
\qed

Since $\gamma$ is $g$-invariant, the map $\pi_\gamma$  is $g$-equivariant:
$g\pi_\gamma(x)=\pi_\gamma(gx)$.
For convenience, we choose a point from the set $\pi_\gamma(x)$
and denote this point by $\pi_\gamma(x)$. We arrange $\pi_\gamma$
in this new definition to be $g$-equivariant. 

\bigskip
\noindent
{\it Proof of Proposition \ref{semi.group}}.
We set $\Delta=10000$.
We also assume $\delta >0$ since 
if $\delta=0$ then $X$ is a tree and this case 
is essentially treated in \cite{harpe.bucher}.
We will define subsets $A, B$ in $X$
that satisfies the assumption of Lemma \ref{pingpong.semigroup} for $g,h$, 
or maybe one or both of their inverses. 
Then we are done by the lemma. 

Let $\gamma, \sigma$ be axes of $g,h$ from Lemma \ref{quasi.axis}.
We arrange the direction (to which the parameter increases) of $\gamma, \sigma$ as follows. If $Fix(g) \cap Fix(h)$ is non-empty
(then it consists of one point by our assumption), then we arrange 
$\gamma(-\infty) =\sigma(-\infty)$. Notice that this case 
is equivalent to the projection $\pi_\gamma(\sigma)$ being unbounded.
Now, if $\pi_\gamma(\sigma)$  is bounded but its diameter is longer than $500 \delta$, then it means that near $\pi_\gamma(\sigma)$, 
$\sigma$ runs parallel to $\gamma$ in a bounded ($\le 30 \delta$) neighborhood of $\sigma$ (use Lemma \ref{quasi.axis} (2)), so that we arrange the directions of $\gamma, \sigma$ 
coincide  in this part.
Otherwise, we put directions randomly. 
In any case, $\gamma(\infty) \not= \sigma(\infty)$.

Since $g,h$ are hyperbolic,
there are $L\not= 0$ with $g(\gamma(t))=\gamma(t + L)$ for all $t$, 
and $K\not= 0$ with 
$h(\sigma(t))=\sigma(t+K)$ for all $t$.
Changing $g$ or $h$ or both to their inverses if necessary, we may assume that $L>0, K>0$. We will prove the proposition for the pair $g,h$.

We choose two base points $P \in \gamma, Q\in \sigma$ in the following way.
\\
{\it Case (a)}. If the distance between $\gamma,\sigma$ is more than $100 \delta$, 
then choose a shortest geodesic $\tau$ between them, and 
set $P=\tau \cap \gamma, Q= \tau \cap \sigma$.
\\
{\it Case (b)}.
If the distance between $\gamma, \sigma$ is at most $100 \delta$, then 
let $P$ be the last point (in terms of the parameter $t$)  in $\gamma$ that is contained
in $(\sigma)_{100 \delta}$, which denotes the $100 \delta$-neighborhood of $\sigma$. 
Likewise, let $Q$ be the last point in $\sigma$ that is contained
in $(\gamma)_{100 \delta}$.
Note that $|P-\pi_\sigma(P)|=|Q-\pi_\gamma(Q)| = 100 \delta$.
In either case, by adding constants to parameters, 
we may assume $P=\gamma(0), Q=\sigma(0)$.

Now we define two subsets in $X$.
Set 
$$A=\{x \in X | \pi_\gamma(x) \in \gamma([\frac{1}{2}\min(L(g),L(h)), \infty))\}$$
$$B=\{x \in X | \pi_\sigma(x) \in \sigma([\frac{1}{2}\min(L(g),L(h)), \infty)) \}$$

We want to show the following.  Then we are done by Lemma \ref{pingpong.semigroup}.
\begin{equation}\label{pingpong.sets}
A \cap B = \emptyset, g(A \cup B) \subset A,
h(A\cup B) \subset B.
\end{equation}

 For two points $x,y \in X$, we write $x\sim y$ if $|x-y| \le 1000 \delta$.
 We first show:
 \\
 (I) $\pi_\sigma(P) \sim Q$. $\pi_\gamma(Q) \sim P$. 
 \\
 (II)  
 If $x\in B$ then $\pi_\gamma(x) \sim P$.
 If $x \in A$ then $\pi_\sigma(x) \sim Q$.
%


 We use the following lemma on 
 the nearest points projection in a $\delta$-hyperbolic space to prove (I) and (II).
 If $\alpha$ is a geodesic, then the lemma is well-known.
 We do not know a reference for our setting, so we give an argument
 at the end of the section.
 We denote the Hausdorff distance between two sets $Y,Z$ in $X$
 by $d_H(Y,Z)$. A geodesic between two points $x,y \in X$
 is denoted by $[x,y]$.
 
 \begin{lemma}\label{lemma.proj}
 (1)
 Let $\alpha$ be an axis from Lemma \ref{quasi.axis},  $x \in X$, and $w\in \alpha$. Let $z=\pi_\alpha(x)$. Then $d_H([x,z]\cup[z,w], [x,w]) \le 16 \delta$.
 Moreover, suppose $y \in X$ with  $w=\pi_\alpha(y)$.
 If $|z-w| \ge 40 \delta$, then 
 $d_H([x,z]\cup [z,w]\cup[w,y], [x,y]) \le 18 \delta$.
 \\
 (2)
 If $|x-y| \le D$, then 
 $|\pi_\alpha(x) - \pi_\alpha(y)| \le D + 36 \delta$.
 \end{lemma}

%

 We start the proof of (I) and (II).
 Note that if $x \in A$, then 
 $|\pi_\gamma(x) - P| \ge \frac{9}{10} 5000 \delta - 24 \delta = 4476 \delta$
 since $L(g), L(h) \ge 10000 \delta$ and $\gamma$ is an axis of $g$.
 Similarly, if $ x \in B$, then $|\pi_\sigma(x) - Q| \ge 4476 \delta$.
 
 First, we discuss the case (a) in terms of $P, Q$,
 ie, $|P-Q| > 100 \delta$.
 Then (I) is trivial since (we can arrange that) $\pi_\sigma(P) = Q$ and $\pi_\gamma(Q) = P$.
 
 For (II), we show that 
 if $x\in B$ then $\pi_\gamma(x) \sim P$. 
  Set $y=\pi_\sigma(x)$ and $z=\pi_\gamma(x)$.
 We want to show $|z-P| \le 1000 \delta$.
 Since $Q=\pi_\sigma(P), y=\pi_\sigma(x)$ and $|Q-y| \ge 4476  \delta$,
 we have $d_H( [x,y] \cup [y,Q]\cup [Q,P], [x,P]) \le 18 \delta$
 by Lemma \ref{lemma.proj}.
 Also  
 $d_H([x,z]\cup [z,P], [x,P])\le 16
  \delta$ since $z=\pi_\gamma(x)$.
  So, $d_H([x,y] \cup [y,Q]\cup [Q,P], [x,z]\cup [z,P]) \le 34 \delta$.
To argue by contradiction, assume  $|P-z|> 1000 \delta$.
Let $R \in [P,Q]$ be the point with $|P-R| = 100 \delta$.
Then there is $R' \in [P,z]$ with $|R-R'| \le 34 \delta$.
Then there is $R'' \in \gamma$ with $|R'-R''| \le 12 \delta$, 
so that $|R-R''| \le 46 \delta < |R-P|$. 
But then
$|R''- Q | \le |Q-R| + |R- R''| < |Q-R| + |R-P|=|Q-P|$.
This is a contradiction since $P$ must be the nearest point
on $\gamma$ from $Q$ (but we found $R''$ is closer).
We are done for the first claim of (II). The second one is similar
(just switch the roles), and we do not repeat.
 Case (a) is finished.

 Suppose we are in case (b). 
 For two points $x,y \in \gamma$, we write $x\le y$ if $x=y$, 
 or 
 $y$ appears after $x$ in terms of the parameter $t$.
 We write $x<y$ if $x \le y$ and $x\not= y$.
 We use the similar notation for points on $\sigma$.

 We argue (I). By the way we chose $P,Q$, we have $\pi_\sigma(P) \le Q$
 and $\pi_\gamma(Q) \le P$.
 By Lemma \ref{lemma.proj}, $d_H( [Q, \pi_\sigma(P)] \cup [\pi_\sigma(P), P], [Q,P]) \le 16 \delta$ and 
  $d_H([Q, \pi_\gamma(Q)] \cup [\pi_\gamma(Q), P], [Q,P]) \le 16 \delta.$
  So, $d_H([Q, \pi_\sigma(P)] \cup [\pi_\sigma(P), P], [Q, \pi_\gamma(Q)] \cup [\pi_\gamma(Q), P]) \le 32 \delta$.

 To argue by contradiction, suppose $|\pi_\sigma(P) -Q| > 1000 \delta$.
 Then since $|Q-\pi_\gamma(Q)|=100\delta$
 and $|P-\pi_\sigma(P)| = 100 \delta$, 
 we have $|\pi_\gamma(Q) - P | > 800 \delta$.
 But this forces that the orientations of $\gamma$ and $\sigma$
 are opposite along the geodesic $[\pi_\gamma(Q), P]$
 (draw a thin rectangle for $\pi_\gamma(Q), Q, \pi_\sigma(P), P$), 
 which is a contradiction. We showed $|\pi_\sigma(P) -Q| \le 1000 \delta$. The other inequality is proved similarly. 
  (I) is shown. 
 
 We argue  (II). As before we only discuss the first one.
 Set $y=\pi_\sigma(x)$ and $z=\pi_\gamma(x)$.
 We want to show $|P-z| \le 1000 \delta$. 
 Since $ \pi_\sigma(P) \le Q$, 
 we have $|y-\pi_\sigma(P)| \ge 4476 \delta$.
 So, $d_H([P, \pi_\sigma(P)]\cup [\pi_\sigma(P), y]\cup [y,x], [P,x]) \le 18 \delta.$
 Also, $d_H([P,z]\cup [z,x], [P,x]) \le 16 \delta$.
 Hence, $d_H([P,z]\cup [z,x], [P, \pi_\sigma(P)]\cup [\pi_\sigma(P), y]\cup [y,x]) \le 34 \delta$.
 
 Now, to argue by contradiction, assume $|P-z| > 1000 \delta$.
 Let $R \in [P,z]$ be the point with $|P-R| = 500 \delta$.
 Pick a point $R' \in \gamma$ with $P<R'<z$
 and $|R-R'| \le 12 \delta$ by Lemma \ref{quasi.axis}.
 Also, pick a point $R'' \in [P, \pi_\sigma(P)]\cup [\pi_\sigma(P), y]\cup [y,x]$
 with $|R-R''| \le 34 \delta$.
 Notice that in fact $R'' \in [\pi_\sigma(P), y]$ since 
 $|P-\pi_\sigma(P)| = 100 \delta$.
 Then pick a point $R''' \in \sigma$
 with $|R''-R'''| \le 12 \delta$.
 But this implies $|R'''- R'| \le (12+34+12)\delta=58 \delta$, 
 so that $R' \in (\sigma)_{100 \delta}$.
 This is a contradiction since $P < R'$ on $\gamma$
 ($P$ must be the last such point).
 (II) is shown. 
 We finished the proof of (I) and (II).
 
 We now show (\ref{pingpong.sets}) using (I) and (II).
Clearly both $A$ and $B$ are not empty. 
 To see $A \cap B $ is empty, take any point $x \in B$. Then by  (II),
 $\pi_\gamma(x) \sim P$, ie,  $|\pi_\gamma(x) - P| \le 1000 \delta$.
 But if $x \in A$ then as we said $|\pi_\gamma(x) -P| \ge 4476 \delta$, so 
 $x \not\in A$, hence  we are done.
$g(A) \subset A$ is immediate from the $g$-equivariance of $\pi_\gamma$.
 To see $g(B) \subset A$, let $x \in B$.
 Since $\pi_\gamma(x) \sim P$ we have 
 $g(\pi_\gamma(x))\sim g(P)$.
 Also $g(P) \in \gamma \cap A$ implies $|P-g(P)| \ge \frac{9}{10} 10000 \delta
 -24 \delta = 8976 \delta $,
 hence $|P-g(\pi_\gamma(x))| \ge 7976 \delta$. This implies 
 $g(\pi_\gamma(x)) \in \gamma \cap A$. 
 But since $g(\pi_\gamma(x))=\pi_\gamma(g(x))$,
 it follows 
 $g(x) \in A$ from the definition of $A$. We showed  $g(B) \subset A$, 
 hence $g(A \cup B) \subset A$.
 $h(A\cup B) \subset B$ is similar  (switch the roles) and 
 we do not repeat. 
(\ref{pingpong.sets}) is proved.
\qed
 
 We give a proof of Lemma \ref{quasi.axis}.
 A similar statement appears as Lemma 1 in \cite{fujiwara},
 in which we only assume $g$ is hyperbolic, with no large
 lower bound of $L(g)$, and find a $g$-invariant piecewise 
 geodesic that satisfies (2) with a constant bigger than $12 \delta$.
 One can not expect (1) holds in that case.

 \proof
If $\delta=0$, then $X$ is a tree and there is a $g$-invariant geodesic, $\gamma$. This is a desired path, 
so we assume $\delta >0$ in the rest. 

If we take any point $x \in X$ and join its $g$-orbit 
in the obvious order by geodesics, we obtain a $g$-invariant,
piecewise geodesic. This path is always a quasi-geodesic but 
we do not have uniform bounds on the quasi-geodesic constants. 
But we choose $m$ to be the midpoint of $[x,gx]$, then 
define a $g$-invariant, piecewise geodesic:
$$\gamma = \cup_{n \in \Bbb Z} g^n ([m,gm]).$$
The merit of $m$ is that $[m,gm]$ is very close to $L(g)$, 
and the ``bumps'' at points $g^n(m)$ on $\gamma$
are small. 
We first show $\gamma$ satisfies (2) using $L(g) \ge 1000 \delta$.
Take any two points $x,y \in \gamma$. Using the $g$-action if necessary, 
we may assume  $x \in [m,g(m)], y \in [g^{n-1}(m),g^n(m)]$ for some $n>1$. (If $n=1$, then the claim (2) is trivial.)
The first observation is that 
$g^i(m) \in N_{2 \delta} ([m,g^n(m)])$ for all $0 \le i \le n$.
This is an easy exercise by induction on $n$,  and we leave it to readers (see
Remark).
From this it follows 
$x,y \in N_{3 \delta} ([m,g^n(m)])$.

Next, we show:
\\
{\it Claim}.
$g(m), \cdots, g^{n-1}(m)$ are all in $N_{10 \delta}([x,y])$.
\\
To see that, take points on $[m,g^n(m)]$ that are close to those points:
$$\overline x, \overline{g(m)}, \overline{g^2(m)}, \cdots,
\overline{g^{n-1}(m)}, \overline y \in [m, g^n(m)]$$
with $|x-\overline x| \le 3 \delta, |g(m)- \overline{g(m)}| \le 2 \delta, \cdots, |g^{n-1}(m) - \overline{g^{n-1}(m)} |
\le 2 \delta, |y- \overline y| \le 3 \delta$.
Take $[\overline x, \overline y]$ to be the subpath of $[m,g^n(m)]$.
Notice that the points
$\overline{g^2(m)}, \cdots,
\overline{g^{n-2}(m)}$ are contained in $[\overline x, \overline y]$
since $L(g) \ge 1000 \delta$, but possibly 
$\overline{g(m)}$ or $\overline{g^{n-1}(m)}$
is not in $[\overline x, \overline y]$, ie, 
$\overline{g(m)} \in [m, \overline x]$ or $\overline{g^{n-1}(m)}
\in [\overline y, g^n(m)]$.
But this exceptional case happens
only when $g(m)$ is close to $x$, or $g^{n-1}(m)$
is close to $y$. We finish this case first.
So suppose $\overline{g(m)} \in [m, \overline x]$. Then
pick a point $z \in [m,x]$ with $|\overline {g(m)} - z| \le 4 \delta$
using $|x-\overline x| \le 3 \delta$.
It implies that $|g(m) - z| \le 2 \delta + 4 \delta = 6 \delta$, 
hence $|x-g(m)| \le 6 \delta$. 
This is a desired bound and the claim is shown for $g(m)$.
Similarly, we have $|y-g^{n-1}(m)| \le 6 \delta$.

Now we go back to the general case, ie, 
$\overline{g^i(m)} \in [\overline x, \overline y]$, 
so that $g^i(m) \in N_{2\delta}([\overline x, \overline y])$.
But,  
$d_H([x,y], [\overline x , \overline y]) \le 8 \delta$
since $|x-\overline x|, |y-\overline y| \le 3 \delta$.
It then implies
that $g^i(m)$ is  in $N_{10 \delta}([x,y])$.
The claim is shown.

Finally, let $\gamma(x,y)$ denote the subpath 
of $\gamma$ between $x,y$ then the above claim 
implies that $d_H([x,y],\gamma(x,y)) \le 12 \delta$
since $\gamma(x,y)$ is a piecewise geodesic joining the points
in the claim. 
(2) is shown. 

We prove (1).  Remember that $g^i(m), 1 \le i \le n-1$, are the points that are contained
in $\gamma(x,y)$, the subpath of $\gamma$ between $x,y$.
Then by (2), pick  points $x_i \in [x,y], 1 \le i \le n-1$,  with $|x_i-g^i(m)| \le 12 \delta$.
Let $x=\gamma(t), y=\gamma(s)$.
Then, 
\begin{align*}
 |x-y| &= |x-x_1| + |x_1 - x_2| + \cdots + |x_{n-1}-y| 
 \\
&\ge (|x-g(m)| - 12 \delta) +(|g(m)-g^{2}(m)| - 24 \delta) + \cdots
\\
&+ (|g^{n-2}(m) - g^{n-1}(m)| - 24 \delta) + (|g^{n-1}(m) -y| -12 \delta)
\\
&\ge \frac{1000 \delta - 24 \delta}{1000 \delta}
( |x-g(m)| +|g(m)-g^{2}(m)|  + \cdots |g^{n-1}(m) -y| )- 24 \delta
\\
&\ge \frac{1000-24}{1000}|t-s| - 24\delta
\ge \frac{9}{10}|t-s| -24\delta.
\end{align*}
Note that we need $-24 \delta$ since
we  do not have $|x-g(m)|, |g^{n-1}(m)-y| \ge 1000 \delta$.
\qed

\begin{remark}
This claim appears in the beginning of the proof of Lemma 1 in \cite{fujiwara}, 
where we do not assume $L(g) \ge 1000 \delta$, so that 
we take a power of $g$, $g^N$, with $L(g^N) \ge 1000 \delta$, 
then prove the same claim for $g^N$.

\end{remark}

We give a proof of Lemma \ref{lemma.proj}.
\proof
 If $\delta=0$, then $X$ is a tree and the lemma is trivial, so 
 we assume $\delta>0$.
 \\
 (1) We prove the first claim. If $|x-z| \le 14 \delta$, then the conclusion is 
 trivial, so assume
 $|x-z| > 14 \delta$. Let $z' \in [x,z]$ with $|z-z'|= 14 \delta$.
 $z' \in N_\delta([x,w] \cup [z,w])$, but 
 in fact $z' \in N_\delta([x,w])$.
 This is because if $z' \in N_\delta([z,w])$, 
 then $z' \in N_{13 \delta} (\alpha)$ since $\alpha$ is an axis, which 
 is impossible since $d(z',\alpha)=d(z',z)=14 \delta$.
 So, pick $z'' \in [x,w]$  such that $|z''-z'| \le \delta$, 
 then $|z-z''| \le 15 \delta$.
 It implies that $d_H([z,w], [x,z]\cup [z,w]) \le 16 \delta$,
 and we are done. 
 An argument for the second claim is similar.
 Drawing a geodesic quadrilateral for $x,z,w,y$ and using $|z-w|$ is long, we find 
 $z'', w'' \in [x,y]$ with $|z-z''|, |w-w''| \le 16 \delta$.
 Then the conclusion easily follows. We omit details. 
 (1) is finished. 
 
 (2) is a consequence of (1). Indeed, 
 suppose $|\pi_\alpha(x) - \pi_\alpha(y)| > D + 36 \delta$. By (1), there must be points $Z, W \in [x,y]$
 with $|Z-\pi_\alpha(x)|, |W-\pi_\alpha(y)|
 \le 18 \delta$. That implies that 
 $|Z-W| > D$, so that $|x-y| >D$, a contradiction.
 \qed

\section{Uniform Tits alternative for groups acting on trees}

There is no uniform Tits alternative for groups acting on trees. This section is devoted to an example showing this. To begin with, we recall the following:

\begin{proposition}\label{semi-tree} If $S$ is a finite set of isometries of a tree  with no global fixed point on the tree nor on its boundary, then one may find a pair $a,b$ in $ (S \cup S^{- 1})^{3}$ generating a free sub-semigroup. 
\end{proposition}

\begin{proof} This assumption implies that $L(S)>0$. Proposition \ref{formula.S} implies that there is $g \in S \cup S^2$ with $L(g)>0$. Since $g$ is a hyperbolic isometry it fixes exactly two points on the boundary of the tree. Say $x$ is the forward fixed point. Our assumption implies that there is $s \in S$  with $sx\neq x$. Now we may apply Proposition \ref{semi.group} to the pair $\{g,sgs^{-1}\}$
(take inverses if necessary) and get the desired conclusion.\end{proof}

So we can find a free sub-semigroup quickly. We now show  that by contrast we may not be able to quickly find a pair generating a free subgroup. This is Proposition \ref{counter-tree} from the Introduction, which we restate here.

\begin{proposition} For every $N \in \N$ one can find a pair $S\:=\{a,b\}$ of automorphisms of the $3$-valent homogeneous tree $T$ such that the subgroup generated by $S$ contains a free subgroup and has no global fixed point on $T$ nor on the boundary $\partial T$, but no pair of elements in the $N$-ball  $B(N):=(\{1\}\cup S \cup S^{-1})^N$ generates a free subgroup. 
\end{proposition}

\begin{proof}Let $a$ and $b$ be two hyperbolic elements with translation length $1$ and whose axes intersect on a finite segment $[p,q]$ of length $L$ on which they translate in the same direction. It is clear that $a^{L+1}$ and $b^{L+1}$ generate a free subgroup. Indeed they play ping-pong on the tree : the attracting and repelling neighborhoods being the four connected components of $T\setminus \{p,q\}$ different from that containing the open segment $(p,q)$.  In particular the subgroup $\Gamma:=\langle a,b\rangle$ generated by $a$ and $b$ does not fix a point in $T$ nor on the boundary $\partial T$. 

Fix a labelling of the edges of $T$ using the alphabet $\{1,2,3\}$. Now choose $a$ and $b$ as above with the further property that $a$ and $b$ preserve the cyclic ordering of the edges at every vertex, except possibly at the end-points $p$ and $q$ of the intersection of the axes and $a$ and $b$. To see that it is possible to do this note first that if we start with one edge of $T$ there is a unique  bi-infinite geodesic through this edge such that any two consecutive edges on the geodesic are labelled by consecutive labels and thus there is a unique tree automorphism $a$ that preserves the cyclic order and acts by unit translation on this bi-infinite geodesic $\Delta_a$. This gives $a$. To find $b$ do the same on a bi-infinite geodesic $\Delta_b$ such that $\Delta_a \cap \Delta_b$ is a segment $[p,q]$ of length $L$. It is possible to obtain the isometry $b$ (in a unique way) so that $b$ preserves the cyclic order at every vertex, except that we need to reverse that order at exactly two points : the end-points $p$ and $q$. 

Now let $c$ and $d$ be any two words of length at most $L/2$ in $a,b$ and their inverses. The commutator $[c,d]$ must fix the point $p$. We claim that $[c,d]$ has finite order (and thus $c,d$ cannot generate a free subgroup). Indeed $[c,d]$ preserves the cyclic ordering at any vertex which is at distance at least $R$ from either $p$ or $q$, where $R$ is the maximum distance from either $p$ or $q$ or the image of either $p$ or $q$ under any word of length at most $2L$ in $a,b$ and their inverses. Since $[c,d]$ fixes $p$ some power $[c,d]^n$ of $[c,d]$ will fix pointwise the entire ball of radius $R$ around $p$. But since the edge ordering outside is also preserved, $[c,d]^n$ must be trivial. This ends the proof.
\end{proof}

\begin{remark} As was pointed out to us by Yves de Cornulier, the above isometries $a,b$ belong to the group of tree automorphisms with prescribed local action and finitely many singularities. These groups, denoted by $G(F,F')$, in \cite{leboudec} have attracted a lot of attention lately. In our example $F'=Sym(3)$ and $F=\Z/3\Z$.
\end{remark}

 \section{Application to uniform exponential growth}\label{application-ueg}
In this section we discuss applications to the exponential growth of groups. 
We recall some definitions. 
 Let $\Gamma$ be a group and  $S$ a finite set in $\Gamma$. Assume that $1 \in S$ and $S=S^{-1}$.
Set 
$$h(S):=\lim_{n \to \infty}\frac{1}{n}\log|S^n|.$$
write $h(S,\Gamma)$ to indicate that the subset $S$ generates $\Gamma$. 
We denote by  $\langle S \rangle$ the subgroup generated by $S$. We recall that the quantities $\ell(S)$ and $L(S)$ were defined in the Introduction.
Let $\Gamma$ be a finitely generated group.
Set
$$h(\Gamma) = \inf_S \{h(S) ; \langle S \rangle = \Gamma\},$$
where $S$ runs over finite generating subsets. If $h(\Gamma)>0$ we say $\Gamma$
has {\it uniform exponential growth}, of {\it growth rate} $h(\Gamma)$.

 \subsection{Trichotomy for actions on hyperbolic spaces}

The following result is a consequence of the Bochi-type inequality for hyperbolic spaces (Theorem \ref{bochi-hyp}) and of the construction of ping-pong pairs from Section \ref{ping-pong}.
 \begin{theorem}[Trichotomy for actions on hyperbolic spaces]\label{hyp.tri}
There is an absolute constant $C>1$ such that the following holds.
 Let $X$ be a $\delta$-hyperbolic geodesic space
 and $S \subset Isom(X)$ a finite set with $1 \in S$.
Assume $S=S^{-1}$.
 Then one of the following holds:
 
 (1) $L(S) \leq C \delta $ on $X$. 

 
 (2) $\langle S \rangle$ leaves a set of two points in the boundary
 of $X$ invariant.
 Moreover, $\langle S \rangle$ contains a hyperbolic 
 isometry $g$, and $Fix(g)$ is the invariant set. 
 
 (3) If $N> C$ is an integer, then $S^N$ contains two elements, which are hyperbolic isometries and are generators of a free semi-group.
In particular, 
$h(S) \ge (\log 2)/N >0$.
 
 \end{theorem}

 \proof

%
%
%
%

Assume $C>K+1$, where  $K$ is the numerical constant from 
Theorem \ref{formula.S.hyp} (and Corollary \ref{growth.S.hyp}). If we are not in case $(1)$, then $L(S) >(K+1)\delta $.
Since by Corollary \ref{growth.S.hyp}, for all $n>0$, 
$n(L(S) -K \delta) \le \lambda_2(S^n)$, 
so that $n\delta < \lambda_2(S^n)$.
It implies that for each $n>0$ there exists $g \in S^{2n}$
with $2n \delta < \ell(g)$. 

Set $N_0$ be the smallest integer with $\Delta \le N_0 $,
where 
$\Delta$ is the constant from Proposition \ref{semi.group}.
Then there is $g \in S^{2N_0}$ with  $\Delta \delta  \le \ell(g)$. 
Since $\ell(g) \le L(g)$, we have $\Delta \delta \le L(g)$
and $g$ is hyperbolic on $X$.

Now, if $S$ preserves $Fix(g)$, then we are in (2), 
otherwise, there must be $s \in S$ such that 
the proposition applies to $g, sgs^{-1} \in S^{2N_0 +2}$
(take inverses if necessary),
and we are in (3) provided  $C>2N_0+2$.
\qed

 \subsection{Uniform exponential growth of hyperbolic groups}
 
We now prove uniform, and uniform uniform exponential growth 
of hyperbolic groups.  

 First we show Theorem \ref{ueg-hyp-intro} from the introduction,
that $\delta$-hyperbolic groups have uniform exponential growth depending only on $\delta$, in the form of a corollary of our trichotomy Theorem \ref{hyp.tri} applied to  hyperbolic groups.
 This slighlty improves on Koubi's result \cite{koubi} and  Champetier-Guirardel \cite{guirardel-champetier}, where a further dependence on the size of the generating set was required.
We then discuss uniform uniform exponential growth.

 \begin{corollary}[Growth of hyperbolic groups]\label{ueg.hyp} 
 There is an absolute constant $C_1>0$ 
 with the following property.
 If $G$ is a non-elementary, hyperbolic group, then for any finite symmetric generating subset $S \subset G$ containing $1$, with the Cayley
 graph $\Gamma(G,S)$ $\delta$-hyperbolic for some $\delta>0$, then $S^{M}$ contains two hyperbolic elements that are generators of a free semi-group, where $M$ is the least integer larger than $C_1\delta$.  In particular, $h(S) \ge (\log 2)/M$.
 \end{corollary}

 
We start with a simple lemma.

 
 
 \begin{lemma}\label{ball2}
 Let $G$ be a group generated by a finite symmetric set $S$ with Cayley graph $\Gamma(G,S)$. Then for the (left) action of $G$ on $\Gamma(G,S)$, 
 we have for all $n>0$, $L(S^n) \ge n$. \end{lemma}
 \proof
Fix $n>0$.   For any $g \in G$, we need to show that there  is $h \in S^n g$ such that $d(g,h) \ge n$. Equivalently this amounts to ask that $h \in S^ng$ but $h \notin g S^{n-1}$. Suppose for contradiction that $S^n g$ is contained in $g S^{n-1}$, then $|S^n| \leq |S^{n-1}|$. But since $S^{n-1} \subset S^n$, this means that $S^n=S^{n-1}$. This implies that $\langle S \rangle$ is finite.
 \qed

\begin{proof}[Proof of Corollary \ref{ueg.hyp}]
Let $C>0$ be the absolute constant
 from Theorem \ref{hyp.tri}.
 (We may assume both $C$ and $\delta$ are 
 integers.)
For each $n>0$, $L(S^n) \ge n$ by Lemma \ref{ball2}. Let $n$ be the least integer larger than $C\delta$.  Theorem \ref{hyp.tri} applied to $S^n$ implies that either $S^{Nn}$ contains generators of a free semigroup, or $S^n$ and hence $\langle S \rangle$  leaves invariant a pair of distinct points on the boundary $\partial X$  of the Cayley graph $X=\Gamma(G,S)$. But $G$ acts with dense orbits on its boundary, so the second case cannot occur. Now set $C_1=NC$.
\end{proof}

We now give a uniform exponential growth result for subgroups of hyperbolic groups. This time we cannot a priori rule out the dependence on the cardinality of the generating set of the ambient group.

\begin{corollary}[Uniform growth of subgroups in  a hyperbolic group]\label{hyp.uueg}
There exists an absolute constant (integer) $C_1>0$
with the following property.
Let $G$ be a group generated by a finite symmetric set $T$ so that the Cayley graph $\Gamma(G,T)$ is $\delta$-hyperbolic for some integer $\delta>0$.  Then for every finite subset $S \subset G$ with $1 \in S$ and $S=S^{-1}$, either $\langle S \rangle$ is finite or virtually cyclic, or $S^{B_1}$ contains two generators of free semi-group, where $B_1$ is the cardinality of the ball $T^{C_1 \delta}$ of radius $C_1 \delta $ in $\Gamma(G,T)$.
 In particular, $h(S) \ge (\log 2)/B_1$. 
\end{corollary}

\begin{proof}[Proof of Corollary \ref{hyp.uueg}]
Let $C>0$ be the absolute constant 
in Theorem \ref{hyp.tri}. We may assume $C$ is
an integer by replacing it with the smallest integer
larger than $C$ if necessary. 

Assume $\langle T \rangle =G$ is infinite, otherwise, 
nothing to show. Set $B=|T^{(C+1)\delta}|.$
Clearly $S \subset Isom(X)$. Assume $\langle S \rangle$ is infinite, otherwise there is  nothing to show. Let  $n>0$ and assume that $L(S^n) \le C\delta$.
This means that there is $g \in G$ such that $gS^ng^{-1} \subset T^{C\delta}$. This forces $n \leq |S^n| < |T^{(C+1)\delta}|=B$. In particular if $n \ge B$, then $L(S^n) >C\delta$. 

Apply Theorem \ref{hyp.tri}  to the set $S^B$. Then, either $S^{(C+1)B}$ contains generators of a free semigroup (set $N=(C+1)$ in (3)), 
or $S^B$ and hence $\langle S \rangle$  leaves invariant a pair of distinct points on the boundary $\partial X$, which form the fix point set of a hyperbolic isometry $g$ in  $\langle S \rangle$. 

In the latter a standard argument implies that $\langle S \rangle$ is virtually cyclic. Indeed if $h \in \langle S \rangle$ fixes both fixed points of $g$, then $hgh^{-1}$ is also hyperbolic and fixes the same points. The Morse lemma implies that the axes of $g$ and $hgh^{-1}$ are uniformly close to each other, and in particular $hgh^{-1}g^{-1}$ lies in a ball of bounded radius (independent of $h$) in $X$ around the identity. This shows that the centralizer of $g$ in $\langle S \rangle$ has finite index in $\langle S \rangle$. But the  centralizer of an element of infinite order in a hyperbolic group is virtually cyclic \cite[ch. 8]{ghys}.

We go back to the former case, 
and argue $S^{(C+1)B}\subset S^{B_1}$ 
for some $B_1$ as desired. 
Let $k$ be the least integer with $2^k >C+1$.
Set $C_1=4^k(C+1)$.
Let $B_1=|T^{C_1 \delta}|$, 
which is equal to $|T^{4^k(C+1) \delta}|
> 2^k|T^{(C+1)\delta}| =2^kB
>(C+1)B$, where the first $>$
follows from a well-known ``doubling'' result
(\cite{freiman}) we state 
below. 
It implies that $B_1>(C+1)B $, so that 
from the previous discussion $S^{B_1}$
contains generators of a free semi-group,
hence
$C_1$ is a desired constant. 

\begin{theorem}[Freiman]
Let $A$ be a finite, symmetric set in a group with $1 \in A$ such that $AA$ is not a 
finite subgroup.
Then $|AA| \ge (3/2) |A|$.
\end{theorem}
By our assumption, $\langle T \rangle$
is not a finite subgroup, so that 
$|T^4| > 2|T|$. Applying this repeatedly, we obtain
$|T^{4^k(C+1) \delta}|
> 2^k|T^{(C+1)\delta}|$.
\end{proof}

\begin{remark}\label{other-ref} In \cite{koubi} Koubi obtains the same result in the case when $S$  is a generating subset of $G$. In \cite{guirardel-champetier} Champetier and  Guirardel show a related result : there is an explicit $n$ depending only on $\delta$ and $|T|$ such that given any $f,g \in G$ either at least one of the pairs $\{f^n,g^n\}$ or $\{f^n,g^{-n}\}$ generates a free semigroup, or $f^n$ and $g^n$ commute. 

In a very recent work Delzant and Steenbock \cite{delzant.2018} give another proof of the entropy lower bound obtained in Corollary \ref{hyp.uueg}, which gets better as
$|S|$ gets larger, a nice feature our result does not have. 
\end{remark}

\begin{remark}[sharpness of Corollary \ref{hyp.uueg}] \label{olshanski-exples} The constant $B_1$ in Corollary \ref{hyp.uueg} -- i.e. the smallest radius such that $S^{B_1}$ contains generators of a free semi-group for any $S\subset G$ generating a non-elementary subgroup -- must depend on the size of $T$ and not just on $\delta$. We cannot have a bound depending only on $\delta$ as in Corollary \ref{ueg.hyp}. To see this consider for an arbitrarily large radius $R$ and an arbitrarily  large prime $p$ one of Olshanski's examples \cite{olshanski} of a Gromov hyperbolic group $G$ generated by a set $T$ consisting of $2$ generators, their inverses and the identity, whose ball of radius $R$ is made of elements of $p$-torsion. Such a group is $\delta$-hyperbolic for some large (say integer) $\delta$ depending on $p,R$. However if we consider the new Cayley graph $\Gamma(G,T^\delta)$, we obtain a  $2$-hyperbolic group with $T^R$ containing no element of infinite order, let alone generators of a free semigroup. 
\end{remark}

\subsection{Actions by virtually nilpotent groups}

We give a proof of the following that is 
stated as Corollary \ref{intro.ell=0}.

\begin{proposition}\label{ell=0}
Let $X$ be a geodesic $\delta$-hyperbolic space and $S \subset Isom(X)$ a finite set.
Assume that $\ell(S)=0$.
Then  $\langle S \rangle$ either has a bounded orbit on $X$, or 
fixes a unique point in $\partial X$.
\end{proposition}

\proof
For every $n>0$ we have $\ell(S^n)=0$ since $n\ell(S)=\ell(S^n)$.
So, $\lambda_2(S^n)=0$ by Lemma \ref{gen-ineq}.
Apply Theorem \ref{bochi-hyp} to $S^n$, we obtain 
$K \delta \ge L(S^n)$.

So,  for each $n>0$, there is a point $x_n \in X$ such that 
$L(S^n,x_n) \le K \delta $.
We will show that $\Gamma=\langle S \rangle $ has a bounded orbit,
or fixes a point in $\partial X$.

We recall an elementary  fact.
Let $y_n$ be an infinite sequence of points in the $\delta$-hyperbolic space $X$.
Then there is a subsequence $y_{m_n}$ such that either 
$y_{m_n}$ converges to a point in $\partial X$, or 
it (coarsely) "rotates" about some point $x \in X$ in the sense that 
any geodesic $[y_{m_n}, y_{m_k}]$, $m_n \not= m_k$, intersects the ball  $B(x, 20 \delta)$.

By this fact, there are two cases for the sequence $x_n$:
there is an infinite  subsequence $x_{m_n}$ such that either 
\\
(i)  it converges to a point $x \in \partial X$, or 
\\
(ii) it  rotates, 
ie, there exist a point $x\in X$ 
such that for any $m_n<m_k$, the distance between $x$
and the geodesic $[x_{m_n},x_{m_k}]$ is at most $20 \delta$.

Then, (i) implies that the point $x$ is fixed by $\Gamma$. We discuss the uniqueness later.
In the case (ii),  since both $x_{m_n}, x_{m_k}$, $m_n < m_k$ are moved by $S^{m_n}$ 
at most by $K \delta $, each point on $[x_{m_n},x_{m_k}]$
is moved by $S^{m_n}$ at most by, say, $K \delta  + 10 \delta$.
It implies that $S^{m_n}$ moves $x$ by at most $40\delta + (K+10) \delta$.
This is for any $n>0$, so the $\Gamma$-orbit of $x$ is
bounded.

Now we argue that if there are at least two fixed points in $\partial X$, then 
$\Gamma$ has a bounded orbit in $X$.
Indeed, if there are three fixed points, then a $\Gamma$-orbit 
must be bounded since $X$ is hyperbolic.
If there are exactly two fixed points in $\partial X$, then 
join those two points by a quasi-geodesic $\gamma$ in $X$.
Then $\gamma$ is coarsely invariant by $\Gamma$, ie, there is a constant 
$C$ such that for any $g\in \Gamma$, the Hausdorff-distance between 
$\gamma$ and $g(\gamma)$ is at most $C$.
Notice that $\Gamma$ does not contain any hyperbolic isometry
since $\ell(S)=0$. It implies that there is a constant $D$, which depends
on $C$ and $\delta$, such that 
every point on $\gamma$ is moved by at most $D$ by any element in $\Gamma$, 
ie, $\Gamma$-orbit is bounded. 
\qed

 \begin{corollary}[Virtually nilpotent groups]\label{nilpotent}
 Let $X$ be a $\delta$-hyperbolic space,
 $S \subset Isom(X)$ a finite set with $1 \in S$
and $S=S^{-1}$. Assume $\Gamma=\langle S \rangle$
 is virtually nilpotent.
 Then one of the following holds:
\begin{enumerate}[label=(\roman*)]
\item $\Gamma$ has a bounded orbit in $X$,
\item $\Gamma$ fixes a unique point $x \in \partial X$, 
\item  $\Gamma$  leaves  invariant a set of two points in $\partial X$, which is the fixed point set of some hyperbolic
 isometry $g \in \Gamma$.
 \end{enumerate}
 \end{corollary}
 
 \proof
 We apply Theorem \ref{hyp.tri} to $S$.
 The case (3) does not happen since $h(S)=0$ since $\Gamma$
 has polynomial growth.
 The case (2) implies the conclusion (iii).
 So assume $S$ does not satisfy  (2).
 Then $S$ satisfies (1), ie, $L(S) < (K+1) \delta $.
 
 We also apply Theorem \ref{hyp.tri} to $S^n$ for each $n>1$.
 Again, (3) does not happen, and (2) is desirable, 
 so we assume $L(S^n) < (K+1) \delta $.
 But this implies $\ell(S)=0$. Apply Proposition \ref{ell=0}, 
and we are done. 
\qed

\subsection{Spaces of bounded packing}

 For a metric space $X$, a subgroup $\Gamma$ in $Isom(X)$ is said to be {\it discrete}
 if for any point $x \in X$ and a bounded subset $Y \subset X$, 
 $\Gamma (x) \cap Y$ is finite. 
 

To state our main result, we recall one definition from \cite{BGT}.
A metric space $X$ has {\it bounded packing} with packing constant 
$P>0$ if every ball of radius 2 in $X$ can be covered
by at most $P$ balls of radius 1.

Here is an elementary  lemma we use later. 
\begin{lemma}\label{balls}
If a geodesic space has bounded packing for $P$, then 
any ball of radius $n>0$, which is an integer, is covered by at most $P^{n-1}$
balls of radius 1.
\end{lemma}
\proof
We argue by induction on $n$. For $n=1$, the claim is trivial.
Suppose the claim holds for $n>0$.
Take a ball $B$ of radius $n+1$, and let $B'$ be the ball
of radius $n$ with the same center. By assumption,
cover $B'$ by at most $P^n$ balls of radius 1.
Now for each of those balls, take the ball of radius 2
with the same center. Those  balls of radius 2 cover $B$
(here we are using that the space is geodesic).
Also, each of balls of radius 2 is covered by at most 
$P$ balls of radius 1. So, by collecting all of those
balls of radius 1, $B$ is covered by 
at most $P \cdot P^n$ balls of radius 1.
\qed

In \cite{BGT} balls of radius 4 instead of 2 were used to define the bounded packing property. This is a minor change, which only affects the constant $P$ according to this lemma. 

In the following theorem we state the bounded packing 
property in terms of balls of radius $\delta$ and $2\delta$,
which is more natural for a $\delta$-hyperbolic space.

\begin{theorem}\label{ueg.new}
Given $P$, there is $N(P)$ with the following property. 
Let $X$ be a geodesic $\delta$-hyperbolic space, 
with $\delta>0$,
such that every ball of radius $2\delta$ is covered
by at most $P$ balls of radius $\delta$.

 Let $S$ be a finite subset in $Isom(X)$ with $S=S^{-1}$ and 
 assume that  $\Gamma=\langle S \rangle$ is a  discrete subgroup of $Isom(X)$. Then
either $\Gamma$ is virtually nilpotent, or 
$S^N$ contains two generators of a free semi-group, and in particular:
$$h(S)\geq \frac{1}{N} \log 2.$$

Moreover, if $\Gamma$ is virtually nilpotent,
then either (i) $\Gamma$ is finite, (ii) fixes a unique point 
in $\partial X$, or 
(iii) $\Gamma$ is virtually cyclic and 
contains a hyperbolic isometry $g$ such that 
$Fix(g)$ in $\partial X$ is invariant by $\Gamma$.
\end{theorem}

%



The following theorem by Breuillard-Green-Tao, which improved Gromov's theorem on groups with polynomial growth, 
is a key ingredient of the argument. 

%

\begin{theorem}(\cite[Cor 11.2]{BGT})\label{bgt.nilpotent}
For $Q\ge 1 $ there is a constant $C(Q)$ with the following property.
Let $S$ be a finite generating set of a group $G$ with $1\in S$.
Suppose there exists a finite subset
 $A$  in $G$ such that 
$|A^2| \le Q |A|$ and $S^{C(Q)} \subset A$.
Then $G$ is virtually nilpotent.
\end{theorem}

Here is a useful consequence.
\begin{corollary}\label{bgt.packing}
For  integers $P, J>0$, set $k=C(P^{2J})$, where $C$ is from  Theorem \ref{bgt.nilpotent}.
Suppose $X$ is a geodesic
space and has bounded packing for $P$.
Let $S$ be a finite set in $Isom(X)$ with $S=S^{-1}$  such that 
$\Gamma=\langle S \rangle$ is discrete. 

If $L(S^k) < J$, then $\langle S \rangle$
is virtually nilpotent. 
\end{corollary}

\proof
For a point $x \in X$ define
$$S_J(x)=\{ \gamma \in \Gamma| |x-\gamma x| \le J \}.$$

By assumption, there is $x \in X$ with $L(S^k,x) <J$,
so $S^k \subset S_{J}(x)$.
Set $A=S_{J}(x)$.  $A$ is finite since $\Gamma$
is discrete. We have $A^2 \subset S_{2J}(x)$
by triangle inequality. 

Since $X$ has bounded packing for $P$, 
 $|A^2| \le P^{2J} |A|$.
Indeed, let $B$ be the ball of radius $2J$ centered
at $x$ in $X$. Then $A^2(x) \subset B$.
By Lemma \ref{balls}, $B$ is covered by 
balls of radius 1: $B_1, \cdots, B_k$ with $k \le P^{2J}$.
Now choose $a_i \in A^2$, if it exists, with 
$a_i(x) \in B_i$ for each $i$.
Now for any $a \in A^2$, since $a(x) \in B$,
there is $B_i$ with $a(x) \in B_i$, 
so that $|a_i(x)-a(x)|\le 2$. This means $a_i^{-1}a \in A$, 
so that $a \in a_iA$.
Since $a\in A^2$ was arbitrary, and $|a_iA|=|A|$, 
we find  $|A^2| \le k|A| \le P^{2J} |A|$.

On the other hand, by definition, 
$S^{C(P^{2J})} = S^k \subset S_{J}(x)=A$.
Now by Theorem \ref{bgt.nilpotent}
with $Q=P^{2J}$, $\langle S\rangle $ is virtually nilpotent. 
\qed

\bigskip

\noindent
{\it Proof of Theorem \ref{ueg.new}}.
By scaling the metric of $X$ by a constant, we assume that $X$ is $1$-hyperbolic. By our assumption, 
$X$ has bounded packing property for $P$ with respect to the new metric.
Set $k=C(P^{2(K+1)})$ as in Corollary \ref{bgt.packing} for $J=K+1$. 

There are two cases.
\\
Case 1: $L(S^k) < K+1$. 

In this case $\langle S \rangle$
is virtually nilpotent by Corollary \ref{bgt.packing} applied to $S^k$ with $J=K+1$.
%
%
\\
Case 2: $L(S^k) \ge K+1$. 

Apply Theorem \ref{hyp.tri} to $S^k$ with $\delta=1$.
Set $N_0=N(1)$, where $N(1)$ is the constant from Theorem  \ref{hyp.tri}.
Since we are in Case 2,   (1) does not happen.
If (2) happens then $\langle S \rangle$ is virtually cyclic since $\langle S \rangle$ is discrete. 
If (3) happens then $S^{kN_0}$ contains two elements
that generate a free semi-group.
The constant $k$ depends only on $P$. ($K$ does not depend on anything.)
Set $N=kN_0$ and we are done.

%
%
%
%
%
%

%

To show the moreover part, we apply Corollary \ref{nilpotent}.
If $\Gamma$ has a bounded orbit, then $\Gamma$
must be finite since the action is discrete. 
If there is $g \in \Gamma$ that is hyperbolic such 
that $Fix(g)$ is invariant by $\Gamma$, then 
$\langle g \rangle$ has finite index in $\Gamma$
since $\Gamma$ is discrete, so that $\Gamma$ is 
virtually cyclic.
Otherwise, $\Gamma$ fixes a unique point $x\in \partial X$.
\qed

 Using Theorem \ref{ueg.new}, we give a quick proof of the following theorem 
 by Besson-Courtois-Gallot.
 
\begin{theorem}(\cite[Theorem 1.1]{BCG-jems}\label{bcg.ueg})
Let $X$ be a $d$-dimensional, simply connected
 Riemannian manifold with curvature $-a^2 \leq K \leq -1$. Let $\Gamma=\langle S \rangle$ be a finitely generated discrete subgroup of $Isom(X)$ with $S=S^{-1}$. Then
either $\Gamma$ is virtually nilpotent, or 
 $S^N$ contains two generators of free semigroup, 
 in particular, 
 $h(S) \ge \frac{1}{N} \log 2$,
 where the constant $N$ depends only on $d$ and $a$.
\end{theorem}

In other words, unless $\Gamma$ is virtually nilpotent,
$\Gamma$ has uniform exponential growth, and the growth rate
depends only on $d,a$.
 \\
 
 Before we start the proof, we quote a well-known fact (see the paragraph in \cite{BGT} before Corollary 11.19).

 \begin{lemma}\label{bishop-gromov}
 Let $d\ge 1$ be an integer, and $a \ge 0$. Then there exists $K(d,a)\ge 1$
 with the following property. Suppose $M$ is a $d$-dimensional 
 complete Riemannian manifold with a Ricci curvature lower bound
 $Ric \ge -(d-1)a^2$. Then $M$ has bounded packing for $K(d,a)$.
 \end{lemma}
For readers' convenience we give an outline of an argument.
\\
{\it Outline of proof of Lemma \ref{bishop-gromov}}.
Let $B(x,R)$ denote the ball of radius $R$ in $M$
centered at $x$. By the Bishop-Gromov inequality,
$\sup_{x \in M} \frac{vol(B(x,4)}{vol(B(x,1/2)}$
is bounded from above by a number $K$ that 
depends only on $d,a$.

We will show $B(x,2)$ is covered by at most $K$
balls of radius 1.
 Let $L$ be the maximal number of disjoint balls
of radius $1/2$ in $B(x,2)$. Let $B(y,1/2)$ be the one of maximal volume, $v$, among them. Since $B(x,2)$ is 
contained in $B(y,4)$, we have 
$Lv \le vol(B(x,2)) \le Kv.$
In particular, $L \le K$. 
But the balls of radius 1 centered at the same points cover $B(x,2)$. We are done. 
\qed

\bigskip

\noindent
 {\it Proof of Theorem \ref{bcg.ueg}}.
 Since $X$ is simply connected and $K \le -1$, 
 $X$ is CAT(0) and $\delta$-hyperbolic for, say, $\delta =2$. 
 Since $-a^2 \le K$ and the dimension of $X$ is $d$, 
 we have $Ric \ge -(d-1) a^2$.
 By Lemma \ref{bishop-gromov}, $X$ has bounded packing for the constant 
$K(d,a)$.
 Namely, any ball of radius 2 is covered by at most $K$ balls 
 of radius 1. So, any ball of radius 4 is covered by at most 
 $K^3$ balls of radius 1 by Lemma \ref{balls}.
 
 Set $N=N(K(d,a)^3)$, where $N$ on the right hand side
 is the function   from Theorem \ref{ueg.new}.
 Since the assumption of Theorem 
 \ref{ueg.new} is satisfied by $X$ for  $\delta=2$ and $P=K(d,a)^3$,
 either $\Gamma$ is virtually nilpotent, or 
 $S^N$ contains two generators of free semigroup, 
 in particular, 
 $h(S) \ge \frac{1}{N} \log 2$.
 \qed

 \section{Questions} \label{questions-sec}

\noindent {\bf 1.} Let $X$ be the metric completion of the Teichmuller space
of a surface $\Sigma$ with the Weil-Petersson metric. $X$ is a complete CAT(0) space. Do we have a Bochi-type inequality (see Thereom \ref{bochi-hyp})?
To be concrete, 
let $a,b $ be  the Dehn twists along curves $\alpha, \beta$. Then they are elliptic isometries. 
Assume that  $a,b$ do not commute (ie, the geometric intersection number 
of $\alpha, \beta$ is not 0), 
then  $Fix(a), Fix(b)$ are 
disjoint.
Do we have a Bochi-type inequality  for the set $\{a,b\}$  ?
 We remark that there is a uniform positive lower bound, which depends on $\Sigma$,  on the distance between 
$Fix(a), Fix(b)$.


\bigskip
\noindent {\bf 2.} Let $\Gamma$ be a finitely generated subgroup of $Isom(X)$, where $X$ is a tree. Assume that $\Gamma$ fixes no point on $X$ nor on $\partial X$. Does there exists $N=N(\Gamma) \in \N$ such that, for every symmetric finite generating set $S$ of $\Gamma$,  $(S\cup\{1\})^N$ contains two free generators of a non-abelian free subgroup ?

Note that it is certainly true for discrete subgroups of isometries of a tree, because they are virtually free \cite{bass-tree}. Note further that according to a result of J. Wilson \cite{wilson} the question has a negative answer, if we drop the assumption that $\Gamma$ fixes no point on the tree (or only its boundary for that matter), while still assuming that $\Gamma$ contains some non-abelian free subgroup.

\bigskip
\noindent {\bf 3.} Does there exist an absolute constant $c>0$ such that $h(S)>c$ for every generating set $S$ of an arbitrary non elementary word hyperbolic group (independently of $\delta$, see \cite[Question 2.1]{osin}).

\bigskip
\noindent {\bf 4.} Can we remove the mutiplicative constant in Proposition \ref{compa-compa} ?  that is given a symmetric space $X$ of non-compact type, does there exist $k,C>0$ such that $\lambda_k(S) \geq L(S) - C$ for every finite set $S$ of isometries of $X$. Does this hold also for isometries of Bruhat-Tits buildings or more generally of any Euclidean building ? 

\bigskip
\noindent {\bf 5.} The proof of the geometric Berger-Wang identity  (Theorem \ref{bochi-sym}) and the Bochi-type inequality  (Proposition \ref{compa-compa}) for symmetric spaces of non-compact type relies on the Bochi inequality for matrices, thus eventually on linear algebra (see the other proof given in \cite{breuillard-gelander}). It would be very interesting to find a geometric proof, akin to our proof of Theorem \ref{bochi-hyp} for hyperbolic spaces. Perhaps this could shed light on Question 4. 

\bigskip
\noindent {\bf 6.} Is there a geometric Bochi inequality for isometries of a CAT(0) cube complex ? what about isometries of median spaces ?

\end{document}